\documentclass[reqno]{amsart} 
\usepackage{amssymb}
\usepackage{textcomp}
\usepackage{amsmath, amsthm, amscd, amsfonts}
\usepackage{mathrsfs}
\usepackage{bm, bbm}
\usepackage{upgreek}
\usepackage{mathtools}

\usepackage{stmaryrd}

\usepackage{hyperref}
\hypersetup{colorlinks=true, citecolor=blue}
\usepackage{color}
\usepackage[usenames,dvipsnames,svgnames,table]{xcolor}

\newcommand{\mb}{\mathbb}

\newcommand{\mf}{\mathfrak}
\newcommand{\mc}{\mathcal}

\newcommand{\ov}{\overline}

\newcommand{\wt}{\widetilde}
\newcommand{\lbr}{\llbracket}
\newcommand{\rbr}{\rrbracket}

\makeatletter 
\@addtoreset{equation}{section}
\makeatother  

\numberwithin{equation}{section}

\title[Gauged Hamiltonian Floer homology]{Gauged Hamiltonian Floer homology I: definition of the Floer homology groups}

\author{Guangbo Xu}
\address{
Department of Mathematics \\
410N Rowland Hall\\
University of California, Irvine \\
Irvine, CA 92697 USA
}
\email{guangbox@math.uci.edu}

\date{\today}

\begin{document}

\newtheorem{thm}{Theorem}[section]
\newtheorem{lemma}[thm]{Lemma}
\newtheorem{cor}[thm]{Corollary}
\newtheorem{prop}[thm]{Proposition}

\theoremstyle{definition}
\newtheorem{defn}[thm]{Definition}
\theoremstyle{remark}
\newtheorem{rem}[thm]{Remark}
\newtheorem{hyp}[thm]{Hypothesis}

\begin{abstract}
We construct the vortex Floer homology group $VHF\left( M, \mu; H\right)$ for an aspherical Hamiltonian $G$-manifold $(M, \omega, \mu)$ and a class of $G$-invariant Hamiltonian loops $H_t$, following the proposal of \cite{Cieliebak_Gaio_Salamon_2000}. This is a substitute for the ordinary Hamiltonian Floer homology of the symplectic quotient of $M$. The equation for connecting orbits is a perturbed symplectic vortex equation on the cylinder ${\mb R} \times S^1$. We achieve the transversality of the moduli space by the classical perturbation argument instead of the virtual technique, so the homology can be defined over ${\mb Z}$.
\end{abstract}

\maketitle

\setcounter{tocdepth}{1}
\tableofcontents

\section{Introduction}

\subsection{Background}

Floer homology, introduced by Andreas Floer (see \cite{Floer_intersection}, \cite{Floer_CMP}), has been a great triumph of Gromov's $J$-holomorphic curve technique (\cite{Gromov_1985}) in many areas of mathematics. Hamiltonian Floer homology gives new invariants of symplectic manifolds and has been the most important approach towards the solution to the celebrated Arnold conjecture; the Lagrangian Floer homology and more general the Fukaya category become fundamental objects in Kontsevich's homological mirror symmetry conjecture \cite{HMS}; several Floer-type homology theories, including the instanton Floer homology (\cite{Floer_instanton}, \cite{Donaldson_YM}), Heegaard-Floer theory (\cite{OZ}), Seiberg-Witten Floer homology (\cite{KM_book}), ECH theory (\cite{Hutchings_Sullivan}, \cite{Hutchings_Taubes}), have become tools of understanding lower dimensional topology. Many variants of Floer homology, such as symplectic homology (\cite{FH_94}, \cite{CFH_95}, \cite{CFHW_96}) and symplectic field theory (\cite{SFT}), lead to many interesting theories of open symplectic manifolds.

In this paper we consider a new Floer homology theory and construct the corresponding Floer homology group. The construction was proposed in \cite{Cieliebak_Gaio_Salamon_2000}. Before we introduce this new theory, we give a brief review on general Morse homology theory and Hamiltonian Floer theory.

All types of Floer theories are certain infinite dimensional Morse theories, whose constructions essentially apply Witten's point of view (\cite{Witten_morse}). Basically, if $f: X\to {\mb R}$ is certain smooth functional on manifold $X$ (which could be infinite dimensional), then with an appropriate choice of metric on $X$, we can study the equation of negative gradient flow of $f$, of the form
\begin{align}\label{equation11}
x'(t) + \nabla f(x(t)) = 0,\ t\in (-\infty, +\infty).
\end{align}
If a solution $x(t)$ has finite ``energy'', then $x(t)$ converges to a critical point of $f$ as $t \to \pm\infty$. Assuming that all critical points of $f$ is nondegenerate, then we can define a Morse-type index $\lambda_f: {\rm Crit} f \to {\mb Z}$. Then for a given pair of critical points $a_-, a_+\in {\rm Crit} f$, the {\it moduli space} of solutions to the negative gradient flow equation which are asymptotic to $a_\pm$ as $t\to \pm \infty$, denoted by ${\mc M}(a_-, a_+)$, has dimension equal to $\lambda_f(a_-) - \lambda_f (a_+)$, if $f$ and the metric are perturbed generically. If $\lambda_f(a_-) - \lambda_f (a_+)= 1$, because of the translation invariance of (\ref{equation11}), we expect to have only finitely many geometrically different solutions connecting $a_-$ and $a_+$. In the ``orientable'' cases we can also associate a sign to each such solutions.

On the other hand, we define a chain complex over ${\mb Z}_2$ (and over ${\mb Z}$ in the oriented case), generated by critical points of $f$ and graded by the index $\lambda_f$; the boundary operator $\partial$ is defined by the (signed) counting of geometrically different trajectories of solutions to (\ref{equation11}) connecting two critical points with adjacent indices. We expect a nontrivial fact that $\partial \circ \partial = 0$. So a homology group is derived.

Last but not least, a nontrivial amount of effort is devoted to proving the independence of the homology groups on various choices of metrics, almost complex structures, etc. Therefore the homology is some invariant of the underlying geometric background, which implies certain interesting results. For example, in Hamiltonian Floer theory, the Floer homology group is isomorphic to the singular homology of the symplectic manifold, which implies the Arnold conjecture.

\subsection{Hamiltonian Floer homology and the transversality issue}

In Hamiltonian Floer theory, we have a compact symplectic manifold $(X, \omega)$ and a time-dependent Hamiltonian $H \in C^\infty(S^1 \times X)$, usually denoted as a family $H_t: X \to {\mb R}$ parametrized by $t \in S^1$. We can define an action functional ${\mc A}_H$ on a covering space $\wt{LX}$ of the contractible loop space of $X$. The space $\wt{LX}$ consists of pairs $(x, w)$ where $x: S^1 \to X$ is a contractible loop and $w: {\mb D} \to X$ with $w|_{\partial {\mb D}} = x$; the action functional is defined as
\begin{align}\label{equation12}
{\mc A}_H (x, w) = -\int_{\mb D} w^* \omega - \int_{S^1} H_t(x(t)) dt.
\end{align}
The Hamiltonian Floer homology is formally the Morse homology of $\left( \wt{LX}, {\mc A}_H \right)$. 

The critical points are pairs $(x, w)$ where $x: S^1 \to X$ satisfying $x'(t) = X_{H_t}(x(t))$, where $X_{H_t}$ is the Hamiltonian vector field associated to $H_t$; these loops are 1-periodic orbits of the Hamiltonian isotopy generated by $(H_t)$. Choose a smooth $S^1$-family of $\omega$-compatible almost complex structures $(J_t)$ on $X$, which induces an $L^2$-metric on the loop space of $X$. We can write (\ref{equation11}) formally as the {\it Floer equation} for maps $u$ from the infinite cylinder $\Theta = {\mb R}\times S^1$ to $X$:
\begin{align*}
{\partial u \over \partial s} + J_t \left( {\partial u \over \partial t} - X_{H_t}(u) \right) = 0.
\end{align*}
Here $(s, t)$ is the standard coordinates on $\Theta$. This is a perturbed Cauchy-Riemann equation, so Gromov's theory of pseudoholomorphic curves can be applied.

A nontrivial issue is how to make the moduli spaces transverse. The effort of dealing with the transversality problem has long history. In both Donaldson's theory and Gromov-Witten-Floer theory for semi-positive manifolds, one has a large space of auxiliary data (metrics or almost complex structures). So a generic choice of the auxiliary data can make the moduli space transverse. However, for $J$-holomorphic curves, certain obstructions cannot be overcome in this way, where there are multiple covers of holomorphic spheres with negative Chern numbers. Therefore such method only works for semi-positive symplectic manifolds where there are no such holomorphic spheres. In Floer theory, the transversality argument was carried out by Floer-Hofer-Salamon \cite{Floer_Hofer_Salamon}, which was applied to define Floer homology groups for monotone (\cite{Floer_CMP}) and semi-positive (\cite{Hofer_Salamon}, \cite{Ono_1995}) manifolds. 

On the other hand, for general symplectic manifolds, people developed the so-called ``virtual technique'' to overcome the obstruction. It leads to the definition of the Floer homology and proof of Arnold conjecture for general symplectic manifolds, by Fukaya-Ono (\cite{Fukaya_Ono}) and Liu-Tian (\cite{Liu_Tian_Floer}). The virtual technique, despite of its power, is much more sophisticated than the perturbation argument. There have been plenty of discussions on clarifications of its details. Moreover, since the virtual technique requires to do ``multi-valued'' perturbations, the resulting theories are only defined over rational coefficients rather than integers.

\subsection{Hamiltonian Floer theory in gauged $\sigma$-model}

The Floer theory considered in this paper plays a role as a substitute of the ordinary Floer theory, while the virtual technique can be bypassed in many interesting cases. It was proposed in \cite{Cieliebak_Gaio_Salamon_2000}, motivating from Dostoglou-Salamon's study of Atiyah-Floer conjecture (see \cite{Dostoglou_Salamon}). The main analytical object is the symplectic vortex equation, which was also independently studied initially in \cite{Cieliebak_Gaio_Salamon_2000} and by Ignasi Mundet in  \cite{Mundet_thesis} \cite{Mundet_2003}.

The role of the symplectic vortex equation in gauged Floer theory considered in this paper is the same as the role of $J$-holomorphic curve equation in ordinary Floer theory. In our case, consider the following action functional on a covering space of the space of contractible loops in $M \times {\mf g}$. Let $H: M \times S^1 \to {\mb R}$ be an $S^1$-family of $G$-invariant Hamiltonians; for any contractible loop $\wt{x}:=(x, f): S^1 \to M \times {\mf g}$ with a homotopy class of extensions of $x: S^1 \to M$, represented by $w: {\mb D}\to M$, the action functional (given first in \cite{Cieliebak_Gaio_Salamon_2000}) is 
\begin{align}\label{equation13}
\wt{\mc A}_H(x, f, w):= - \int_{\mb D} w^* \omega + \int_{S^1}\left( \mu(x(t)) \cdot f(t) - H_t(x(t)) \right) dt.
\end{align}
The critical loops of $\wt{\mc A}_H$ corresponds to 1-periodic orbits of the induced Hamiltonian $(\ov{H}_t)$ on the symplectic quotient $\ov{M}:= \mu^{-1}(0)/ G$. The equation of negative gradient flows of $\wt{\mc A}_H$, is just the symplectic vortex equation on the trivial bundle $G \times \Theta$, with the standard area form $ds \wedge dt$, and the connection $A$ is in temporal gauge (i.e., $A$ has no $ds$ component). More precisely, if we choose an $S^1$-family of $G$-invariant, $\omega$-compatible almost complex structures $J_t$, then (\ref{equation11}) is written as a system of $(u, \Psi) : \Theta \to M \times {\mf g}$:
\begin{align}\label{equation14}
\left\{ \begin{array}{ccc}
\displaystyle {\partial u \over \partial s} + J_t \left( {\partial u \over \partial t} + X_{\Psi}(u) - Y_{H_t} \right) & = & 0;\\[0.4cm]
\displaystyle {\partial \Psi \over \partial s} + \mu(u) & = & 0.
\end{array}\right.
\end{align}
Solutions with finite energy are asymptotic to loops in ${\rm Crit} \wt{\mc A}_H$. Then the moduli space of such trajectories, especially those zero-dimensional ones, gives the definition of the boundary operator in the Floer chain complex, and hence the Floer homology group. We call these homology theory the {\bf vortex Floer homology}. 

Here we state the main theorem of this paper, which looks very much parallel to the corresponding statements in ordinary Hamiltonian Floer theory.
\begin{thm}
To a pair $(J, H)$ where $H = (H_t) \in C^\infty_c(S^1 \times M)^G$ is an admissible Hamiltonian (see Definition \ref{defn65}), and $J = (J_t)$ is a generic ``admissible'' almost complex structure $J = (J_t)$ (see Definition \ref{defn62}), we can associate the vortex Floer chain complex 
\begin{align*}
\left( VCF_* \left( M, \mu; J, H; \Lambda_{\mb Z} \right), \delta_J \right).
\end{align*}
This is a chain complex of $\Lambda_{\mb Z}$-modules (where $\Lambda_{\mb Z}$ is defined in Subsection \ref{subsection24}), generated by certain equivalence classes of critical points of the action functional $\wt{\mc A}_H$ given by (\ref{equation13}). The chain complex is ${\mb Z}$-graded, and the grading is given by a Conley-Zehnder type index, defined in Subsection \ref{subsection43}. The homology of this chain complex is denoted by
\begin{align*}
VHF_* \left( M, \mu; J, H; \Lambda_{\mb Z} \right).
\end{align*}

Moreover, if $(J', H')$ is another pair for which we can define the vortex Floer chain complex $\left( VCF_*(M, \mu; J', H'; \Lambda_{\mb Z}), \delta_{J'} \right)$, then there exists a chain homotopy equivalence (the continuation map)
\begin{align*}
\mf{cont}: VCF_*\left( M, \mu; J, H; \Lambda_{\mb Z}\right) \to VCF_* \left( M, \mu; J', H'; \Lambda_{\mb Z} \right).
\end{align*}
The construction of the continuation map results in a unique isomorphism 
\begin{align*}
\mf{cont}: VHF_* \left( M, \mu; J, H; \Lambda_{\mb Z} \right) \to VHF_* \left( M, \mu; J', H'; \Lambda_{\mb Z} \right).
\end{align*}
So there exists a ${\mb Z}$-graded $\Lambda_{\mb Z}$-module $VHF_*\left(M, \mu; \Lambda_{\mb Z} \right)$ which is canonically isomorphic to $VHF_* \left( M, \mu; J, H; \Lambda_{\mb Z} \right)$ for all $(J, H)$ for which we can define the vortex Floer chain complex. This module is called the vortex Floer homology of $(M, \mu)$.
\end{thm}

\subsection{Lagrange multipliers}\label{subsection14}

The action functional (\ref{equation13}) seems to be quite complicated, not to mention its gradient flow equation (\ref{equation14}). However, the action functional (\ref{equation13}) is just a Lagrange multiplier of the action functional (\ref{equation12}). Indeed there is a much simpler situation in the case of the Morse theory of a finite-dimensional Lagrange multiplier, which is worth mentioning in this introduction as a model. 

Suppose $X$ is a Riemannian manifold and $\mu: X \to {\mb R}$ is a smooth function, with $0$ a regular value. Then consider a function $f: X \to {\mb R}$ whose restriction to $\ov{X} = \mu^{-1}(0)$ is Morse. Then critical points of $f|_{\ov{X}}$ are the same as critical points of the Lagrange multiplier $F: X \times {\mb R} \to {\mb R}$ defined by $F(x, \eta) = f(x) + \eta \mu(x)$, and the Morse index as a critical point of $f|_{\ov{X}}$ is one less than the index as a critical point of $F$. Then instead of considering the Morse-Smale-Witten complex of $f|_{\ov{X}}$, we can consider that of $F$. In generic situation, these two chain complexes have the same homology, and a concrete correspondence can be constructed through the ``adiabatic limit'' (for details, see \cite{Lagrange_multiplier}).

The vortex Floer homology studied in this paper is an infinite-dimensional and equivariant generalization of this Lagrange multiplier technique. Therefore, the vortex Floer homology is expected to coincide with the ordinary Hamiltonian Floer homology of the symplectic quotient.

\subsection{Advantage in achieving transversality}

It seems that considering the complicated equation (\ref{equation14}) only gives what we have already known about the Hamiltonian Floer theory of the symplectic quotient. But the trade-off is that the most crucial and sophisticated step--transversality of the moduli space--can be achieved more easily. This advantage comes from the fact that in many cases, $M$ has simpler topology than $\ov{M}$. For example, toric manifolds are symplectic quotients of the Euclidean space, where the latter allows no nonconstant holomorphic sphere. So the issue caused by spheres with negative Chern numbers is ruled out. This phenomenon allows us to achieve transversality of the moduli space by using the traditional ``concrete perturbation'' to the equation, similar to the transversality in Hamiltonian Floer theory for semi-positive manifolds. Moreover, when using virtual technique, the Floer homology group of the symplectic quotient can only be defined over ${\mb Q}$ but here it can be defined over ${\mb Z}$ or ${\mb Z}_2$.

\subsection{Ring structure and computation of the Floer homology}

The construction of the quantum multiplication on $VHF_*(M, \mu)$ and the computation of the group are two indispensable parts of the whole theory. We postpone the details for later work but would like to give a brief description here.

The quantum ring structure on $VHF_*(M, \mu)$ is constructed on the chain level, by counting isolated solutions to the symplectic vortex equation on a pair-of-pants. The pair-of-pants is conformal to the 3-punctured Riemann sphere. We choose an area form on $\Sigma$ of cylindrical type near $z_\sigma$ ($\sigma = 0, 1, \infty$), and we perturb the symplectic vortex equation by three admissible Hamiltonians $\left( H^\sigma_t \right)$. All necessary analytical ingredients in defining the chain level multiplication has been provided in the current paper, so there is no essential difficulties in the construction.

On the other hand, it has been conjectured in \cite{Cieliebak_Gaio_Salamon_2000} that the vortex Floer homology group should be isomorphic to the singular homology of the symplectic quotient $\ov{M}$ over integers. This implies the homological Arnold conjecture for $\ov{M}$ over integers, which means a refinement of the Arnold conjecture proved by \cite{Fukaya_Ono}, \cite{Liu_Tian_Floer} over rationals. Naturally one thinks of imitating similar methods used for ordinary Hamiltonian Floer. However, the computation is not so straightforward, espeically if we want to establish the isomorphism over integers.

The first possible approach is to use a time-independent function as the Hamiltonian, and try to prove that when the function is very small in $C^2$-norm, there is no ``quantum contribution'' when defining the boundary operator in the Floer chain complex; this was also Floer's original argument. However, there are certain obstructions to achieve transversality: the Hamiltonian and almost complex structure we used in the current paper have special properties which we don't expect to see for time-independent ones (see the appendix, especially Definition \ref{defn62} and Definition \ref{defn65}); on the other hand, though $M$ is assumed to have no nonconstant holomorphic spheres, solutions to (\ref{equation14}) for time-independent $(J, H)$ can still form multiple covers of negative Chern numbers. These difficulties almost invalidate this approach, if one want to avoid using virtual technique.

The second possible approach is the Piunikhin-Salamon-Schwarz (PSS) construction (\cite{PSS}). Part of the following ideas emerge during discussions with Weiwei Wu. In this approach one would like to define some moduli space (of ``spiked disks'') interpolating between the vortex Floer chain complex and some Morse chain complex whose homology is isomorphic to the singular homology of $\ov{M}$. Since the vortex equation is not conformally invariant, one has to choose the correct metric on the (punctured) disk. Basically, consider the following equation on $\Theta$.
\begin{align}\label{equation15}
\left\{ \begin{array}{ccc} \displaystyle {\partial u \over \partial s} + X_\Phi + J_t \left( {\partial u \over \partial t} + X_\Psi - \rho(s) Y_{H_t}(u) \right) & = & 0,\\[0.3cm]
            \displaystyle {\partial \Psi \over \partial s} - {\partial \Phi \over \partial t} + [\Phi, \Psi] + \mu(u) & = & 0.
\end{array} \right.
\end{align}
Here $\rho$ is a cut-off function supported either near $-\infty$ or $+\infty$. Suppose it is the latter case. Then near $-\infty$ one can prove that a finite energy solution $(u, \Phi, \Psi)$ is asymptotic to a $G$-orbit $\ov{x} \subset \mu^{-1}(0)$. On the other hand, we choose a Morse-Smale pair $(f, g)$ on the symplectic quotient $\ov{M}$. Therefore a spiked disk in this case is a pair $(\ov{x}, [u, \Phi, \Psi])$, where $\ov{x}: (-\infty, 0] \to \ov{M}$ is a negative gradient flow line of $f$ starting from a critical point, and $(u, \Phi, \Psi)$ is a solution to (\ref{equation15}) such that $u$ is asymptotic to $\ov{x}(0)$.

The counting of isolated spiked disks should define a chain map, but the chain complex associated with $(f, g)$ is not its Morse-Smale-Witten complex, but a ``pearl complex'' similar to that of Biran-Cornea \cite{Biran_Cornea}\cite{Biran_Cornea_09}. The pearls, instead of holomorphic disks in Biran-Cornea's case, are solutions on $\Theta$ to the following equation.
\begin{align}\label{equation16}
\left\{ \begin{array}{ccc} \displaystyle {\partial u \over \partial s} + X_\Phi + J_t \left( {\partial u \over \partial t} + X_\Psi  \right) & = & 0,\\[0.3cm]
            \displaystyle {\partial \Psi \over \partial s} - {\partial \Phi \over \partial t} + [\Phi, \Psi] + \mu(u) & = & 0.
\end{array} \right.
\end{align}
Each solution to (\ref{equation16}) represents a class in $H_2^G(M)$. (One intends to use $t$-independent almost complex structures but then will have trouble with multiple covers.)

The last step is to prove that the pearl complex is chain homotopy equivalent to the Morse-Smale-Witten complex of $\ov{M}$. To do this we deform $J_t$ to a $t$-independent almost complex structure $J_0$ and try to use a cobordism argument. However, the existence of multiple covers for solutions to (\ref{equation16}) for $J_0$ instead of $J_t$ will cause nontrivial obstructions. But this is a special (and possibly the simplest) case of the situation considered in \cite{Fukaya_Ono_integer}, where by analyzing the contribution of multiple covers one can prove that the only essential contribution to the boundary operator of the pearl complex are those Morse flow lines.

Another possible way of proving the isomorphism between the vortex Floer homology and the singular homology of $\ov{M}$ is to use the adiabatic limit technique. This ideas was used by Gaio-Salamon \cite{Gaio_Salamon_2005}, which shows that in certain cases, the Hamiltonian-Gromov-Witten invariants with low degree insertions coincide with the Gromov-Witten invariants of the symplectic quotient. More generally, adiabatic limit leads to a quantum deformation of the Kirwan map (see \cite{quantumkirwan}, \cite{Woodward_Kirwan}).

In our case, for any $\lambda>0$, we could consider a variation of (\ref{equation14})
\begin{align}\label{equation17}
\left\{ \begin{array}{ccc}
\displaystyle {\partial u \over \partial s} + J_t \left( {\partial u \over \partial t} + X_{\Psi}(u) - Y_{H_t} \right) & = & 0;\\[0.3cm]
\displaystyle {\partial \Psi \over \partial s} + \lambda^2 \mu(u) & = & 0.
\end{array}\right.
\end{align}
Then we would like to let $\lambda$ approach to $\infty$. By a simple energy estimate, solutions of (\ref{equation17}) will ``sink'' into the symplectic quotient $\ov{M}$ and become Floer trajectors of the induced pair $(\ov{H}, \ov{J})$; at isolated points there will be energy blow up, and certain ``affine vortices'' will appear, which are finite energy solutions to the symplectic vortex equation over the complex plane ${\mb C}$. If we can analyze the contribution of affine vortices (maybe with similar restriction on $M$ as in \cite{Gaio_Salamon_2005}), we could prove that $VHF(M, \mu; H)$ is isomorphic to $HF(\ov{M}; \ov{H})$, with appropriate changes of coefficients. However, since generally $HF(\ov{M}; \ov{H})$ can be defined only over ${\mb Q}$, this method has its limitations.

As a remark, it is interesting to consider the reversed limit $\lambda \to 0$, and it actually motivated the work of the author with S. Schecter \cite{Lagrange_multiplier}, where they considered the nonequivariant, finite dimensional Morse homology for Lagrange multiplier type functions on $M \times {\mb R}$ (see Subsection \ref{subsection14}). In \cite{Lagrange_multiplier} it was shown that, the Morse-Smale trajectories, as $\lambda \to 0$, will converge to certain ``fast-slow'' trajectories, and the counting of such trajectories defines a new chain complex, which also computes the same homology.

\subsection{Compare with gauged Lagrangian Floer theroy}

In Frauenfelder's thesis \cite{Frauenfelder_thesis} (and in a slightly different published version \cite{Frauenfelder_2004}), he used the symplectic vortex equation on the strip ${\mb R}\times [0, 1]$ to define Lagrangian Floer homology for certain types of pairs of Lagrangians $(L_0, L_1)$ in $M$. The Lagrangians are not $G$-invariant in general, but their intersections with $\mu^{-1}(0)$ reduce to a pair of Lagrangians $(\ov{L}_0, \ov{L}_1)$ in the symplectic quotient $\ov{M}$. Then by the calculation in the Morse-Bott case, he managed to prove the Arnold-Givental conjecture (in monotone case in \cite{Frauenfelder_thesis} and general case in \cite{Frauenfelder_2004}). 

Woodward also defined a version of Lagrangian Floer theory in \cite{Woodward_toric}, where he considered $G$-invariant Lagrangians $L \subset \mu^{-1}(0)$. They project down to Lagrangians downstairs (i.e. in $\ov{M}$). His equation is the naive limit of the symplectic vortex equation with Lagrangian boundary condition by setting the area form to be zero. He applied this Floer theory to study displacibility of toric fibres in toric orbifolds, which reproduces and extends some of the results of Fukaya et. al. \cite{FOOO_toric_I} \cite{FOOO_toric_II}.

The problem considered in this paper shares some similarities with the above two Lagrangian Floer theories. First, they both take the advantage that no nonconstant holomorphic sphere exist upstairs (i.e., in the Hamiltonian $G$-manifold). Second, the current paper and \cite{Frauenfelder_thesis}, \cite{Frauenfelder_2004} both use vortex equation, therefore some analysis are similar. We will give more detailed comments in the remaining sections.

\subsection{Organization and conventions of this paper}

In Section \ref{section2} we give the basic setup, including the action functional, the definition of the Floer chain group and the equation of connecting orbits. In Section \ref{section3} we proved that each finite energy solution is asymptotic to critical loops of the action functional. In Section \ref{section4} we study the Fredholm theory of the equation of connecting orbits; we show that the linearized operator is a Fredholm operator whose index is equal to the difference of Conley-Zehnder indices of the two ends of the connecting orbit. In Section \ref{section5} we prove that our moduli space is compact up to breaking, if assuming the nonexistence of nontrivial holomorphic spheres. In Section \ref{section6} we prove the transversality of the moduli spaces by choosing certain type of Hamiltonian and choosing a generic $t$-dependent almost complex structure. In Section \ref{section7} we give the definition of the vortex Floer homology and prove the invariance of the homology group by using continuation method. 

\subsubsection*{Notations and conventions}

We use $\Theta$ to denote the infinite cylinder ${\mb R}\times S^1$, with the axial coordinate $s$ and angular coordinate $t$. We denote $\Theta_+ = [0, +\infty) \times S^1$ and $\Theta_- = (-\infty, 0]\times S^1$.

$G$ is a connected compact Lie group, with Lie algebra ${\mf g}$. Any $G$-bundle over $\Theta$ is trivial, and we just consider {\it the} trivial bundle $P = G \times \Theta$. Any connection $A$ can be written as a ${\mf g}$-valued 1-form on $\Theta$. We always use $\Phi$ to denote its $ds$ component and $\Psi$ to denote its $dt$ component. 

There is a small $\epsilon>0$ such that for the $\epsilon$-ball ${\mf g}_\epsilon^* \subset {\mf g}^*$ centered at the origin of ${\mf g}^*$, $U_\epsilon := \mu^{-1}({\mf g}_\epsilon^*)$ can be identified with $\mu^{-1}(0) \times {\mf g}_\epsilon^*$. We denote by $\pi_\mu: U_\epsilon \to \mu^{-1}(0)$ the projection on the the first component, and by $\ov\pi_\mu: U_\epsilon \to \ov{M}$ the composition with the projection $\mu^{-1}(0) \to \ov{M}$. 

A solution to the gradient flow equation will be called a flow line; once we know the asymptotics of a flow line, we call it a connecting orbit. A trajectory is an orbit of the natural ${\mb R}$-action on the space of connecting orbits.

\subsection{Acknowledgments} The author would like to thank his PhD advisor Gang Tian for introducing him to this field and for his support and encouragement. He
would like to thank Urs Frauenfelder, Kenji Fukaya, Yong-Geun Oh, Chris Woodward, and Weiwei Wu for many helpful discussions and encouragement.

The author want to thank Chris Woodward for pointing out a mistake in the appendix of the first arXiv version of this paper.

During the preparation of the current version, the author is visiting Institute for Advanced Study. He would like to thank Professor Helmut Hofer for hospitality.

\section{Basic setup and outline of the construction}\label{section2}

Let $(M, \omega)$ be a connected symplectic manifold. We assume that it is aspherical, i.e., for any smooth map $f: S^2 \to M$, $\displaystyle \int_{S^2} f^*\omega = 0$. This implies that for any $\omega$-compatible almost complex structure $J$ on $M$, there is no nonconstant $J$-holomorphic spheres.

Let $G$ be a {\it connected} compact Lie group which acts on $M$ smoothly. The infinitesimal action ${\mf g} \ni \xi \mapsto X_\xi \in \Gamma(TM)$ is an anti-homomorphism of Lie algebras. We assume the action is Hamiltonian, which means that there exists a smooth function $\mu: M\to {\mf g}^*$ satisfying
\begin{align*}
\mu(gx) = \mu(x) \circ {\rm Ad}_g^{-1} ,\ \forall \xi\in {\mf g},\ d \left( \mu \cdot \xi \right) = \iota_{X_\xi} \omega.
\end{align*}

Suppose we have a $G$-invariant, time-dependent Hamiltonian
\begin{align*}
H  = (H_t) \in C^\infty_c \left( S^1 \times M \right)
\end{align*}
with compact support. For each $t \in S^1$, the associated Hamiltonian vector field $Y_{H_t} \in \Gamma(TM)$ is determined by
\begin{align*}
\omega( Y_{H_t}, \cdot ) = d H_t\in \Omega^1(M).
\end{align*}
The flow of $Y_{H_t}$ is a one-parameter family of diffeomorphisms
\begin{align*}
\phi^H_t: M \to M,\ {d \phi^H_t(x)\over dt} = Y_{H_t} \left( \phi^H_t(x) \right)
\end{align*}
which we call a Hamiltonian path. 

We need to put several assumptions to the given structures, which are still general enough to include the most important cases (e.g., toric manifolds as symplectic quotients of Euclidean spaces). 
\begin{hyp}\label{hyp21}
We assume that $\mu: M \to {\mf g}^*$ is proper, $0\in {\mf g}^*$ is a regular value of $\mu$ and the $G$-action restricted to $\mu^{-1}(0)$ is free. Moreover, the symplectic quotient $\ov{M}:= \mu^{-1}(0)/G$ has positive dimension.
\end{hyp}
With this hypothesis, $\mu^{-1}(0)$ is a smooth submanifold of $M$ and the symplectic quotient $\ov{M}$ is a symplectic manifold, which has a canonically induced symplectic form $\ov{\omega}$. Also, the Hamiltonian $(H_t)$ descends to a time-dependent Hamiltonian
\begin{align*}
\ov{H} = (\ov{H}_t) \in C^\infty(S^1 \times \ov{M})
\end{align*}
by the $G$-invariance of $(H_t)$. It is easy to check that $Y_{H_t}$ is tangent to $\mu^{-1}(0)$ and the projection $\mu^{-1}(0) \to \ov{M}$ pushes $Y_{H_t}$ forward to $Y_{\ov{H}_t}$. Then we assume
\begin{hyp}\label{hyp22}
The induced Hamiltonian $\ov{H}_t: \ov{M} \to {\mb R}$ is nondegenerate in the usual sense.
\end{hyp}

Finally, in the case when $M$ is noncompact, we need the following convexity condition (see \cite[Definition 2.6]{Cieliebak_Gaio_Mundet_Salamon_2002}).
\begin{hyp}\label{hyp23} There exists a pair $({\mf f}, {\mf J})$, where ${\mf f}: M \to [0, +\infty)$ is a $G$-invariant and proper function, and ${\mf J}$ is a $G$-invariant, $\omega$-compatible almost complex structure on $M$, such that there exists a constant ${\mf c}_0 >0$ with
\begin{align*}
{\mf f}(x) \geq {\mf c}_0 \Longrightarrow \langle \nabla_\xi \nabla {\mf f}(x), \xi \rangle + \langle \nabla_{{\mf J}\xi} \nabla {\mf f}(x), {\mf J} \xi \rangle \geq 0,\ d{\mf f}(x) \cdot {\mf J} X_{\mu(x)} \geq 0,\ \forall \xi \in T_x M.	
\end{align*}
\end{hyp}
In this paper, to achieve transversality, we need to perturb ${\mf J}$ near $\mu^{-1}(0)$ (see the appendix). The above condition is only about the behavior ``near infinity'', so such perturbations don't break the hypothesis.

\subsection{Equivariant topology}

\subsubsection{Equivariant spherical classes}

Recall that the Borel construction for the action of $G$ on $M$ is $M_G:= EG\times_G M$, where $EG \to BG$ is a universal $G$-bundle over the classifying space $BG$. The equivariant (co)homology of $M$ is the ordinary (co)homology of $M_G$, denoted by $H_*^G(M)$ for homology and $H^*_G(M)$ for cohomology. 

On the other hand, for any smooth manifold $N$, we denote by $S_2(N)$ to be the image of the Hurwitz map $\pi_2(N) \to H_2(N; {\mb Z})$, and classes in $S_2(N)$ are called spherical classes. We define the equivariant spherical homology of $M$ to be $S_2^G(M):= S_2(M_G)$. Geometrically, a generator of $S_2^G(M)$ can be represented by a smooth principal $G$-bundle $P \to S^2$ and a smooth section $\phi: S^2\to P\times_G M$. So there is a natural map $S_2(M) \to S_2^G(M)$. We denote the class of the pair $(P, \phi)$ to be $[P, \phi]\in S_2^G(M)$. 

\subsubsection{Equivariant symplectic form and equivariant Chern numbers}

The equivariant cohomology of $M$ can be computed using the equivariant de Rham complex $\left( \Omega^*(M)^G, d^G \right)$. In $\Omega^2(M)^G$, there is a distinguished closed form $\omega^G = \omega - \mu$, which represents an equivariant cohomology class $\left[ \omega^G\right] \in H^2_G(M; {\mb R})$. We are interested in the pairing $\left\langle \left[ \omega^G \right], \left[P, u \right] \right\rangle \in {\mb R}$. It can be computed in the following way. Choose any smooth connection $A$ on $P$. Then there exists an associated closed 2-form $\omega_A $ on $P\times_G M$, called the {\bf minimal coupling form}. If we trivialize $P$ locally over a subset $U\subset S^2$, such that $A= d + \alpha$, $\alpha \in \Omega^1(U, {\mf g})$ with respect to this trivialization, then $\omega_A$ can be written as
\begin{align*}
\omega_A= \pi^* \omega - d( \mu\cdot \alpha)\in \Omega^2( U \times M).
\end{align*}
Then we have
\begin{align*}
\left\langle [\omega^G], [P, u] \right\rangle = \int_{S^2} u^* \omega_A
\end{align*}
which is independent of the choice of $A$. On the other hand, any $G$-invariant, $\omega$-compatible almost complex structure $J$ on $M$ makes $TM$ an equivariant complex vector bundle. So we have the equivariant first Chern class
\begin{align*}
c_1^G:= c_1^G(TM) \in H^2_G(M; {\mb Z}).
\end{align*}
This is independent of the choice of $J$. 

\subsubsection{Kirwan maps}

The cohomological Kirwan map is a map
\begin{align*}
\kappa: H^*_G ( M; {\mb R}) \to H^*( \ov{M}; {\mb R}).
\end{align*}
Here we take ${\mb R}$-coefficients for simplicity. It is easy to check that 
\begin{align*}
\kappa( [\omega^G]) = [\ov{\omega}] \in H_2(\ov{M}; {\mb R}),\ \kappa(c_1^G) = c_1(T\ov{M}) \in H_2( \ov{M}; {\mb R}).
\end{align*}
We define
\begin{align*}
\begin{array}{c}
N_2^G(M) = {\rm ker} [\omega^G] \cap {\rm ker} c_1^G \subset S_2^G(M),\\[0.2cm]
N_2(\ov{M}) = {\rm ker} [\ov\omega] \cap {\rm ker} c_1(T\ov{M}) \subset S_2(\ov{M}),\\[0.2cm]
\Gamma:= S_2^G(M)/ N_2^G(M).
\end{array}
\end{align*}

\subsection{The spaces of loops and equivalence classes}

Let $\wt{\mc L}$ be the space of smooth contractible loops in $M \times {\mf g}$. A general element of $\wt{\mc L}$ is denoted by 
\begin{align*}
\wt{x}:= (x, f): S^1 \to M \times {\mf g}.
\end{align*}

Let $\wt{\mf L}$ be a covering space of $\wt{\mc L}$, consisting of triples ${\mf x}:= \left( x, f, [w]\right)$ where $\wt{x}= (x, f) \in \wt{\mc L}$ and $[w]$ is an equivalence class of smooth extensions of $x$ to the disk ${\mb D}$. The equivalence relation is described as follows. For each pair $w_1, w_2: {\mb D}\to M$ both bounding $x:S^1 \to M$, we have the continuous map
\begin{align*}
w_{12}:= w_1 \# (-w_2): S^2 \to M
\end{align*}
by gluing them along the boundary $x$. It defines a class $[w_{12}]\in S_2^G(M)$, via the natural map $S_2(M) \to S_2^G(M)$. We define
\begin{align*}
w_1 \sim w_2 \Longleftrightarrow [w_{12}]=0 \in S_2^G(M).
\end{align*}

Denote by $LG:= C^\infty( S^1, G)$ the smooth free loop group of $G$. For any point $x_0 \in M$, there is a homomorphism
\begin{align*}
l(x_0): \pi_1(G) \to \pi_1(M, x_0)
\end{align*}
which is induced by mapping a loop $t\mapsto \gamma(t) \in G$ to a loop $t \mapsto \gamma(t) x_0 \in M$. For different $x_1 \in M$ and a homotopy class of paths connecting $x_0$ and $x_1$, we have an isomorphism $\pi_1(M, x_0) \simeq \pi_1(M, x_1)$ which intertwines $l(x_0)$ and $l(x_1)$. So ${\rm ker} l(x_0) \subset \pi_1(G)$ is independent of $x_0$. We define
\begin{align*}
L_M G:= \left\{ \gamma\in C^\infty (S^1, G) \ |\ [\gamma] \in {\rm ker} l(x_0) \subset \pi_1(G) \right\}.
\end{align*}  
Let $L_0 G \subset L_M G$ be the subgroup of contractible loops in $G$.

It is easy to see that $L_M G$ acts on $\wt{\mc L}$ (on the right) by
\begin{align*}
\begin{array}{ccc}
 \wt{\mc L} \times L_M G & \to & \wt{\mc L}\\
 \displaystyle \left((x, f), h\right) & \mapsto & h^*(x, f)(t) =  \left( h(t)^{-1} x(t), {\rm Ad}_{h(t)}^{-1} ( f(t)) + h(t)^{-1} \partial_t h(t) \right).
\end{array}
\end{align*}
Here the action on the second component can be viewed as the gauge transformation on the space of $G$-connections on the trivial bundle $S^1 \times G$. (For short, we denote by $d \log h$ the ${\mf g}$-valued 1-form $h^{-1} d h$, which is the pull-back by $h$ of the left-invariant Maurer-Cartan form on $G$.)

However, $L_M G$ doesn't act on $\wt{\mf L}$ naturally; only the subgroup $L_0 G$ does: for a contractible loop $h: S^1 \to G$, extend $h$ arbitrarily to $h: {\mb D} \to G$. The homotopy class of extensions is unique because $\pi_2(G)= 0$ for any connected compact Lie group (see \cite{Cartan_1936}). Then the class of $( h^{-1} x, h^* f, [ h^{-1} w] )$ in $\wt{\mf L}$ is independent of the extension. It is easy to see that the covering map $\wt{\mf L}\to \wt{\mc L}$ is equivariant with respect to the inclusion $L_0 G \to L_M G$. Hence it induces a covering
\begin{align*}
\wt{\mf L} / L_0 G \to \wt{\mc L} / L_M G.
\end{align*}

\subsection{The action functional}

\subsubsection{The deck transformations}

We now define an action of $S_2^G(M)$ on $\wt{\mf L} / L_0 G$. Take a class $A \in S_2^G(M)$ represented by a pair $(P, u)$, where $P \to S^2$ is a principal $G$-bundle and $u: S^2 \to P\times_G M$ is a section of the associated bundle. 

Consider $U \simeq {\mb C}^* \cup \{\infty\}  \subset S^2$ as the complement of the south pole $0 \in S^2$. Take an arbitrary trivialization $\phi: P|_{U}\to U\times G$, which induces a trivialization
\begin{align*}
\phi: P\times_G M|_U \to U \times M.
\end{align*}
Then $\phi \circ u : U \to M$ and there exist a loop $h: S^1 \to G$ and $x \in M$ such that
\begin{align}\label{equation21}
\lim_{r\to 0}  \phi \circ u (re^{2\pi {\bm i} t})  = h(t) x.
\end{align}
Note that the homotopy class of $h$ is independent of the choices of $\phi$ and $x$. Then for any element $(x, f, [w])\in \wt{\mf L}$, find a smooth path $\gamma: [0,1] \to M$ such that $\gamma(0) = w(0)$ and $\gamma(1) = x$. Then define $\gamma_h: S^1 \times [0,1] \to M$ by $\gamma(t, \nu) = h(t) \gamma(\nu)$. On the other hand, view ${\mb D}\setminus \{0\} \simeq (-\infty, 0] \times S^1$. Consider the map
\begin{align*}
w_h (r, t) = h(t) w(r, t)
\end{align*}
and the ``connected sum'':
\begin{align*}
u\wt{\#} w:= \left( \phi\circ u \right) \# \gamma_h \# w_h: {\mb D} \to M.
\end{align*}
It extends the loop $x_h(t) = h(t) x(t)$. Denote $f_h:= {\rm Ad}_h f - \partial_t h \cdot h^{-1}$. We define
\begin{align}\label{equation22}
A \# [x, f, [w ]] = \left[  x_h,   f_h, [ u\wt{\#} w], \right] \in \wt{\mf L}/L_0 G.
\end{align}

On the other hand, there exists a morphism
\begin{align*}
S_2^G(M)  \to {\rm ker} l(x_0) \subset \pi_1(G)
\end{align*}
which sends the homotopy class of $[P, u]$ to the homotopy class of $h: S^1 \to G$ where $h$ is the one in (\ref{equation21}). Then it is easy to see the following.
\begin{lemma}
The action (\ref{equation22}) is well-defined (i.e., independent of the representatives and choices) and every deck transformation of the covering $\wt{\mf L}/ L_0 G  \to \wt{\mc L}/ L_M G$ is given by such an action.
\end{lemma}
Now by this lemma, we denote ${\mf L}:= \left( \wt{\mf L}/ L_0 G \right)/ N_2^G$, which is again a covering of ${\mc L}:= \wt{\mc L}/ L_M G$, with the group of deck transformations isomorphic to $\Gamma$. We will use ${\mf x}$ to denote an element in $\wt{\mf L}$, $[{\mf x}]$ for $\wt{\mf L}/L_0 G$ and $\llbracket{\mf x}\rrbracket$ for ${\mf L}$.

\subsubsection{The action functional}

We define a 1-form $\wt{\mc B}_H$ on $\wt{\mc L}$ by
\begin{align*}
 T_{(x, f)} \wt{\mc L}\ni (\xi, h)\mapsto       \int_{S^1}\left( \omega\left( \dot{x}(t)+ X_f-  Y_{H_t},\xi(t)\right) + \langle \mu(x(t)), h(t)\rangle \right) dt \in {\mb R}.
\end{align*}
Its pull-back to $\wt{\mf L}$ is exact and is the differential of the following action functional:
\begin{align*}
\wt{\mc A}_H({\mf x}) = \wt{\mc A}_H(x, f, [w]):= -\int_B w^* \omega+ \int_{S^1}  \left( \mu(x(t))\cdot f(t) - H_t(x(t))\right) dt.
\end{align*}

The zero set of the one-form $\wt{\mc B}_H$ consists of pairs $(x,f)$ such that 
\begin{align*}
\mu(x(t))\equiv 0, \ \dot{x}(t) + X_{f(t)} (x(t))- Y_{H_t}(x(t))=0.
\end{align*} 
The critical point set of $\wt{\mc A}_H$ is the preimage of ${\rm Zero}\wt{\mc B}_H$ under the covering $\wt{\mf L} \to \wt{\mc L}$.

\begin{lemma}
$\wt{\mc A}_H$ is $L_0G$-invariant and $\wt{\mc B}_H$ is $L_M G$-invariant.
\end{lemma}

\begin{proof}
For any $h: S^1 \to G$, extend it to a map $h: {\mb D}\to G$. Then we have
\begin{align*}
(h^{-1} w)^* \omega = \omega \left( \partial_x (h^{-1} w), \partial_y (h^{-1}w) \right) dx dy = w^* \omega + d \left( \mu(h^{-1} w) \cdot d \log h \right).
\end{align*}
Also, we see that
\begin{align*}
\mu(h^{-1}(t) x(t)) \cdot \left({\rm Ad}_{h(t)}^{-1} f(t) + h(t)^{-1} h'(t) \right) = \mu(x(t)) \cdot f(t) + \left. \left( \mu(h^{-1}w) \cdot d\log h \right)\right|_{S^1}.
\end{align*}
The invariance of $\wt{\mc A}_H$ follows from Stokes' theorem and the $G$-invariance of $(H_t)$ and the invariance of $\wt{\mc B}_H$ follows in a similar way.
\end{proof}

Therefore $\wt{\mc A}_H$ descends to $\wt{\mf L}/ L_0 G$ and it satisfies the following.
\begin{lemma}
For any $[{\mf x}] = [ x, f, [w ]] \in \wt{\mf L} / L_0 G$ and any $A \in S_2^G(M)$, we have
\begin{align*}
\wt{\mc A}_H \left( A\# [\mf x] \right) = \wt{\mc A}_H \left( [\mf x] \right) -  \left\langle \left[ \omega^G \right], A \right\rangle.
\end{align*}
\end{lemma}

\begin{proof}
Use the same notation as we define the action $A\# [{\mf x}]$, we see that
\begin{multline*}
\int_{{\mb D}\setminus \{0\}} w_h^* \omega = \int_{\mb D} w^* \omega + \int_{-\infty}^0 ds \int_{S^1} dt \omega \left( h_* \partial_s w, h_* X_{\partial_t \log h} (w) \right)\\
= \int_{\mb D} w^* \omega - \int_{-\infty}^0 \int_0^1  d\left( \mu(w) \cdot d \log h \right) = \int_{\mb D} w^* \omega - \int_{S^1} ( \mu(x(t)) - \mu(w(0)) ) d \log h.
\end{multline*}
Also
\begin{align*}
\int_{S^1 \times [0,1]} \gamma_h^* \omega = - \int_{S^1} \left( \mu( w(0)) - \mu(x) \right) d\log h.
\end{align*}
In the same way we can calculate
\begin{align*}
\int_{S^2\setminus \{0\}} \left( \phi\circ u \right)^* \omega = \left\langle \left[ \omega^G \right], \left[ P, u \right] \right\rangle - \int_{S^1} \mu(x) d\log h.
\end{align*}
So we have
\begin{multline*}
\wt{\mc A}_H \left( A\# [\mf x] \right) = - \int_{{\mb D}} \left( u \wt{\#} w \right)^* \omega + \int_{S^1} \left\{ \left\langle \mu( h(t) x(t)), {\rm Ad}_h f - h' h^{-1} \right\rangle - H_t( h(t) x(t))  \right\} dt \\
= - \left\langle \left[ \omega^G \right], A \right\rangle   - \int_{\mb D} w^* \omega + \int_{S^1} \langle \mu(x(t)), f(t) \rangle - H_t(x(t)) dt =- \left\langle\left[ \omega^G \right], A \right\rangle  + \wt{\mc A}_H \left( [\mf x]\right).
\end{multline*}
\end{proof}
This lemma implies that $\wt{\mc A}_H$ descends to a well-defined function ${\mc A}_H: {\mf L} \to {\mb R}$. Our Floer theory will be formally a Morse theory of the pair $\left( {\mf L}, {\mc A}_H \right)$.

\subsubsection{Lagrange multiplier}

Before we move on to the chain complex, we see that ${\mc A}_H$ is a Lagrange multiplier function associated to the action functional ${\mc A}_{\ov{H}}$ of the induced Hamiltonian $\ov{H}$ on the symplectic quotient $\ov{M}$. Let $\wt{\mf L}_{\ov{M}}$ be the space of contractible loops in $\ov{M}$ and let ${\mf L}_{\ov{M}}$ be pairs $(\ov{x}, [\ov{w}])$ where $\ov{x} \in \wt{\mf L}_{\ov{M}}$ and $\ov{w}: {\mb D } \to \ov{M}$ extends $\ov{x}$; $[\ov{w}] = [\ov{w}']$ if $(-\ov{w}) \# \ov{w}'$ is annihilated by both $\ov{\omega}$ and $c_1(T\ov{M})$. Then for any $(\ov{x}, [\ov{w}]) \in {\mf L}_{\ov{M}}$, we can pull-back the principal $G$-bundle $\mu^{-1}(0)\to \ov{M}$ to ${\mb D}$. Any trivialization (or equivalently a section $s$) of this bundle over ${\mb D}$ induces a map $w_s: {\mb D} \to \mu^{-1}(0)$ whose boundary restriction, denoted by $x: S^1\to \mu^{-1}(0)$, lifts $\ov{x}$. Now, if $(\ov{x}, [\ov{w}]) \in {\rm Crit} {\mc A}_{\ov{H}}$, i.e.
\begin{align*}
0= \ov{x}'(t) - Y_{\ov{H}_t}(\ov{x}(t))  = \left(\ov\pi_\mu\right)_* \left( x'(t) - Y_{H_t}(x(t))\right)
\end{align*}
there exists a smooth function, $f_s: S^1 \to {\mf g}$ such that
\begin{align*}
x'(t) + X_{f_s(t)}(x(t)) -  Y_{H_t}(x(t)) = 0.
\end{align*}
Then this gives a map
\begin{align}\label{equation23}
\begin{array}{cccc}
\iota: & \wt{\mf L}_{\ov{M}} & \to & \wt{\mf L}/ L_0 G \\
       &  (\ov{x}, [\ov{w}] ) & \mapsto & \left[ x_s, f_s, [w_s] \right]
\end{array}
\end{align}
By the correspondence of the symplectic forms and Chern classes between upstairs and downstairs, we have 
\begin{prop}\label{proposition27}
The class $\left[ x_s, f_s, [w_s]\right]$ is independent of the choice of the section $s$ and only depends on the homotopy class of $\ov{w}$. Moreover, it induces a map
\begin{align*}
\iota: \left( {\mf L}_{\ov{M}}, {\rm Crit} {\mc A}_{\ov{H}}\right)  \to \left( {\mf L}, {\rm Crit} {\mc A}_H \right)
\end{align*}
\end{prop}

\subsection{The Floer chain complex}\label{subsection24}

For $R= {\mb Z}_2$, ${\mb Z}$ or ${\mb Q}$, we consider the downward Novikov ring over the base ring $R$:
\begin{align*}
\Lambda_R:= \Lambda_R^\downarrow:= \left\{ \left. \sum_{B \in \Gamma} \lambda_B q^B\ \right| \ \forall K>0,  \# \left\{ B\in \Gamma \ |\ \left\langle [\omega^G], B \right\rangle > K,\ \lambda_B \neq 0 \right\} < \infty\right\}.
\end{align*}

We denote the free $\Lambda_R$-module generated by ${\rm Crit}{\mc A}_H\subset {\mf L}$ by $\widehat{VCF}(M, \mu; H; \Lambda_R)$. We define an equivalence relation on $\widehat{VCF}(M, \mu; H; \Lambda_R)$ by
\begin{align*}
\llbracket {\mf x} \rrbracket \sim q^{-B} \llbracket {\mf x}' \rrbracket \Longleftrightarrow B\# \llbracket {\mf x}'\rrbracket = \llbracket  {\mf x} \rrbracket \in {\mf L}.
\end{align*}
Denote by $VCF(M, \mu; H; \Lambda_R)$ the quotient $\Lambda_R$-module by the above equivalence relation, which is will be graded by a Conley-Zehnder type index which will be defined later in Section \ref{section4}.

In the remaining of this paper, we will restrict to the case that $R= {\mb Z}$.

\subsection{Gradient flow and symplectic vortex equation}

Now we choose an $S^1$-family of $G$-invariant, $\omega$-compatible almost complex structures $J = (J_t)$ on $M$, we assume that 
\begin{align*}
{\mf f}(x) \geq {\mf c}_0 \Longrightarrow J_t(x) = {\mf J}(x)
\end{align*}
where $({\mf f}, {\mf J})$ is the convex structure which we assume to exist in Hypothesis \ref{hyp23}. Then for each $t\in S^1$, $\omega$ and $J_t$ defines a Riemannian metric on $M$. They induce an $L^2$-metric on the loop space $LM$. On the other hand, we fix a biinvariant metric on the Lie algebra ${\mf g}$ which induces a metric on $L{\mf g}$; it also identifies ${\mf g}$ with ${\mf g}^*$ and we use this identification everywhere in this paper without mentioning it. These choices induce a metric on ${\mf L}$. 

Then, it is easy to see that formally, the equation for the negative gradient flow line of ${\mc A}_H$ is the following equation for a pair $\wt{u} = (u, \Psi): \Theta \to M \times {\mf g}$
\begin{align}\label{equation24}
\left\{ \begin{array}{ccc} 
\displaystyle {\partial u \over\partial s}  + J_t \left( {\partial u \over \partial t} + X_\Psi(u)- Y_{H_t}(u)  \right) &= &0, \\[0.3cm]
\displaystyle {\partial \Psi \over \partial s} + \mu(u) &= &0. \end{array}\right.
\end{align}
This equation is invariant under the action of $LG$ on $(u, \Psi)$, which is defined by
\begin{align}\label{equation25}
      g^* (u, \Psi)(s, t) = \left( g(t)^{-1} u(s, t), {\rm Ad}_{g(t)}^{-1} \Psi(s, t) + g(t)^{-1} \partial_t g(t) \right).
\end{align}

\begin{defn}\label{defn28} The energy for a flow line $\wt{u} = (u, \Psi)$ is defined to be
\begin{align*}
E\left( \wt{u}\right) := E\left( u, \Psi \right) = \Big\| {\partial u\over \partial s} \Big\|_{L^2(\Theta)}^2 + \Big\| \mu(u) \Big\|_{L^2(\Theta)}^2.
\end{align*}
Here the $L^2$-norm is taken with respect to the standard metric on $\Theta$ and the $t$-dependent metric on $M$ determined by $\omega$ and $J_t$.
\end{defn}

The connection form $d+ \Psi dt$ has already been put in the {\bf temporal gauge}, i.e., it has no $ds$ component. A more general equation on pairs $(u, \alpha)$, with $\alpha = \Phi ds + \Psi dt \in \Omega^1(\Theta) \otimes {\mf g}$, thought of a connection form on the trivial $G$-bundle over $\Sigma$, reads
\begin{align}\label{equation26}
\left\{ \begin{array}{ccc}  
\displaystyle {\partial u \over \partial s} +   X_\Phi(u) + J_t \left( {\partial u \over \partial t}  +  X_\Psi(u) - Y_{H_t}(u) \right)&= &0,\\[0.3cm]
\displaystyle {\partial \Psi \over \partial s} - {\partial \Phi \over \partial t} + [\Phi, \Psi] + \mu(u)  & = & 0.\end{array} \right.
\end{align}
Throughout this paper, the variable of (\ref{equation26}) is denoted by a triple $\wt{u} = (u, \Phi, \Psi) \in C^\infty \left( \Sigma, M \times {\mf g} \times {\mf g} \right)$. (\ref{equation26}) is invariant under the action by ${\mc G}_\Theta : = C^\infty\left( \Theta, G\right)$, which is defined by
\begin{align*}
\begin{array}{ccc} 
{\mc G}_\Theta \times C^\infty \left( \Theta, M \times {\mf g} \times {\mf g} \right) & \to & C^\infty \left( \Theta, M \times {\mf g} \times {\mf g } \right)\\
g^* \left( \begin{array}{c} u\\ \Phi \\ \Psi  \end{array} \right)(s, t) & = & \left( \begin{array}{c} g(s, t)^{-1} u(s, t) \\ {\rm Ad}_{g(s, t)}^{-1} \Phi(s, t) + g(s, t)^{-1} \partial_s g(s, t)\\
 {\rm Ad}_{g(s, t)}^{-1} \Psi(s, t) + g(s, t)^{-1} \partial_t g(s, t) \end{array}\right)
\end{array}
\end{align*} 
Solutions to (\ref{equation26}) are called generalized flow lines. The energy is defined by 
\begin{align*}
E(u, \Phi, \Psi) = \Big\| {\partial u \over \partial s} + X_{\Phi} \Big\|_{L^2(\Theta)}^2 + \Big\| \mu(u) \Big\|_{L^2(\Theta)}^2 .
\end{align*}
Every smooth generalized flow line is gauge equivalent via a gauge transformation in ${\mc G}_\Theta$ to a smooth flow line in temporal gauge; and the energy is gauge invariant.

\subsection{Moduli space and the naive idea of defining Floer homology}

Under the current setting, the definition of our vortex Floer homology group is very similar to that of Morse homology and ordinary Hamiltonian Floer homology. We briefly describe the construction and the details are provided in due course.

In Section \ref{section3} we will show that, any finite energy solution to (\ref{equation26}) is gauge equivalent to a solution $\wt{u}= (u, \Phi, \Psi)$ such that for some pair $\wt{x}_\pm = (x_\pm, f_\pm) \in {\rm Zero}\wt{\mc B}_H$,
\begin{align*}
\lim_{s \to \pm\infty} \Phi(s, t) = 0,\ \lim_{s\to \pm\infty} (u(s, \cdot), \Psi(s, \cdot)) = \wt{x}_\pm.
\end{align*}
Hence for any pair $\llbracket{\mf x}_\pm\rrbracket \in {\rm Crit} {\mc A}_H$, we can consider solutions which ``connect'' them. Those solution are called connecting orbits between $\llbracket{\mf x}_-\rrbracket$ and $\llbracket{\mf x}_+\rrbracket$. We denote by
\begin{align*}
{\mc M}\left( \llbracket{\mf x}_- \rrbracket, \llbracket{\mf x}_+ \rrbracket; J, H \right)
 \end{align*}
the moduli space of all such solutions, modulo gauge transformation.

In Section \ref{section6} we will show that, for certain type of Hamiltonians $H = (H_t)$ (which we call admissible ones, see Definition \ref{defn65}), and a generic choice of $S^1$-family of almost complex structures $J = (J_t)$ (which are admissible with respect to $H$, see Definition \ref{defn62}), the space ${\mc M}\left( \llbracket{\mf x}_-\rrbracket, \llbracket{\mf x}_+\rrbracket; J, H \right)$ is a smooth manifold, whose dimension is equal to the difference of the Conley-Zehnder indices (see Subsection \ref{subsection43}) of $\llbracket{\mf x}_\pm \rrbracket$. Moreover, since $M$ cannot have any nonconstant pseudoholomorphic spheres, ${\mc M}\left( \llbracket{\mf x}_-\rrbracket, \llbracket{\mf x}_+ \rrbracket; J, H \right)$ is compact modulo breaking. Finally, there exist coherent orientations on different moduli spaces, which is similar to the case of ordinary Hamiltonian Floer theory (see \cite{Floer_Hofer_Orientation}). So in our case, the signed counting of isolated gauged equivalence classes of trajectories has exactly the same nature as counting trajectories in the finite-dimensional Morse-Smale-Witten theory, which defines a boundary operator
\begin{align*}
\delta_J: VCF_*(M, \mu; H; \Lambda_{\mb Z}) \to VCF_{*-1}(M, \mu; H; \Lambda_{\mb Z}).
\end{align*}
The vortex Floer homology is then defined as the homology
\begin{align*}
VHF_*\left( M, \mu; J, H; \Lambda_{\mb Z} \right) = H \left( VCF_*(M, \mu; H; \Lambda_{\mb Z}), \delta_J \right).
\end{align*}
Moreover, for a different choice of the pair $(J', H')$, we can use continuation principle to prove that the chain complex $\left( VCF_*(M, \mu; H'; \Lambda_{\mb Z}), \delta_{J'}\right)$ is chain homotopic to $\left( VCF_*(M, \mu; H; \Lambda_{\mb Z}); \delta_J \right)$. There is a canonical isomorphism between the homologies, and we denote the common homology group by $VHF_*(M, \mu; \Lambda_{\mb Z})$. The details are given in Section \ref{section7} and the appendix.

\section{Asymptotic behavior}\label{section3}

In this section we analyze the asymptotic behavior of solutions $\wt{u} = (u, \Phi, \Psi)$ to (\ref{equation26}) which has finite energy and for which $u(\Theta)$ has compact closure in $M$. We call such a solution a {\bf bounded solution}. We denote the space of bounded solutions by $\wt{\mc M}_\Theta^b$. We can also consider the equation on the half cylinder $\Theta_+$ or $\Theta_-$ and denote the spaces of bounded solutions over $\Theta_\pm$ by $\wt{\mc M}_{\Theta_\pm}^b$. 

The main theorem of this section is
\begin{thm}
\begin{enumerate}
\item Any $(u, \Phi, \Psi )\in \wt{\mc M}_{\Theta_\pm} ^b$ is gauge equivalent (via a smooth gauge transformation $g: \Theta_\pm \to G$) to a solution $(u', \Phi', \Psi')\in \wt{\mc M}_{\Theta_\pm}^b$ such that there exist $\wt{x}_\pm = (x_\pm, f_\pm) \in {\rm Zero} \wt{\mc B}_H$
and
\begin{align*}
\lim_{s\to \pm \infty} \left( u'(s, \cdot), \Psi'(s, \cdot) \right) = \wt{x}_\pm,\ \lim_{s\to \pm\infty} \Phi(s, \cdot) = 0
\end{align*}
uniformly for $t \in S^1$.

\item There exists a compact subset $K_H\subset M$ such that for any $(u, \Phi, \Psi) \in \wt{\mc M}^b_\Theta$, we have $u(\Theta) \subset K_H$.
\end{enumerate}
\end{thm}

We will prove (1) for $\wt{u}\in \wt{\mc M}_{\Theta_+}^b$ in temporal gauge, i.e., $\Phi \equiv 0$ and the case for $\Theta_-$ is the same. Then (2) follows from a maximum principle argument, given at the end of this section. The proof is based on estimates on the energy density, which has been given by others in different settings (see \cite{Cieliebak_Gaio_Mundet_Salamon_2002}, \cite{Gaio_Salamon_2005}).

\subsection{Estimate on the energy density}\label{subsection31}

In this subsection we prove the following result. 
\begin{prop}\label{prop32}
If $(u, \Psi)$ is a bounded solution to (\ref{equation24}) on $\Theta_+$, then the energy density function $\left| \partial_s u \right|^2 +\left| \partial_s \Psi \right|^2$ converges to 0 as $s\to \pm\infty$, uniformly in $t$.
\end{prop}

\begin{rem}
The basic idea of proving Proposition \ref{prop32} is standard, i.e., to estimate the Laplacian of the energy density function and to use a mean value estimate. In the context of vortex equation, the calculation and estimates can be found in \cite[Section 9]{Gaio_Salamon_2005}, in the case that $J$ is domain-independent and $H \equiv 0$. We carry out the estimates in detail in our situationin the case that $J$ depends on the paramter $t$ and $H \neq 0$ (which can be transformed into the case that $H \equiv 0$, as we will see), though there is no essential difficulty for this extension. 
\end{rem}

\begin{proof}[Proof of Proposition \ref{prop32}, Part I] We first transform the problem to the case where we can assume that $H \equiv 0$ as in \cite[Remark 3.2, 3.4]{Frauenfelder_thesis}. Let $\phi_H^t: M \to M$, $t \in {\mb R}$ be the family of Hamiltonian diffeomorphisms generated by $H_t$, i.e., the solution to 
\begin{align*}
\phi_H^0 ={\rm Id}_M,\ {d\over dt} \phi_H^t(x) = Y_{H_t}( \phi_H^t(x)).
\end{align*}
Suppose $(u, \Psi)$ is a solution to (\ref{equation24}) on $\Theta_+$. Regard $u$ as a map from $[0, +\infty) \times {\mb R}$ which is periodic in the second variable $t$ with period 1. Then define $v: [0, +\infty) \times {\mb R} \to M$ by
\begin{align*}
v(s, t) = \left( \phi_H^t\right)^{-1} u(s, t).
\end{align*}
Then by the $G$-invariance of $H_t$, we have
\begin{align*}
{\partial v \over \partial s} + \wt{J}_t \left( {\partial v \over \partial t} + X_\Psi(v) \right) = 0,\ {\partial \Psi \over \partial t} + \mu(v) = 0.
\end{align*}
Here $\wt{J}_t$ is the family of almost complex structures on $M$ defined by
\begin{align*}
\wt{J}_t =\left( \left( \phi_H^t \right)_*\right)^{-1} \circ J_t \circ \left( \phi_H^t \right)_*,\ t \in {\mb R}.
\end{align*}
Note that for each $t$, $\wt{J}_t$ is still $G$-invariant, $\omega$-compatible, and for $t \in [0,1]$, the family of metrics $\omega \left( \cdot, \wt{J}_t \cdot \right)$ is uniformly comparable with the original family of metrics determined by $\omega$ and $J_t$. Therefore, in order to prove Proposition \ref{prop32}, it suffices to prove that $|\partial_s|^2 + |\partial_s \Psi|^2$ converges to zero as $ s\to +\infty$ uniformly for $t \in [0,1]$, for a bounded solution $(u, \Psi)$ on $[0, +\infty) \times {\mb R}$ to (\ref{equation24}) in the case where $H\equiv 0$ and $J_t$ is parametrized by $t \in {\mb R}$. ({\it Proof to be continued.})
\end{proof}

We need to introduce some notations. The family of metrics $g_t:= \omega( \cdot, J_t \cdot)$ induces a metric connection $\nabla$ on the bundle $u^* TM$. We define
\begin{align*}
\nabla_{A, s} \xi = \nabla_s \xi,\ \nabla_{A, t} \xi = \nabla_t \xi + \nabla_\xi X_\Psi.
\end{align*}

On the trivial bundle $\Theta \times {\mf g}$, define the covariant derivative
\begin{align*}
\nabla_{A, s} \theta = \nabla_s \theta, \ \nabla_{A, t} \theta = \nabla_t \theta + [\Psi, \theta].
\end{align*}
We denote by $\nabla_A$ the direct sum connection on $u^* TM \times {\mf g}$. Note that it is compatible with respect to the metric on this bundle induced by the family $g_t$. 

Define the ${\mf g}$-valued 2-form $\rho_t$ on $M$ by
\begin{align}\label{equation31}
\langle \rho_t ( \xi_1, \xi_2), \eta \rangle_t = \langle \nabla_{\xi_1} X_\eta, \xi_2 \rangle_t = - \langle \nabla_{\xi_2} X_\eta, \xi_1   \rangle_t,\  \xi_i \in TM,\  \eta \in {\mf g}.
\end{align}
Here $\langle \cdot, \cdot \rangle_t$ is the inner product of the metric $g_t$. 


\begin{lemma}\label{lemma34} For any compact subset $K \subset M$, there exist positive constants $c_1$ and $c_2$ depending only on $(X, \omega, J, \mu, H)$ and $K$, such that for any solution $(u, \Psi)$ to (\ref{equation24}) on an open subset $\Sigma \subset [0, +\infty) \times [-1, 2]$ with $u(\Sigma)\subset K$, we have
\begin{align*}
\Delta \left( | v_s |^2 + |\mu(u)|^2  \right) \geq -c_1 | v_s|^4- c_2.
\end{align*}
\end{lemma}

\begin{proof} Abbreviate $D_s = \nabla_{A, s}$, $D_t = \nabla_{A, t}$. We have	
\begin{align*}
\begin{split}
{1\over 2} \Delta \left| v_s \right|^2  = &\ \left( \partial_s^2+ \partial_t^2\right) | v_s |^2
= \partial_s   \left\langle D_s v_s, v_s \right\rangle + \partial_t \langle D_t v_s, v_s \rangle \\
    = &\ \left| D_s v_s \right|^2 +  \left| D_t v_s \right|^2 +  \left\langle \left( D_s^2 +  D_t^2 \right) v_s, v_s \right\rangle.
	\end{split}
	\end{align*}
Then we denote
\begin{align*}
\begin{split}
D_s^2 v_s+ D_t^2 v_s & = D_s \left( D_s v_s + D_t v_t\right)  - D_t \left( D_s v_t - D_t v_s\right) + \left( D_t D_s- D_s D_t \right) v_t\\
 & = : Q_1 + Q_2 + Q_3.
\end{split}
\end{align*}

We compute $Q_i$, $i=1, 2, 3$ as follows.

\begin{align}\label{equation32}
\begin{split}
Q_1 = &\ D_s \left( D_s ( -J_t v_t) +  D_t (J_t v_s)\right)\\
 = &\ D_s \left( -\nabla_s (J_t v_t) +\nabla_t (J_t  v_s )+ \nabla_{J_t v_s} X_\Psi \right)\\	
 = &\ D_s \left( - (\nabla_s J_t) v_t - J_t \nabla_s v_t + (\nabla_t J_t) v_s + J_t \nabla_t v_s + \nabla_{J_t v_s} X_\Psi \right)\\
 = &\ D_s \left( - (\nabla_{v_s} J_t) v_t - J_t \nabla_s X_\Psi + (\nabla_t J_t) v_s + [J_t v_s, X_\Psi] + \nabla_{X_\Psi} (Jv_s) \right)\\
 = &\ D_s \big( - (\nabla_{v_s} J_t) v_t - J_t X_{\partial_s \Psi} - J_t \nabla_{v_s} X_\Psi + \dot{J}_t v_s + (\nabla_{v_t - X_\Psi} J_t)v_s \big) \\
  &\  + D_s \left(  J_t [v_s, X_\Psi] + (\nabla_{X_\Psi} J_t) v_s + J_t \nabla_{X_\Psi} v_s \right)\\
= &\ D_s \big( - (\nabla_{v_s} J_t) v_t + J_t X_\mu + \dot{J}_t v_s + (\nabla_{v_t} J_t) v_s \big).
\end{split}\\
\begin{split}\label{equation33}
Q_2 = &\ - D_t \left( D_s v_t - D_t v_s  \right) \\
 = &\ - D_t \left(  \nabla_s \left( \partial_t u+ X_\Psi \right)- \nabla_t v_s - \nabla_{v_s}X_\Psi \right) \\
 = &\ - D_t  X_{\partial_s \Psi} = D_t  X_{\mu(u)}\\
 = &\ \nabla_t X_\mu + \nabla_{X_\mu} X_\Psi \\
= &\ X_{d\mu \cdot \partial_t u} + \nabla_{v_t} X_\mu + [X_\mu , X_\Psi]\\
 = &\ X_{d\mu \cdot v_t} - X_{d\mu \cdot X_\Psi}  + \nabla_{v_t} X_\mu - X_{[\mu, \Psi]}\\
= &\ X_{d\mu \cdot v_t} + \nabla_{v_t} X_\mu
\end{split}
\end{align}

On the other hand, for any $(s, t)\in \Sigma$, any $\xi \in T_{u(s, t)} M$, we extend $\xi$ and $v_s(s, t)$ to $G$-invariant vector fields locally. Then for the curvature $R_t$ of $\omega( \cdot, J_t\cdot)$, we have
\begin{align}\label{equation34} 
R_t ( v_s, X_\Psi ) \xi = \nabla_{v_s} \nabla_{ X_\Psi} \xi - \nabla_{X_\Psi} \nabla_{v_s} \xi -\nabla_{[v_s, X_\Psi ]} \xi = \nabla_{v_s} \nabla_\xi X_\Psi - \nabla_{{\nabla_{v_s} \xi}} X_\Psi.
\end{align}
Hence
\begin{align}\label{equation35}
\begin{split}
Q_3 = &\ \left( D_t D_s- D_s D_t \right) v_t \\[0.1cm]
 = &\ \nabla_t (\nabla_s v_t) + \nabla_{\nabla_s v_t} X_\Psi - \nabla_s \Big( \nabla_t v_t + \nabla_{v_t} X_\Psi \Big) \\
 = &\ - R_t (v_s, \partial_t u) v_t - \nabla_s \nabla_{v_t} X_\Psi + \nabla_{\nabla_s v_t} X_\Psi + \Big(  {d\over dt} \nabla_s  \Big) v_t \\
 = &\ - R_t (v_s , \partial_t u + X_\Psi ) v_t - \nabla_{v_t} X_{\partial_s \Psi} +  \Big( {d\over dt} \nabla_s \Big) v_t \\
 = &\ -R_t (v_s, v_t) v_t + \nabla_{v_t} X_\mu + \Big( {d\over dt} \nabla_s \Big) v_t.
\end{split}
\end{align}
Here the fourth equality uses (\ref{equation34}).

By (\ref{equation32}), for some $C_1>0$ depending on $K$, we have
\begin{align*}
\begin{split}
\langle Q_1, v_s\rangle = &\ \left\langle \nabla_s \left( -\left( \nabla_{v_s} J_t\right) v_t +\left( \nabla_{v_t} J_t \right) v_s + J_t X_\mu + \dot{J}_t v_s \right), v_s \right\rangle\\
                       \geq &\ - C_1 \left( |v_s|^3 + |v_s|^2 + |v_s|^4 + |v_s|^2 \left| D_s v_s \right|  + |v_s| \left| D_s v_s \right| \right).
											\end{split}
\end{align*}
By (\ref{equation33}), for some $C_2>0$, we have
\begin{align*}
\langle Q_2, v_s\rangle =  \langle X_{d\mu \cdot v_t} + \nabla_{v_t} X_\mu, v_s \rangle \geq -C_2 |v_s|^2.
\end{align*} 
By (\ref{equation35}), for some $C_3>0$ depending on the solution, we have
\begin{align*}
\begin{split}
 \langle Q_3, v_s\rangle  = &\  \Big\langle - R_t (v_s, v_t) v_t + \nabla_{v_t} X_\mu  + \left( {d\over dt} \nabla_s \right) v_t , v_s \Big\rangle\\
                        \geq &\  -C_3 \left(  |v_s|^4 + |v_s|^3 + |v_s|^2  \right).
												\end{split}
\end{align*}

Hence for some $C_4>0$ and $c_1, c_2>0$, we have
\begin{align*}
\begin{split}
{1\over 2} \Delta |v_s|^2 = &\ |D_s v_s|^2 + |D_t v_t|^2+ \langle Q_1 + Q_2 + Q_3,  v_s\rangle \\[0.2cm]
\geq &\ | D_s v_s|^2 +\langle Q_1 + Q_2 + Q_3, v_s\rangle \\
    \geq &\ |D_s v_s|^2- C_4 \left( \sum_{i=2}^4 |v_s|^i + \sum_{i=1}^2 |v_s|^i \left| D_s v_s \right| \right)\\
		\geq &\ - c_1 |v_s|^4 - c_2.
            \end{split}
							\end{align*}

For the other part of the energy density, we have
\begin{align*}
{1\over 2} \Delta \left| \mu(u) \right|^2 = \left| \nabla_{A, s} \mu \right|^2 + \left| \nabla_{A, t} \mu \right|^2 + \langle \nabla_{A, s} \nabla_{A, s} \mu + \nabla_{A, t} \nabla_{A, t} \mu, \mu \rangle.
\end{align*}
Moreover, (see \cite[Lemma C.2]{Gaio_Salamon_2005})
\begin{align*}
\begin{split}
&\ \nabla_{A, s} \nabla_{A, s} \mu(u) + \nabla_{A, t} \nabla_{A, t} \mu(u) \\[0.1cm]
= &\ \nabla_{A, s} d\mu \cdot v_s + \nabla_{A, t} d\mu \cdot v_t  = \nabla_{A, t} \left( d\mu \cdot J_t v_s \right) - \nabla_{A, s} \left( d\mu \cdot J_t v_t \right)\\
= &\ -2 \rho_t ( v_s, v_t) + d\mu \cdot J_t \left( \nabla_{A, t} v_s - \nabla_{A, s} v_t \right)  + d\mu \Big(\dot{J}_t v_s \Big)\\
= &\ d\mu \cdot \Big( J_t X_\mu + \dot{J}_t v_s \Big) - 2\rho_t (v_s, v_t) .
\end{split}
\end{align*}
Here $\rho_t$ is a ${\mf g}$-valued 2-form on $M$ defined by
\begin{align*}
\langle \rho_t ( \xi_1, \xi_2), \eta \rangle = \langle \nabla_{\xi_1} X_\eta, \xi_2 \rangle = - \langle \nabla_{\xi_2} X_\eta, \xi_1   \rangle,\  \xi_i \in TM,\  \eta \in {\mf g}.
\end{align*}
Since $\displaystyle \sup_{u(\Sigma),\ t\in [-1, 2]} \left| \rho_t \right| <\infty$, there exist $c_3, c_4>0$ such that
\begin{align*}
\Delta \left| \mu(u) \right|^2 \geq - c_3 - c_4  \left| v_s \right|^4.
\end{align*}

\end{proof}

\begin{proof}[Proof of Proposition \ref{prop32}, Part II] We quote the following lemma.
\begin{lemma} (cf. \cite[Page 12]{Salamon_lecture})\label{lemma35}
Let $r>0$ and Let $\Omega\subset {\mb R}^2$ be an open subset containing the origin and $e: \Omega \to [0, +\infty)$ satisfying 
\begin{align}\label{equation36}
\Delta e\geq -A- Be^2,
\end{align}
then for any disk $B_r(0)\subset \Omega$ centered at the origin, we have
\begin{align*}
\int_{B_r(0)} e\leq {\pi \over 16 B}\ \Longrightarrow \ e(0)\leq {8 \over \pi r^2} \int_{B_r(0)} e + {A r^2\over 4}.
\end{align*}
\end{lemma}
On the other hand, Lemma \ref{lemma34} implies that for $e = |v_s|^2 + |\mu(u)|^2$, there exist constants $A, B$ satisfying (\ref{equation36}) for $\Omega = (0, +\infty) \times (-1, 2)$. Therefore Proposition \ref{prop32} is proved by applying the above lemma, as in the case of ordinary Hamiltonian Floer theory.
\end{proof}

\subsection{Approaching to periodic orbits}

\begin{prop}\label{prop36}
Any bounded solution $(u, \Psi)$ to (\ref{equation24}) on $\Theta_+$ is gauge equivalent to a solution $(u', \Phi', \Psi')$ to (\ref{equation26})  on $\Theta_+$ with the following properties
\begin{enumerate}
\item $(u', \Psi')|_{\{s\}\times S^1}$ converges in $C^0$ to an element of ${\rm Zero} \wt{\mc B}_H$ as $s\to +\infty$;
\vspace{0.2cm}
\item $\displaystyle \lim_{s\to +\infty} \Phi'(s, \cdot) =0$ uniformly in $t$.
\end{enumerate}
\end{prop}

\begin{proof} 
By Proposition \ref{prop61}, there exists a neighborhood $U$ of $\mu^{-1}(0)$ and a diffeomorphism
\begin{align}\label{equation37}
i_U: U \to \mu^{-1}(0) \times {\mf g}_\epsilon^*
\end{align}
where ${\mf g}^*$ is the $\epsilon$-open ball of ${\mf g}^*$ centered at the origin, such that $\mu\circ i_U^{-1}$ is equal to the projection $\mu^{-1}(0) \times {\mf g}^* \to {\mf g}_\epsilon^*$. Let $\pi_\mu$ be the projection on to the first factor, and define the almost complex structure
\begin{align*}
J_{0, t}: =  \pi_\mu^* \left(  J_t \left|_{\mu^{-1}(0)} \right. \right)
\end{align*}
and the vector field
\begin{align*}
Y_{0, H_t}:= \pi_\mu^*\left(  Y_{H_t} \left|_{\mu^{-1}(0)} \right. \right).
\end{align*}
Then there exists $K_1 >0$ depending on $(M, \omega, J, \mu, H_t)$ such that
\begin{align*}
\left| J_{0, t}(x)- J_t (x)\right|\leq K_1 |\mu(x)|,\ \left| Y_{0, H_t}(x)- Y_{H_t}(x)\right| \leq K_1 | \mu(x)|,\ \forall x\in U.
\end{align*}
We denote $\ov\pi: U \to \ov{M}$ the composition of $\pi_\mu$ with the projection $\mu^{-1}(0) \to \ov{M}$. 

If $(u, 0, \Psi)\in \wt{\mc M}_{\Theta_+}^b$, then by Proposition \ref{prop32}, $\mu(u)$ converges to $0$ uniformly as $s\to + \infty$. So for $N$ sufficiently large, $u(s, \cdot)$ maps $\Theta_+^N:= [N, +\infty )\times S^1$ into $U$. So on $\Theta_+^N$, we have
\begin{multline*}
\partial_s u+ J_{0, t} \left( \partial_t u + X_\Psi (u)-  Y_{0, H_t}(u)\right) \\
= (J_{0, t}- J_t ) \left( \partial_t u + X_\Psi(u)- Y_{0, H_t}\right) + J_t \left( Y_{H_t}- Y_{0, H_t}\right) .
\end{multline*}
Denoting $\ov{u}:= \ov\pi_\mu \circ u: \Theta_+^N \to \ov{M}$ and applying $\left( \ov\pi_\mu \right)_*$ to the above equality, we see
\begin{align*}
\left| \partial_s \ov{u} + \ov{J}_t \left( \partial_t \ov{u} - Y_{\ov{H}_t} (\ov{u}) \right) \right| \leq K_2 \epsilon
\end{align*}
for some constant $K_2$. Here $\ov{J} = (\ov{J}_t)$ is the induced family of almost complex structures on $\ov{M}$. Hence for $s \geq N$, the family of loops $\ov{u}(s, \cdot)$ in the quotient will be close (in $C^0$) to some 1-periodic orbits $\gamma: S^1 \to \ov{M}$ of $Y_{\ov{H}_t}$. 

Take a lift $p\in \mu^{-1}(0)$ with $\ov\pi(p)= \gamma(0)$. Then there exists a unique $g_p \in G$ such that
\begin{align}\label{equation38}
\phi_H^1 (p)= g_p p.
\end{align}
Suppose $g_p= \exp \xi_p$, $\xi_p\in {\mf g}$. It is easy to see that the loop
\begin{align*}
(x(t), f(t)) :=  \left( \exp(-t \xi_p) \phi_H^t (p), \xi_p \right)
\end{align*}
is an element of ${\rm Zero} \wt{\mc B}_H$. We will construct a gauge transformation $\wt{g}$ on $\Theta_+$ and show that $\wt{g}^* (u,0, \Psi)$ satisfies the condition stated in this proposition.

Take a local slice of the $G$-action near $p$. That is, an embedding $i: B^{2n-2k}_\delta \to \mu^{-1}(0)$ where $B^{2n-2k}_\delta$ is the $\delta$-ball in ${\mb R}^{2n-2k}$ such that $i(0)=p$ and $(y, g)\mapsto g(i(y))$ is a diffeomorphism from $B^{2n-2k}_\delta \times G$ onto its image. 

Denote $u(s, t)= \left( v(s, t), \xi(s, t)\right) \in \mu^{-1}(0)\times {\mf g}^*$ with respect to (\ref{equation37}). Then for $s$ large enough, there exists a unique $g(s)\in G$ such that 
 \begin{align}\label{equation39}
g(s) v(s, 0)\in i \left( B_\delta^{2n-2k} \right),\ g(s) v(s, 0)\to p.
\end{align}
Moreover, by the fact that $|\partial_s u|$ converges to zero, we see 
\begin{align}\label{equation310}
\lim_{s\to +\infty} \left| g(s)^{-1} \dot{g}(s)\right| =0. 
\end{align} 

Define $h_s(t) \in G$ by
\begin{align*}
h_s(0) =1,\  h_s(t)^{-1} {\partial h_s(t) \over \partial t} = \Psi (s, t).
\end{align*}
Then by the fact that $\lim_{s\to +\infty} \left| \partial_s \Psi \right| = 0$ we see that
\begin{align}\label{equation311}
\lim_{s\to \infty} \left| \partial_s \log h_s(t) \right| = 0.
\end{align}
 
Thus we have
\begin{align*}
\begin{split}
& \ \ \ \ d\left( g_p p, g(s) h_s(1) g(s)^{-1} p\right)\\
 & \leq   d\left( g_p p, \phi_H^1 g(s) v(s, 0) \right) + d\left( \phi_H^1 g(s) v(s, 0), g(s) h_s(1) v(s, 0)  \right) \\
  &\ \ \ \ + d \left( g(s) h_s(1) v(s, 0), g(s) h_s(1) g(s)^{-1} p\right)  \\
 & =   d\left( g_p p, \phi_H^1 g(s) v(s, 0) \right) + d\left( \phi_H^1 v(s, 0), h_s(1) v(s, 0)  \right)+  d \left(  g(s) v(s, 0),  p\right) \\
  & =:  d_1(s) + d_2(s) + d_3(s).
	\end{split}
	\end{align*}
Here $d$ is the $G$-invariant distance function induced by an invariant Riemannian metric. By (\ref{equation38}) and (\ref{equation39}), we have $d_1(s) + d_3(s) \to 0$. By the decay of energy density, i.e., 
\begin{align*}
 \lim_{s\to +\infty} \sup_t \left| \partial_t u + X_\Psi - Y_{H_t}\right| = 0,
\end{align*}
we have $d_2(s) \to 0$. Hence we have
\begin{align*}
\lim_{s\to +\infty} d\left( g_p p, g(s) h_s(1) g(s)^{-1} p\right) = 0.
\end{align*}
Since the $G$-action on $\mu^{-1}(0)$ is free, we have
\begin{align}\label{equation312}
\lim_{s\to +\infty} g(s) h_s(1) g(s)^{-1} = g_p
\end{align}

Then by (\ref{equation312}), there exists a continuous curve $\xi(s) \in {\mf g}$ defined for large $s$, such that 
\begin{align*}
g(s) h_s(1) g(s)^{-1}=\exp \xi(s),\  \lim_{s\to +\infty} \xi(s)=  \xi_p.
\end{align*}	

Then apply the gauge transformation
\begin{align}\label{equation313}
\widetilde{g}(s, t)= h_s(t)^{-1} g(s)^{-1} \exp(t \xi(s))
\end{align}
to the pair $(u, 0, \Psi)$, we see that
\begin{align*}
\left( \widetilde{g}^* u \right) (s, 0)= \widetilde{g}(s, 0)^{-1}(u(s, 0))= g(s) u(s, 0) \to p,\ \widetilde{g}^* \left( \Psi dt \right) = \xi dt + \wt{g}^{-1} \partial_s \wt{g} ds.
\end{align*}
By (\ref{equation310}) and (\ref{equation311}), we have
\begin{align*}
\lim_{s \to +\infty} \left\| \wt{g}^{-1}(s, \cdot) \partial_s \wt{g}(s, \cdot) \right\| = 0.
\end{align*}
Therefore
\begin{align*}\lim_{s\to +\infty} \left.  \left( \widetilde{g}^* u \right) \right|_{\{s\}\times S^1} = \left( \exp(-t \xi_p) \phi_H^t p, \xi_p \right)\in {\rm Zero} \wt{\mc B}_H.
\end{align*}
\end{proof}

\begin{defn}
Let $\wt{x}_\pm:= (x_\pm, f_\pm) \in {\rm Zero} \wt{\mc B}_H$. We denote 
\begin{align*}
\wt{\mc M}\left( \wt{x}_-, \wt{x}_+; J, H \right):= \left\{ (u, \Phi, \Psi) \in \wt{\mc M}_\Theta^b \ |\ \lim_{s\to \pm\infty} (u, \Phi, \Psi) |_{\{s\}\times S^1} = ( \wt{x}_\pm, 0)\right\}.
\end{align*}
For ${\mf x}_\pm = \left( \wt{x}_\pm, [w_\pm] \right) \in {\rm Crit} \wt{\mc A}_H$ which projects to $\wt{x}_\pm$ via ${\rm Crit} \wt{\mc A}_H  \to {\rm Zero} \wt{\mc B}_H$, we define
\begin{align*}
\wt{\mc M}\left( {\mf x}_-, {\mf x}_+; J, H \right) : = \left\{ (u, \Phi, \Psi) \in \wt{\mc M}(\wt{x}_-, \wt{x}_+) \ |\  \left[ u\# w_-\right] = \left[ w_+\right]  \right\}.
\end{align*}
Sometimes we omit the dependence on $J$ and $H$ in the notations.
\end{defn}

It is easy to deduce the following energy identity for which we omit the proof.
\begin{prop}\label{prop38}Let ${\mf x}_\pm \in {\rm Crit} \wt{\mc A}_H$. Then for any $(u, \Phi, \Psi) \in \wt{\mc M}\left( {\mf x}_-, {\mf x}_+ \right)$, we have
\begin{align*}
E\left( u, \Phi, \Psi \right) = \wt{\mc A}_H\left( {\mf x}_- \right)- \wt{\mc A}_H \left( {\mf x}_+ \right).
\end{align*}
\end{prop}

\subsection{Convexity and uniform bound on flow lines}

We will show in this subsection the following	uniform boundedness of bounded solutions. One can compare it with \cite[Theorem 3.9]{Frauenfelder_thesis}. 
\begin{prop}
There exists a compact subset $K_H \subset M$ such that for any $(u, \Phi, \Psi) \in \wt{\mc M}_\Theta^b$, $u(\Theta) \subset K_H$.
\end{prop}

\begin{proof}
The proof is to use maximum principle as in \cite[Subsection 2.5]{Cieliebak_Gaio_Mundet_Salamon_2002}. We claim that this proposition is true for 
\begin{align*}
K_H={\rm Supp} H \cup {\mf f}^{-1} \left( [0, {\mf c}_1] \right)
\end{align*}
where 
\begin{align*}
{\mf c}_1 = \max \left\{ {\mf c}_0, \sup_{|\mu(x)|\leq 1} {\mf f}(x) \right\}
\end{align*}
where ${\mf c}_0$ is the one in Hypothesis \ref{hyp23}. Suppose the claim is not true. Then there exists a solution $\wt{u} = (u, \Phi, \Psi)\in\wt{\mc M}_\Theta^b$ which violates this condition and $(s_0, t_0) \in \Theta$ such that $u(s_0, t_0) \notin {\rm Supp}H$ and ${\mf f}(u(s_0, t_0)) > {\mf c}_1$. On the other hand, by the previous results, we know that $\displaystyle \lim_{s \to \pm\infty} \mu(u(s, t)) = 0$ so $\displaystyle \lim_{s\to \pm \infty} {\mf f}(u(s, t)) \leq {\mf c}_1$. Hence ${\mf f}(u)$ achieves its maximum at some point of $\Theta$. As in the proof of \cite[Lemma 2.7]{Cieliebak_Gaio_Mundet_Salamon_2002}, we see that ${\mf f}(u)$ is subharmonic on $u^{-1} \left( M \setminus K_H \right)$ and hence ${\mf f}(u)$ must be constant. However, this contradicts with the fact that $\displaystyle \lim_{s\to \pm \infty} {\mf f}(u(s, t)) \leq {\mf c}_1$. 
\end{proof}

\section{Fredholm theory}\label{section4}
	
In this section we investigate the infinitesimal deformation theory of solutions to our equation (modulo gauge). 

Compared to the Fredholm theory studied in \cite[Section 4.1-4.3]{Frauenfelder_thesis}, this section has some minor difference. Since we only consider nondegenerate case, rather than the Morse-Bott case in \cite{Frauenfelder_thesis}, we don't need weighted Sobolev spaces to define the Banach manifolds. Moreover, the Fredholm theory is essentially determined by the limiting self-adjoint operators $R_{\wt{x}_\pm}$ (see Section \ref{section42}); if they have trivial kernel, the Fredholm theory is well-understood. The reader can refer to \cite[Section 2]{Salamon_lecture} for similar treatment in the classical case, and to \cite{Robbin_Salamon_spectral} in the abstract situation. We don't reprove properties about the spectrum of $R_{\wt{x}_\pm}$ as did in \cite[Section 4.1]{Frauenfelder_thesis} because they are standard results.

\subsection{Banach manifolds, bundles, and local slices}

First we fix two loops $\wt{x}_\pm \in {\rm Zero}\wt{\mc B}_H$. For any $k \geq 1$, $p>2$, we consider the space of $W^{k, p}_{loc}$-maps $\wt{u}:= (u, \Phi, \Psi): \Theta \to M \times {\mf g}\times {\mf g}$, such that $\Phi \in W^{k, p}\left( \Theta, {\mf g} \right)$ and $(u, \Psi)$ is asymptotic to $\wt{x}_\pm = (x_\pm, f_\pm)$ at $\pm \infty$ in $W^{k, p}$-sense. This means the following: there exist $T>0$ and $(\xi_-, \eta_-) \in \Gamma( (-\infty, -T] \times S^1, x_-^* TM \oplus {\mf g})$, $(\xi_+, \eta_+) \in \Gamma([T, +\infty) \times S^1, x_+^* TM \oplus {\mf g})$ such that
\begin{align*}
u(\pm s, t) = \exp_{x_\pm (t)} \xi_\pm (\pm s, t),\ \Psi(\pm s, t) = f_\pm(t) + \eta_\pm(\pm s, t),\  \forall s\geq T.
\end{align*}
Here we pull back $x_\pm$ to any cylinder $[a, b]\times S^1$ and still denote them by $x_\pm$.

The space of such objects is a separable Banach manifold, which we denote by
\begin{align*}
\wt{\mc B}^{k, p}( \wt{x}_-, \wt{x}_+).
\end{align*}
The tangent space at any element $\wt{u} \in \wt{\mc B}^{k, p}$ is the Sobolev space
\begin{align*}
T_{\wt{u}} \wt{\mc B}^{k, p}( \wt{x}_-, \wt{x}_+) = W^{k, p} \left( \Theta, u^* TM \oplus {\mf g} \oplus {\mf g} \right).
\end{align*}
We denote by $\wt{\exp}^t$ the exponential map of $M \times {\mf g}\times {\mf g}$, where the Riemannian metric on $M$ is $\omega( \cdot, J_t \cdot )$ which is $t$-dependent. Then the map $\wt\xi \mapsto \wt\exp^t_{\wt{u}} \wt\xi$ is a local diffeomorphism from a neighborhood of $0 \in T_{\wt{u}} \wt{\mc B}^{k, p}( \wt{x}_-, \wt{x}_+) $ and a neighborhood of $\wt{u}$ in $\wt{\mc B}^{k, p}(\wt{x}_-, \wt{x}_+)$.

Then consider a pair ${\mf x}_\pm = (x_\pm, f_\pm, [w_\pm] )\in {\rm Crit} \wt{\mc A}_H$ with $\wt{x}_\pm = (x_\pm, f_\pm) \in {\rm Zero} \wt{\mc B}_H$. We define
\begin{align*}
\wt{\mc B}^{k, p}({\mf x}_-, {\mf x}_+) := \Big\{ \wt{u} = (u, \Phi, \Psi) \in \wt{\mc B}^{k, p} ( \wt{x}_-, \wt{x}_+ )\ |\  [w_- \# u] = [w_+] \Big\}.
\end{align*}
This is the union of some connected components of $\wt{\mc B}^{k, p}(\wt{x}_-, \wt{x}_+)$. 

Let ${\mc G}^{k+1, p}_0$ be the space of $W^{k+1, p}_{loc}$-maps $g: \Theta \to G$ which are asymptotic to the identity of $G$ at $\pm \infty$ in $W^{k+1, p}$-sense. This means the following: there exists $T>0$ and $\eta_- \in W^{k+1, p}((-\infty, -T]\times S^1, {\mf g})$, $\eta_+ \in W^{k+1, p}([T, +\infty) \times S^1, {\mf g})$ such that
\begin{align*}
g(\pm s, t) = \exp \eta_\pm(\pm s, t),\ \forall s\geq T.
\end{align*}
Then ${\mc G}^{k+1, p}_0$ is a Banach Lie group. The gauge transformation extends to a free ${\mc G}^{k+1, p}_0$-action on $\wt{\mc B}^{k, p}(\wt{x}_-, \wt{x}_+)$ (resp. $\wt{\mc B}^{k, p}({\mf x}_-, {\mf x}_+)$), because the symplectic quotient $\ov{M}$ is a free quotient. Then this makes the quotient
\begin{align*}
{\mc B}^{k, p}(\wt{x}_-, \wt{x}_+):= \wt{\mc B}^{k, p} (\wt{x}_-, \wt{x}_+) / {\mc G}_0^{k+1, p}\ \left( {\rm resp.}\ {\mc B}^{k, p} ( {\mf x}_-, {\mf x}_+) := \wt{\mc B}^{k, p}({\mf x}_-, {\mf x}_+) / {\mc G}^{k+1, p}_0 \right)
\end{align*}
a Banach manifold. Indeed, to see this we have to construct local slices of the ${\mc G}^{k+1, p}_0$-action. For any $\wt{u}\in \wt{\mc B}^{k, p}(\wt{x}_-, \wt{x}_+)$ (whose image in ${\mc B}^{k, p}( \wt{x}_-, \wt{x}_+)$ is denoted by $[\wt{u}]$), consider the operator
\begin{align*}
\begin{array}{cccc}
d_0^*: & T_{\wt{u}} \wt{\mc B}^{k, p} (\wt{x}_-, \wt{x}_+) & \to & W^{k-1, p} \left( \Theta, {\mf g} \right)\\
 & ( \xi, \phi, \psi) & \mapsto & - d\mu(J_t \xi) - \partial_s \phi - [ \Phi, \phi] - \partial_t \psi - [\Psi, \psi],
\end{array}
\end{align*}
which is the formal adjoint of the infinitesimal ${\mc G}^{k+1, p}_0$-action. Then as in gauge theory, we have a natural identification
\begin{align}\label{equation41}
T_{[\wt{u}]} {\mc B}^{k, p} \left( \wt{x}_-, \wt{x}_+ \right) \simeq {\rm ker} d_0^*
\end{align}
(where the orthogonal complement is taken with respect to the $L^2$-inner product) and the exponential map $\wt{\exp}^t$ induces a local diffeomorphism
\begin{align*}
 {\rm ker} d_0^* \ni \wt\xi \mapsto \left[ \wt\exp^t_{\wt{u}} \wt\xi \right] \in {\mc B}^{k, p} \left( \wt{x}_-, \wt{x}_+ \right).
\end{align*}

If we have $g_\pm \in L_0 G$, and ${\mf x}_\pm' = g_\pm^* {\mf x}_\pm \in {\rm Crit} \wt{\mc A}_H$, then the pair $(g_-, g_+)$ extends to a smooth gauge transformation on $\Theta$ which identifies $\wt{\mc B}^{k, p}({\mf x}_-, {\mf x}_+)$ with $\wt{\mc B}^{k, p}({\mf x}_-', {\mf x}_+')$. Then, if $[{\mf x}_\pm] \in {\rm Crit} \wt{\mc A}_H / L_0 G$, then we can denote by ${\mc B}^{k, p}([{\mf x}_-], [{\mf x}_+])$ to be the common quotient space. Finally, each $A\in S_2^G(M)$ defines a natural identification
\begin{align*}
{\mc B}^{k, p}\left( [{\mf x}_-], [{\mf x}_+] \right) \to {\mc B}^{k, p} \left( A\# [{\mf x}_-], A\# [{\mf x}_+] \right).
\end{align*}
Then for two pairs $\llbracket{\mf x}_\pm\rrbracket \in {\rm Crit} {\mc A}_H$, consider
\begin{align*}
\bigcup_{ [{\mf y}_\pm] \in {\rm Crit} \wt{\mc A}_H/ L_0 G,\ \llbracket{\mf y}_\pm\rrbracket = \llbracket {\mf x}_\pm\rrbracket } {\mc B}^{k, p}( [{\mf y}_-], [{\mf y}_+]),
\end{align*}
which has a natural action by $N_2^G(M) \subset S_2^G(M)$. We define
\begin{align*}
{\mc B}^{k, p}(\llbracket{\mf x}_-\rrbracket, \llbracket{\mf x}_+\rrbracket) := \left( \bigcup_{ [{\mf y}_\pm] \in {\rm Crit} \wt{\mc A}_H/ L_0 G,\ \llbracket{\mf y}_\pm\rrbracket = \llbracket {\mf x}_\pm\rrbracket } {\mc B}^{k, p}( [{\mf y}_-], [{\mf y}_+]) \right) / N_2^G(M).
\end{align*}

Over $\wt{\mc B}^{k, p}({\mf x}_-, {\mf x}_+)$, we have the smooth Banach space bundle $\wt{\mc E}^{k-1, p}({\mf x}_-, {\mf x}_+)$, whose fibre over $\wt{u}$ is the Sobolev space
\begin{align*}
\wt{\mc E}^{k-1, p}_{\wt{u}}  \left( {\mf x}_-, {\mf x}_+ \right):= W^{k-1, p} \left( u^* TM \oplus {\mf g} \right).
\end{align*}
The ${\mc G}^{k+1, p}_0$-action makes $\wt{\mc E}^{k-1, p}$ an equivariant bundle, hence descends to a Banach space bundle
\begin{align*}
{\mc E}^{k-1, p} ( {\mf x}_-, {\mf x}_+) \to {\mc B}^{k, p} ({\mf x}_-, {\mf x}_+)\ \left( {\rm or}\ {\mc E}^{k-1, p} \left( \llbracket{\mf x}_-\rrbracket, \llbracket{\mf x}_+\rrbracket \right) \to {\mc B}^{k, p} \left( \llbracket{\mf x}_-\rrbracket, \llbracket{\mf x}_+ \rrbracket \right) \right).
\end{align*}
Moreover, the $H$-perturbed vortex equation (\ref{equation26}) gives a section
\begin{align*}
\wt{\mc F}: \wt{\mc B}^{k, p}\left( {\mf x}_-, {\mf x}_+ \right) \to \wt{\mc E}^{k-1, p} ( {\mf x}_-, {\mf x}_+)
\end{align*}
which is ${\mc G}^{k+1, p}_0$-equivariant. So it descends to a section 
\begin{align*}
\begin{array}{c}
{\mc F} : {\mc B}^{k, p}\left( [ {\mf x}_- ] , [ {\mf x}_+ ] \right) \to {\mc E}^{k-1, p} \left( [ {\mf x}_- ], [ {\mf x}_+ ] \right),\\
 {\rm or}\\
{\mc F} : {\mc B}^{k, p}\left( \lbr {\mf x}_- \rbr , \lbr {\mf x}_+ \rbr \right) \to {\mc E}^{k-1, p} \left( \lbr {\mf x}_- \rbr, \lbr {\mf x}_+ \rbr \right).
\end{array}
\end{align*}
Then we see that $\wt{\mc M}\left( {\mf x}_-, {\mf x}_+\right)$ is the intersection of $\wt{\mc F}^{-1}(0) \subset \wt{\mc B}^{k, p} \left( {\mf x}_-, {\mf x}_+ \right)$ with smooth objects. We define
\begin{align*}
{\mc M}\left( \lbr {\mf x}_- \rbr , \lbr {\mf x}_+ \rbr \right) = {\mc F}^{-1}(0),
\end{align*}
whose elements, by the standard regularity theory about symplectic vortex equation (see \cite[Theorem 3.1]{Cieliebak_Gaio_Mundet_Salamon_2002}), all have smooth representatives. Therefore ${\mc M}\left( \lbr {\mf x}_- \rbr , \lbr {\mf x}_+ \rbr\right)$ is independent of $k, p$. 

The linearization of $\wt{\mc F}$ at $\wt{u}$ reads
\begin{multline*}
\wt{D}_{\wt{u}}(\xi, \phi, \psi):= d \wt{\mc F}_{\wt{u}} (\xi, \phi, \psi) \\
:= \left( \begin{array}{c} \nabla_{A, s}\xi + (\nabla_\xi J_t) \left( \partial_t u + X_\Psi - Y_{H_t} \right) + J_t \left( \nabla_{A, t} \xi  - \nabla_\xi Y_{H_t}\right) + X_\phi + J X_\psi \\
        \partial_s \psi+ [\Phi, \psi] - \partial_t \phi - [\Psi, \phi] + d\mu(\xi)              \end{array}\right).
\end{multline*}
Hence the linearization of ${\mc F}$ at $[\wt{u}] \in {\mc B}^{k, p} \left( {\mf x}_-, {\mf x}_+ \right)$, under the isomorphism (\ref{equation41}), is the restriction of $\wt{D}_{\wt{u}}$ to $\left({\rm ker} d_0^* \right)^\bot$. We also define the {\it augmented} linearized operator
\begin{align}\label{equation42}
{\mf D}_{\wt{u}}:= \wt{D}_{\wt{u}}\oplus d_0^*: T_{\wt{u}} \wt{\mc B}^{k, p} \left( {\mf x}_-, {\mf x}_+ \right) \to \wt{\mc E}^{k-1, p}_{\wt{u}} \left( {\mf x}_-, {\mf x}_+ \right) \oplus W^{k-1, p} \left( \Theta, {\mf g} \right).
\end{align}
It is a standard result that the Fredholm property of $d{\mc F}_{[\wt{u}]}$ is equivalent to that of ${\mf D}_{\wt{u}}$ for any representative $\wt{u}$. Hence in the remaining of the section we will study the Fredholm property of the augmented operator.

\subsection{Asymptotic behavior of ${\mf D}_{\wt{u}}$}\label{section42}

Up to an $L_0 G$-action we can choose representatives ${\mf x}_\pm = (x_\pm, f_\pm, [w_\pm])$ such that $f_\pm$ are constants $\theta_\pm \in {\mf g}$. Take any $\wt{u}= (u, \Phi, \Psi) \in \wt{\mc B}^{k, p}( {\mf x}_-, {\mf x}_+)$. For $\wt{\xi}:= (\xi, \psi, \phi) \in T_{\wt{u}} \wt{\mc B}^{k, p}({\mf x}_-, {\mf x}_+)$, define $\wt{J}(\xi, \psi, \phi) = ( J_t \xi, - \phi, \psi)$. Then
\begin{multline*}
{\mf D}_{\wt{u}} \left( \begin{array}{c} \xi\\ \psi \\ \phi \end{array} \right) = \nabla_{A, s} \left( \begin{array}{c} \xi \\ \psi \\ \phi \end{array} \right)+ \widetilde{J} \left( \begin{array}{c} \nabla_{A, t} \xi - \nabla_\xi Y_{H, t}\\ \nabla_{A, t} \psi \\ \nabla_{A, t} \phi  \end{array}\right) \\
+ \left( \begin{array}{ccc} 0 & J L_u & L_u\\ d\mu & 0 & 0 \\
                                                                            L_u^* &  0 & 0 \end{array}\right)\left(\begin{array}{c}  \xi \\ \psi \\ \phi \end{array}\right) + q(s, t) \left( \begin{array}{c} \xi \\ \psi \\ \phi \end{array} \right)\\
 = :  \left[ \left( \begin{array}{c} \nabla_s  \\ \partial_s  \\ \partial_s \end{array} \right)+ \widetilde{J} \left( \begin{array}{c} \nabla_t \\ \partial_t\\ \partial_t  \end{array} \right) + R(s, t) + q(s, t)\right]\left( \begin{array}{c}  \xi \\ \psi \\ \phi \end{array} \right) .
 \end{multline*}
Here $q(s, t)$ is a linear operator such that $\displaystyle \lim_{s\to \pm \infty} q(s, t)=0$; and 
\begin{align}
\label{equation43}
R_{\wt{x}_\pm} (t):= \lim_{s\to \pm \infty} R(s, t)= \left( \begin{array}{ccc} J_t \nabla(X_{\theta_\pm} - Y_{H_t} ) & J_t L_u & L_u \\
                                                          d\mu                                                 & 0     & -{\rm ad} \theta_\pm \\
                                                          L_u^*                                                & {\rm ad} \theta_\pm & 0  \end{array}    \right).
\end{align}
Here $L_u: {\mf g} \to u^*TM$ is the infinitesimal action along the image of $u$ and $L_u^*$ is its dual. 

\begin{prop}\label{nondegenerate} $\wt{x}:= (x, \theta) \in {\rm Zero} \wt{\mc B}_H$, then for all $t$-dependent, $G$-invariant, $\omega$-compatible almost complex structure $J_t$, the self-adjoint operator
\begin{align}\label{equation44}
\begin{array}{ccc}
L^2 \left( S^1, x^* TM \oplus {\mf g} \oplus {\mf g} \right) & \to & L^2 \left( S^1, x^* TM \oplus {\mf g} \oplus {\mf g} \right)\\
\left( \begin{array}{c} \xi \\ \psi \\ \phi \end{array} \right) & \mapsto & \left[\widetilde{J} \left( \begin{array}{c} \nabla_t \\ \partial_t \\ \partial_t \end{array} \right) + R_{\wt{x}}(t) \right] \left(  \begin{array}{c} \xi \\ \psi \\ \phi  \end{array} \right).
\end{array}
\end{align}
has zero kernel.
\end{prop}

\begin{proof} Suppose $(\xi, \psi, \phi)^T$ is in the kernel, which means
\begin{eqnarray}
J_t \nabla_t \xi + J_t \nabla_\xi \left( X_\theta - Y_{H_t} \right) + X_{\phi} + J_t X_{\psi} &= & 0, \label{equation45}\\
- {d\phi\over dt} + d\mu (\xi) - [\theta, \phi] & = & 0,\label{equation46}\\
{d \psi \over dt} + d\mu(J_t \xi) + [\theta, \psi ] & = & 0.\label{equation47}
\end{eqnarray}
Apply $d\mu\circ J_t$ to (\ref{equation45}), we get
\begin{align*}
d\mu(J_t X_{\phi} ) = d\mu ( \nabla_t \xi + \nabla_\xi ( X_\theta- Y_{H_t} ) ).
\end{align*}
Hence for any $\eta\in {\mf g}$, 
\begin{align*}
\begin{split}
&\ {d\over dt} \langle d\mu(\xi) , \eta \rangle_{\mf g} = {d\over dt} \omega( X_\eta, \xi )\\
= &\ \omega ( [ Y_{H_t}- X_\theta, X_\eta ] , \xi )+ \omega ( X_\eta, \nabla_t \xi - \nabla_\xi ( Y_{H_t}- X_\theta) ) \\
= &\ \omega( X_{[ \theta, \eta] } , \xi)+ \langle d\mu( J_t X_{\phi}), \eta \rangle_{\mf g} \\
= &\ \langle d\mu(\xi), [\theta, \eta] \rangle_{\mf g} + \langle d\mu( J_t X_{\phi}), \eta \rangle_{\mf g} \\
= &\ \langle d\mu(J_t X_{\phi})- [\theta , d\mu(\xi)],  \eta \rangle_{\mf g}.
\end{split}
\end{align*}
Therefore,
\begin{align*}
{d\over dt} d\mu(\xi)= d\mu( J_t X_{\phi}) -[\theta, d\mu(\xi)].
\end{align*}
Then by (\ref{equation46}),
\begin{align}\label{equation48}
\begin{split}
&\ d\mu(J_t X_{\phi}) \\
= &\ {d\over dt} d\mu(\xi) + [\theta, d\mu(\xi)] \\
 = &\ {d\over dt} \left( {d\phi\over dt} + [\theta, \phi] \right) + \left[ \theta, {d\phi \over dt}  +[\theta, \phi] \right]\\
 =  &\ \phi'' + 2\left[ \theta, \phi' \right]+ [\theta, [\theta, \phi ]]. 
\end{split}
\end{align}
Suppose $| \phi |$ takes its maximum at $t= t_0 \in S^1$. Then for $t\in (t_0-\epsilon, t_0 + \epsilon)$, define $\displaystyle \wt{\phi} (t)= {\rm Ad}_{e^{(t-t_0) \theta}} \phi(t)$. Then the right hand side of (\ref{equation48}) is equal to ${\rm Ad}_{e^{(t_0 - t) \theta}} \wt\phi''(t)$. Hence at $t= t_0$, 
\begin{align*}
\begin{split}
0 \geq  &\ {1\over 2} { d^2 \over dt^2} | \phi |^2 = {1\over 2} {d^2 \over dt^2} \left| \wt\phi \right|^2 = \left\langle \wt\phi'', \wt\phi \right\rangle + \left| \wt\phi' \right|_{\mf g}^2 \\
= &\ \left\langle d\mu (J_t X_{\phi(t_0)} ), \phi(t_0) \right\rangle + \left| \wt\phi'(t_0) \right|^2 = \left| X_{\phi(t_0)} \right|^2 + \left| 	\wt\phi'(t_0) \right|^2. 
\end{split}
\end{align*}
Hence $X_{\phi}\equiv 0$, which implies $\phi \equiv 0$ and by (\ref{equation46}), $\xi$ is tangent to $\mu^{-1}(0)$.

Now $x^* T\mu^{-1}(0)= E_t \oplus L_x {\mf g}$, where $E_t = (L_x {\mf g})^\bot \cap x^* T\mu^{-1}(0)$ and the orthogonal complement is taken with respect to the Riemannian metric $g_t = \omega ( \cdot, J_t \cdot)$. Then with respect to this ($G$-invariant) decomposition, write 
\begin{align*}\xi = \xi^\bot(t) + X_{\eta(t)}.
\end{align*}
Then take the $E_t$-component of (\ref{equation45}), and use the nondegeneracy assumption on $\ov{H}$ (Hypothesis \ref{hyp22}), we see that $\xi^\bot \equiv 0$. Then the only equation left is
\begin{align*}
\left\{ \begin{array}{cc} \nabla_t X_{\eta(t)} + \nabla_{X_{\eta(t)}} (X_\theta- Y_{H_t} ) + X_{\psi}& =0,\\[0.2cm]
 \psi' + [\theta, \psi] + d\mu (J_t X_\eta) & =0. \end{array} \right.
\end{align*}
The first equation is equivalent to 
\begin{align*}
\eta'+ [\theta, \eta(t)] + \psi =0.
\end{align*}
Hence 
\begin{align*}
\eta''(t)  + 2 \left[ \theta, \eta'(t) \right] +  \left[ \theta, [\theta, \eta] \right] = d\mu ( J_t X_{\eta}).
\end{align*}
This can be treated similarly as (\ref{equation48}), using the maximum principle, which shows that $\eta \equiv 0$ and hence $\psi \equiv 0$.
\end{proof}

\begin{cor}
For any $\wt{x}_\pm \in {\rm Zero} \wt{\mc B}_H$ and any $\wt{u} \in \wt{\mc B}^{k, p}\left( \wt{x}_-, \wt{x}_+ \right)$, the augmented linearized operator ${\mf D}_{\wt{u}}$ is Fredholm for any $k \geq 1$, $p \geq  2$.
\end{cor}

\begin{cor} There exists $\delta = \delta( \wt{x}_\pm)>0$ such that, for any $\wt{u}= (u, \Phi, \Psi) \in \wt{\mc M}\left( \wt{x}_-, \wt{x}_+; J, H  \right)$ and any $\wt{\xi}\in {\rm ker} {\mf D}_{\wt{u}}$, there exists $c>0$ such that
\begin{align*}
\left| \wt{\xi}(s, t) \right| \leq c e^{-\delta |s|}.
\end{align*}
In particular, if $(u, 0, \Psi) \in \wt{\mc M}_{\Theta}^b$, then $|\partial_s u|$ and $|\partial_s \Psi|$ decay exponentially.
\end{cor}

\begin{proof} The first part is standard, see for example \cite[Lemma 2.11]{Salamon_lecture}. For a solution $\wt{u}= (u, 0, \Psi)$ in temporal gauge, by the translation invariance of the equation (\ref{equation24}), we see that $\widetilde{\xi}=(\partial_s u, \beta_s=0, \beta_t= \partial_s \Psi)\in {\rm ker} d\wt{\mc F}_{\wt{u}}$. Moreover,
\begin{align}\label{equation49}
\begin{split}
 -d_0^* \widetilde{\xi}\ = & \ \partial_t \partial_s \Psi+ L_u^* (\partial_s u)+ [\Psi, \partial_s \Psi]\\
= &\ -\partial_t (\mu(u))+ d\mu( J_t \partial_s u)+ [\Psi, \partial_s \Psi] \\
= &\ - d\mu( \partial_t u)+ d\mu( J_t \partial_s u) + [\Psi, \partial_s \Psi] \\
= & \ d\mu\left( X_\Psi- Y_{H_t} \right) - [\Psi ,\mu]=0.
\end{split}
\end{align}
This implies that $\wt{\xi} \in {\rm ker} {\mf D}_{\wt{u}}$. Choose a smooth gauge transformation $g: \Theta \to G$ such that $g^*\wt{u}$ satisfies the asymptotic condition of Proposition \ref{prop36}. Then $g^* \wt{\xi} \in {\rm ker} {\mf D}_{g^*\wt{u}}$, which decays exponentially. So does $\wt\xi$.
\end{proof}

\subsection{The Conley-Zehnder indices}\label{subsection43}

In this subsection we define a grading on the set ${\rm Crit} {\mc A}_H$, which is analogous to the Conley-Zehnder index in usual Hamiltonian Floer theory, and we will call it by the same name. 

For the induced Hamiltonian $(\ov{H}_t)$ on the symplectic quotient $\ov{M}$, we have the usual Conley-Zehnder index
\begin{align*}
\ov{\sf cz}: {\rm Crit} {\mc A}_{\ov{H}} \to {\mb Z}.
\end{align*}
We prove the following theorem
\begin{thm}\label{thm44}
There exists a function
\begin{align*}
{\sf cz}: {\rm Crit} {\mc A}_H \to {\mb Z}
\end{align*}
satisfying the following properties
\begin{enumerate}
\item For the embedding $ \iota: {\rm Crit} {\mc A}_{\ov{H}} \to {\rm Crit} {\mc A}_H$, we have
\begin{align*}
{\sf cz} \circ \iota = \ov{\sf cz};
\end{align*}

\item For any $B \in \Gamma$ and $\llbracket {\mf x} \rrbracket \in {\rm Crit} {\mc A}_H$ we have
\begin{align*}
{\sf cz}\left( B \# \llbracket {\mf x} \rrbracket \right) = {\sf cz} \left( \llbracket {\mf x} \rrbracket \right) - 2 c_1^G (B). 
\end{align*}

\item For $ \llbracket {\mf x}_\pm \rrbracket \in {\rm Crit} {\mc A}_H$ and $[\wt{u}] \in {\mc B}^{k, p} \left( \llbracket{\mf x}_-\rrbracket, \llbracket{\mf x}_+ \rrbracket\right)$, we have
\begin{align*}
{\rm ind} \left( d{\mc F}_{[\wt{u}]}\right) = {\sf cz} \left( \llbracket {\mf x}_- \rrbracket \right) - {\sf cz} \left( \llbracket {\mf x}_+ \rrbracket \right).
\end{align*}
\end{enumerate}
\end{thm}

To prove Theorem \ref{thm44}, we first review the notion of Conley-Zehnder index in Hamiltonian Floer homology. Let $A: [0, 1]\to {\rm Sp}(2n)$ be a continuous path of symplectic matrices such that 
\begin{align*}
A(0) = {\rm I}_{2n},\  {\rm det} \left( A(1)- {\rm I}_{2n} \right) \neq 0.
\end{align*}
We can associate an integer ${\sf cz}(A)$, the Conley-Zehnder index, to $A$. We list some of its properties below which we will use here (see for example \cite{Robbin_Salamon}).
\begin{enumerate}
\item For any path $B: [0, 1] \to {\rm Sp}(2n)$, we have ${\sf cz}(B A B^{-1}) = {\sf cz}(A)$;

\item ${\sf cz}$ is homotopy invariant;

\item If for $t>0$, $A(t)$ has no eigenvalue on the unit circle, then ${\sf cz}(A) = 0$;

\item If $A_i: [0, 1] \to {\rm Sp}(2n_i)$ for $n=1, 2$, then ${\sf cz} (A_1 \oplus A_2) = {\sf cz}(A_1) + {\sf cz}(A_2)$;

\item If $\Phi: [0, 1] \to {\rm Sp}(2n)$ is a loop with $\Phi(0) = \Phi(1) = {\rm Id}$, then 
\begin{align*}
{\sf cz}(\Phi A) = {\sf cz}(A) + 2 \mu_M (\Phi)
\end{align*}
where $\mu_M(\Phi)$ is the Maslov index of the loop $\Phi$.
\end{enumerate}

With this algebraic notion, in the usual Hamiltonian Floer theory one can define the Conley-Zehnder indices for nondegenerate Hamiltonian periodic orbits. In our case, the induced Hamiltonian $\ov{H}_t: \ov{M}\to {\mb R}$ has the usual Conley-Zehnder index
\begin{align*}
\ov{\sf cz}: {\rm Crit} {\mc A}_{\ov{H}} \to {\mb Z}.
\end{align*}

Then, for each ${\mf x} = ( x, f, [w]) \in {\rm Crit} \wt{\mc A}_H$, the homotopy class of extensions $[w]$ induces a homotopy class of trivializations of $x^* TM$ over $S^1$. With respect to this class of trivialization, the operator (\ref{equation44}) is equivalent to an operator $J_0 \partial_t + A(t)$, which defines a symplectic path. We define the Conley-Zehnder index of ${\mf x}$ to be the Conley-Zehnder index of this symplectic path. By the second and fifth axioms listed above, this index induces a well-defined function
\begin{align*}
{\sf cz} : {\rm Crit } {\mc A}_H\to {\mb Z}
\end{align*}
which satisfies (2) and (3) of Theorem \ref{thm44}.

Now we prove (1). For any contractible periodic orbits $\ov{x}: S^1 \to \ov{M}$ of $Y_{\ov{H}_t}$ and any extension $\ov{w}: {\mb D} \to \ov{M}$ of $\ov{x}$, we can lift the pair $(\ov{x}, \ov{w})$ to a tuple $[{\mf x}] = [x, f, [w\rrbracket \in {\rm Crit} \wt{\mc A}_H/ L_0 G$ as (\ref{equation23}). Since the Conley-Zehnder index is homotopy invariant, and the space of $G$-invariant $\omega$-compatible almost complex structures is connected, we will compute the Conley-Zehnder index using a special type of almost complex structures, and modify the Hamiltonian $H$. 

Starting with any almost complex structure $\ov{J}$ on $\ov{M}$ and a $G$-connection on $\mu^{-1}(0)\to \ov{M}$, $\ov{J}$ lifts to the horizontal distribution defined by the connection. On the other hand, the biinvariant metric on ${\mf g}$ gives an identification ${\mf g}\simeq {\mf g}^*$. We denote by $\eta^* \in {\mf g}^*$ the metric dual of $\eta\in {\mf g}$. Recall that we have a symplectomorphism
\begin{align*}
\mu^{-1} \left( {\mf g}_\epsilon^* \right) \simeq \mu^{-1}(0) \times {\mf g}_\epsilon^*.
\end{align*}
For $\eta \in {\mf g}$, we define $J X_\eta= \eta^* \in {\mf g}^*$, as a vector field on $\mu^{-1}(0) \times {\mf g}^*$. Then this gives a $G$-invariant almost complex structure on $TM|_{\mu^{-1}(0)}$. Then we pullback $J$ by the projection $\mu^{-1}(0) \times {\mf g}^*_\epsilon \to \mu^{-1}(0)$ and denote the pullback by $J$. 

We also modify $H_t$ by requiring that $H_t(x, \eta) = H_t(x)$ for $(x, \eta) \in \mu^{-1}(0) \times {\mf g}_\epsilon^*$. Then the modified $H_t$ can be continuously deformed to the original one, and it doesn't change $\ov{H}$ hence doesn't change ${\rm Crit} {\mc A}_{\ov{H}}$. Moreover, it is easy to check that for the modified pair $(J, H)$, 
\begin{align}\label{equation410}
\left( {\mc L}_{Y_{H_t}} J \right) X_\eta = \left[ {\mc L}_{Y_{H_t}}, J X_\eta \right] = 0.
\end{align}

Now for any $(\ov{x}, \ov{w}) \in {\rm Crit} {\mc A}_{\ov{H}}$, we lift it to $(x, f, [w]) \in {\rm Crit} \wt{\mc A}_H$ with $w:{\mb D}\to \mu^{-1}(0)$ and $f$ being a constant $\theta \in {\mf g}$. Then any symplectic trivialization of $\wt{x}^* T\ov{M} \to S^1$ induces a symplectic trivialization
\begin{align}\label{equation411}
\phi: x^* TM  \simeq S^1 \times \left[ {\mb R}^{2n-2k} \oplus \left( {\mf g}\oplus {\mf g}\right)\right]
\end{align}
such that $\phi( X_\eta, J X_{\zeta} ) = (0, \eta, \zeta)$. Then we see, with respect to $\phi$, the operator (\ref{equation44}) restricted to ${\mf g}^4$ is
\begin{align*}
\left(\begin{array}{c} \eta_1\\ \psi \\ \eta_2 \\ \phi \end{array} \right) \mapsto \wt{J} {d\over dt} \left( \begin{array}{c} \eta_1 \\ \psi \\ \eta_2 \\ \phi \end{array}  \right) + \left[ \begin{array}{c}  \phi -[\theta, \eta_2] \\ \eta_2 - [\theta, \phi]  \\ \psi + [\theta, \eta_1] \\ \eta_1 + [\theta, \psi] \end{array}\right] =: \left( \wt{J} {d\over dt} + S  \right) 
\left(\begin{array}{c} \eta_1\\ \psi \\ \eta_2 \\ \phi \end{array} \right).
\end{align*}
Here we used the property (\ref{equation410}) and 
\begin{align*}
\wt{J}:= \left[\begin{array}{cc} 0 & -{\rm Id}_{{\mf g}\oplus {\mf g}} \\
                                   {\rm Id}_{{\mf g} \oplus {\mf g}} & 0  																	  
\end{array} \right],\ S=\left[ \begin{array}{cc} 0  & {\rm Id}_{{\mf g}\oplus {\mf g} } - {\rm ad}_\theta \\
													      {\rm Id}_{{\mf g}\oplus {\mf g}} + {\rm ad}_\theta & 0\end{array}\right]
	\end{align*}
Moreover, the operator (\ref{equation44}) respect the decomposition in (\ref{equation411}). Hence by the fourth axiom of Conley-Zehnder indices we listed above, we have
\begin{align*}
{\sf cz} \left( x, \theta, [w] \right) = \ov{\sf cz} ( \ov{x}, \ov{w} ) + {\sf cz}\left(  e^{\wt{J} S t} \right).
\end{align*}
As we have shown in the proof of Proposition \ref{nondegenerate} that for any $\theta$ the operator (\ref{equation44}) is an isomorphism, we can deform $\theta$ to zero and compute instead ${\sf cz}(e^{\wt{J}S_0 t})$ for
\begin{align*}
S_0=\left[ \begin{array}{cc} 0 &  {\rm Id}_{{\mf g} \oplus {\mf g} }\\
													 {\rm Id}_{{\mf g}\oplus {\mf g} } & 0   \end{array}\right],
\end{align*}
thanks to the homotopy invariance property. Then we see that
\begin{align*}
e^{\wt{J} S_0 t} = \left[ \begin{array}{cc} e^{-t} & 0 \\ 0 & e^t  \end{array}\right]
\end{align*}
which has no eigenvalue on the unit circle for $t > 0$. By the third axiom of the Conley-Zehnder index, ${\sf cz}( e^{\wt{J}S_0 t}) = 0$.

\section{Compactness of the moduli space}\label{section5}

In this section we consider the compactification of moduli space of gauge equivalence classes of connecting orbits. Because of the aspherical assumption on $(M, \omega)$ we don't have sphere bubbles and the noncompactness only comes from the breaking of trajectories. Therefore the situation is very similar to the case of ordinary Morse homology theory. Compactness results of the vortex equation has been considered in many places in the literature (see \cite{Mundet_2003}, \cite{Frauenfelder_thesis}, \cite{Mundet_Tian_2009}, \cite{Guangbo_compactness}, especially \cite[Section 3]{Cieliebak_Gaio_Mundet_Salamon_2002} for technical details) and the related techniques are standard.

\subsection{Moduli space of stable connecting orbits and its topology}

Let's fix a pair $\llbracket {\mf x}_\pm \rrbracket \in {\rm Crit} {\mc A}_H$. Denote by 
\begin{align*}
\widehat{\mc M}( \llbracket{\mf x}_-\rrbracket, \llbracket{\mf x}_+\rrbracket ) := \widehat{\mc M}( \llbracket{\mf x}_-\rrbracket, \llbracket{\mf x}_+\rrbracket; J, H) = {\mc M}( \lbr {\mf x}_- \rbr, \lbr {\mf x}_+ \rbr; J, H)/ {\mb R}
\end{align*}
the quotient of the moduli space by the translation in the $s$-direction. We denote by $\{\wt{u}\}$ the ${\mb R}$-orbit in $\widehat{\mc M}(\llbracket{\mf x}_-\rrbracket, \llbracket{\mf x}_+\rrbracket)$ of $[\wt{u}] \in {\mc M}( \llbracket{\mf x}_-\rrbracket, \llbracket{\mf x}_+ \rrbracket; J, H)$ and call it a {\it trajectory} from $\llbracket{\mf x}_- \rrbracket$ to $\llbracket {\mf x}_+ \rrbracket$.

\begin{defn}
A broken trajectory from $\llbracket {\mf x}_- \rrbracket $ to $\llbracket {\mf x}_+ \rrbracket$ is a collection
\begin{align*}
\wt{\mf u}:= \left( \left\{ \wt{u}^{(\alpha)} \right\} \right)_{\alpha=1, \ldots, m}:= \left( \left\{ u^{(\alpha)}, \Phi^{(\alpha)}, \Psi^{(\alpha)} \right\}  \right)_{\alpha=1, \ldots, m}
\end{align*}
where for each $\alpha$, $\left\{ \wt{u}^{(\alpha)} \right\} \in \widehat{\mc M} \left( \llbracket {\mf x}_{\alpha-1} \rrbracket, \llbracket {\mf x}_{\alpha} \rrbracket \right)$ and $E(\wt{u}^{(\alpha)}) \neq 0$. Here $\left\{ \lbr {\mf x}_\alpha \rbr  \right\}_{\alpha=0, \ldots, m}$ is a sequence of critical points of ${\mc A}_H$ and
\begin{align*}
\llbracket {\mf x}_0 \rrbracket  = \llbracket {\mf x}_-  \rrbracket,\  \llbracket {\mf x}_m  \rrbracket =  \llbracket {\mf x}_+  \rrbracket.
\end{align*}
We regard the domain of $\wt{\mf u}$ as the disjoint union $\cup_{\alpha=1}^m \Theta$ and let $\Theta^{(\alpha)} \subset \cup_{\alpha=1}^m \Theta$ the $\alpha$-th cylinder.
\end{defn}
We denote by
\begin{align*}
\ov{\mc M} \left( \llbracket {\mf x}_- \rrbracket, \llbracket {\mf x}_+ \rrbracket \right)
\end{align*}
the space of all broken trajectories from $\llbracket {\mf x}_- \rrbracket$ to $\llbracket {\mf x}_+ \rrbracket$.	We have a natural inclusion
\begin{align*}
\widehat{\mc M}\left( \llbracket {\mf x}_- \rrbracket, \llbracket {\mf x}_+ \rrbracket \right) \hookrightarrow \ov{{\mc M}}\left( \llbracket {\mf x}_- \rrbracket, \llbracket {\mf x}_+ \rrbracket \right).
\end{align*}

\begin{defn}\label{defn52} We say that a sequence $\left\{\wt{u}_i\right\} = \left\{ u_i, \Phi_i, \Psi_i\right\} \in \widehat{\mc M}\left( \llbracket {\mf x}_- \rrbracket, \llbracket {\mf x}_+ \rrbracket \right)$ from $\llbracket{\mf x}_-\rrbracket$ to $\llbracket {\mf x}_+\rrbracket$ converges to a broken trajectory
\begin{align*}
\wt{\mf u}:= \left( \left\{ \wt{u}^{(\alpha)}\right\} \right)_{\alpha=1, \ldots, m}
\end{align*}
if: for each $i$, there exists sequences of numbers $s_i^{(1)}< s_i^{(2)}< \cdots < s_i^{(m)}$ and gauge transformations $g_i^{(\alpha)} \in {\mc G}_\Theta $ such that for each $\alpha$, the sequence
\begin{align*}
\left( g_i^{(\alpha)} \right)^* \left( s_i^{(\alpha)} \right)^* (u_i, \Phi_i, \Psi_i)
\end{align*}
converges to $\left(u^{(\alpha)}, \Phi^{(\alpha)}, \Psi^{(\alpha)}\right)$ uniformly on any compact subset of $\Theta$ and such that for any sequence of $(s_i, t_i)$ with 
\begin{align*}
\lim_{i \to \infty} |s_i- s_i^{(\alpha)} |= \infty,\ \forall \alpha
\end{align*}
we have
\begin{align*}
\lim_{i\to \infty} e ( \wt{u}_i ) (s_i, t_i) = 0.
\end{align*}
Here $e(\wt{u}_i)$ is the energy density function.

It is easy to see that this convergence is well-defined and independent of the choices of representatives of the trajectories. We can also extend this notion to sequences of broken trajectories. We omit that for simplicity.
\end{defn}

The main theorem of this section is
\begin{thm}\label{thm53} The space $\ov{\mc M} \left( \llbracket {\mf x}_- \rrbracket, \llbracket {\mf x}_+ \rrbracket \right)$ is compact with respect to the topology defined in Definition \ref{defn52}.
\end{thm}

The proof is provided at the end of this section. Indeed the proof is routine and it has been carried out in many places in the literature for general symplectic vortex equations, for example \cite{Mundet_2003}, \cite{Cieliebak_Gaio_Mundet_Salamon_2002}, \cite{Mundet_Tian_2009}, \cite{Guangbo_compactness}. One can also see \cite[Theorem 4.12]{Frauenfelder_thesis} in a much closer case. Since bubbling is ruled out, the proof is almost the same as that for finite dimensional Morse theory, while the gauge symmetry is the only additional ingredient.

\subsection{Local compactness with uniform bounded energy density}

For any compact subset $K \subset \Theta$, consider a sequence of solutions $\wt{u}_i:= \left( u_i, \Phi_i, \Psi_i \right)$ such that the image of $u_i$ is contained in the compact subset $K_H \subset M$ and such that
\begin{align*}
\limsup_{i \to \infty} E(\wt{u}_i) < \infty.
\end{align*}
We have the following local compactness result, which can be compared with \cite[Page 23, Step 1]{Frauenfelder_thesis}.
\begin{prop}
There exists a subsequence (still indexed by $i$), a sequence of smooth gauge transformation $g_i: K \to G$ and a solution $\wt{u}_\infty= (u_\infty, \Phi_\infty, \Psi_\infty ): K \to M \times {\mf g} \times {\mf g}$ to (\ref{equation24}) on $K$, such that the sequence $g_i^* \widetilde{u}_i$ converges to $\wt{u}_\infty$ uniformly with all derivatives on $K$.
\end{prop}

\begin{proof} 
By the fact that $u_i$ is contained in the compact subset $K_H$, and the assumption that there exists no nontrivial holomorphic spheres in $M$, we have
\begin{align*}
\sup_{z\in K, i} e_{\wt{u}_i}(z) < \infty.
\end{align*}
Then this proposition can be proved in the standard way, such as in \cite[Section 3]{Cieliebak_Gaio_Mundet_Salamon_2002} or \cite{Mundet_Tian_2009}, using Uhlenbeck's compactness theorem. \end{proof}


\subsection{Energy quantization}

To prove the compactness of the moduli space, we need the following energy quantization property. It uses a similar argument as in \cite[Theorem 3.3]{Hofer_Salamon}. In Frauenfelder's Lagrangian Floer setting, when $H\equiv 0$, he proved the energy quantization \cite[Lemma 4.11]{Frauenfelder_thesis} by using an {\it a priori} estimate and the isoperimetric inequality.

\begin{prop}\label{prop55}
There exists $\epsilon_0:= \epsilon_0(J, H)>0$, such that for any connecting orbit $\wt{u} \in \wt{\mc M}_\Theta^b$ with positive energy, we have $E(\wt{u}) \geq \epsilon_0$.
\end{prop}

\begin{proof} Suppose it is not true. Then there exists a sequence of connecting orbits, represented by solutions in temporal gauge $\wt{v}_i:=(v_i, 0, \Psi_i)\in \wt{\mc M}_\Theta^b$, such that
\begin{align*}
E(\wt{v}_i) > 0,\ \lim_{i \to \infty} E(\wt{v}_i) = 0.
\end{align*} 
We first know that there is a compact subset $K_H \subset M$ such that for every $i$, the image $v_i(\Theta)$ is contained in $K_H$. Then we must have
\begin{align*}
\lim_{i \to \infty} \sup_{\Theta} \left( \left| \partial_s v_i \right| + \left| \mu(v_i) \right| \right)= 0.
\end{align*}
Indeed, if the equality doesn't hold, then we can find a subsequence which either bubbles off a nonconstant holomorphic sphere at some point $z\in \Theta$ (if the above limit is $\infty$), or (after a sequence of proper translation in $s$-direction) converges to a solution (with positive energy) on compact subsets (if the above limit is positive and finite). Either case contradicts the assumption. Therefore we conclude that for any $\epsilon>0$, the image of $v_i$ lies in $U_\epsilon:= \mu^{-1} ({\mf g}_\epsilon^* )$ for $i$ sufficiently large.

Recall that we have projections $\pi_\mu: U_\epsilon \to \mu^{-1}(0)$ and $\ov{\pi}: U_\epsilon \to\ov{M}$. Then for all large $i$ and any $s$, $\ov\pi (v_i (s, \cdot) )$ is $C^0$-close to a periodic orbit of $\ov{H}$ in $\ov{M}$. Since those orbits are discrete, we may fix one such orbit $\ov{\gamma} \in {\rm Zero} {\mc B}_{\ov{H}}$ and choose a subsequence (still indexed by $i$) such that 
\begin{align*}
\limsup_{i \to \infty} \sup_{(s, t) \in \Theta} d \left( \ov\pi (v_i(s, t)), \ov{\gamma}(t) \right) = 0.
\end{align*}
Then, use a fixed Riemannian metric on $\ov{M}$ with exponential map $\ov\exp$, we can write 
\begin{align*}
\ov\pi (v_i(s, t)) = \ov\exp_{\ov{\gamma}(t)} \ov{\xi}_i(s, t)
\end{align*}
where $\ov\xi_i\in \Gamma \left( S^1, \ov\gamma^* T\ov{M} \right)$. Let $B_\epsilon( \ov\gamma^* T\ov{M})$ be the $\epsilon$-disk of $\ov\gamma^* T\ov{M}$. Then $\ov\exp_{\ov\gamma}$ pulls back $\mu^{-1}(0) \to \ov{M}$ to a $G$-bundle $Q \to B_\epsilon(\ov\gamma^* T\ov{M})$, together with a bundle map $\wt\gamma: Q \to \mu^{-1}(0)$. We can trivialize $Q$ by some
\begin{align*}
\phi: Q \to G \times B_\epsilon(\ov\gamma^* T\ov{M}).
\end{align*}

Now we take a lift $\wt{x}:= (x, f) \in {\rm Zero}\wt{\mc B}_H$ of $\ov\gamma$. Then we can write
\begin{align*}
\phi\left(\wt\gamma^{-1}(x(t)) \right) = (g_0(t), \ov\gamma(t)).
\end{align*}
On the other hand, we write
\begin{align*}
\phi\left( \wt\gamma^{-1} \pi_\mu (v_i(s, t)) \right) = (g_i(s, t), \ov\xi_i(s, t)).
\end{align*}
Take the gauge transformation $\wt{g}_i(s, t) = g_i (s, t) g_0(t)^{-1}$. Then write
\begin{align*}
\wt{v}_i':= (v_i', \Phi_i', \Psi_i') := \wt{g}_i^* \wt{v}_i.
\end{align*}
Then by the exponential convergence of $v_i$ as $s \to \pm\infty$, we see that $\partial_s g_i(s, t)$ decays exponentially and hence $\Phi_i'$ converges to zero as $s \to \pm\infty$. On the other hand, we see that 
\begin{align*}
\phi\left( \wt\gamma^{-1} \pi_\mu(v_i'(s, t)) \right) = ( g_0(t), \wt\xi_i(s, t)).
\end{align*}
Therefore
\begin{align*}
\wt{v}_i' \in \wt{\mc M}(\wt{x}, \wt{x}; J, H).
\end{align*}
But it is also easy to see that the homotopy class of $\wt{v}_i'$ is trivial. Because the energy of connecting orbits only depends on its homotopy class, the energy of $\wt{v}_i'$, and hence that of $\wt{v}_i$, is equal to zero. It contradicts with the hypothesis.
\end{proof}

\subsection{Proof of Theorem \ref{thm53}}

It suffices to prove, without essential loss of generality, that for any sequence $[\wt{u}_i] \in {\mc M}\left( \llbracket{\mf x}_- \rrbracket, \llbracket{\mf x}_+ \rrbracket; J, H \right)$ represented by unbroken connecting orbits $(u_i, \Phi_i, \Psi_i ) \in \wt{\mc M} \left( {\mf x}_-,  {\mf x}_+ \right)$, there exists a convergent subsequence. By the assumption that there exists no nontrivial holomorphic sphere in $M$, we have
\begin{align*}
\sup_{i, \Theta} \left|\partial_s u_i + X_{\Phi_i}(u_i) \right| < +\infty.
\end{align*}
Then the limit (broken) connecting orbits can be constructed by induction and the energy quantization property (Proposition \ref{prop55}) guarantees that the induction stops at finite time. The details are standard and left to the reader.

\section{Transversality by perturbing the almost complex structure}\label{section6}

In this section, we treat the transversality problem about our moduli space of connecting orbits. Our method is very much close to the one used by Floer-Hofer-Salamon \cite{Floer_Hofer_Salamon}, which is to use a generic choice of certain class of $t$-dependent almost complex structure. We remark that our method applies to the case when the symplectic quotient $\ov{M}$ is not semi-positive, for which traditionally people can only use the virtual technique to treat the transversality. This is indeed the benefits of using the gauged $\sigma$-model technique proposed in \cite{Cieliebak_Gaio_Salamon_2000}. 

The key point is to prove the existence of $G$-regular points (see Definition \ref{defn611}) for any nontrivial connecting orbit. In \cite[Section 4.4]{Frauenfelder_thesis} Frauenfelder gave an argument in the context of his Lagrangian Floer theory, which is different in some places from the original Floer-Hofer-Salamon approach. In his argument, it suffices to look at the asymptotics of the solution near infinity, which is prescribed by an eigenvector of the Hessian of the action functional at the critical point (by \cite[Theorem B]{Robbin_Salamon_strip}). However, he didn't show that the eigenvector necessarily has nonzero projection onto the symplectic quotient. So {\it a priori}, there could be nontrivial connecting orbits which are contained entirely inside a $G_{\mb C}$-orbit. These objects can neither be excluded, nor be made transverse by his argument. 

In this section we fill this gap in our context of Hamiltonian Floer theory. More precisely, we start with a nondegenerate Hamiltonian $(\ov{H}_t)$ on the symplectic quotient $\ov{M}$. There are plenty of freedom to choose a $G$-invariant lift $(H_t)$, and those ones for which all the nonzero modes of the Hessian of ${\mc A}_H$ have nonzero projections onto $\ov{M}$ are called admissible lifts. We believe that a generic lift is admissible but instead, we give an {\it ad hoc} construction of an admissible lift in Lemma \ref{lemma66}. Then Frauenfelder's argument applies to our situation. 

We first recall the following theorem on the ``normal form'' of symplectic structure near $\mu^{-1}(0)$.

\begin{prop}\label{prop61}(\cite{Gotay}, \cite{Guillemin_Sternberg_1984}, \cite{Marle})
Suppose $0$ is a regular value of $\mu$ and $G$ acts freely on $\mu^{-1}(0)$. Then there exists a neighborhood $U$ of $\mu^{-1}(0)$, and a diffeomorphism
\begin{align}\label{equation61}
i_U: \mu^{-1}(0) \times {\mf g}_\epsilon^* \to  U
\end{align}
where ${\mf g}_\epsilon^*$ is the $\epsilon$-ball of ${\mf g}^*$ centered at the origin, satisfying the following conditions.
\begin{enumerate}
\item $\mu\circ i_U$ is equal to the projection $\theta: \mu^{-1}(0) \times {\mf g}_\epsilon^* \to {\mf g}_\epsilon^*$;

\item There is a connection $\tau \in \Omega^1(U, {\mf g})$ of the $G$-bundle $\mu^{-1}(0) \to \ov{M}$ such that 
\begin{align*}
i_U^* \omega = \pi^* \ov\omega + d( \tau(\theta)).
\end{align*}
Here $\pi: \mu^{-1}(0) \times {\mf g}_\epsilon^* \to \ov{M}$ is the natural projection and $\ov{\omega}$ is the symplectic form of $\ov{M}$.

\item The $G$-action on $U$ is given by
\begin{align*}
g(x, \zeta) = \left( gx, {\rm Ad}_{g^{-1}}^* \zeta \right).
\end{align*}
\end{enumerate}
\end{prop}

Throughout this appendix, we fix $U$ and identify $U$ with $\mu^{-1}(0) \times {\mf g}_\epsilon^*$ by $i_U$. There are two complementary distributions, $\ov{TM}$ and ${\mf g}_{\mb C}$, where the former is the horizontal distribution of $\tau$ on $\mu^{-1}(0)$ pulled back to $U$, and the latter is the span of $X_\xi$ for all $\xi\in {\mf g}$ and $T{\mf g}_\epsilon^*$. We identify ${\mf g}$ with ${\mf g}^*$ via a pre-chosen biinvariant metric. So $T{\mf g}_\epsilon^*$ is isomorphic to ${\mf g}$ and the distribution ${\mf g}_{\mb C}$ is isomorphic to ${\mf g} \oplus {\bm i} {\mf g}$. We denote by $\pi_1: TM|_U \to \ov{TM}$, $\pi_2: TM|_U \to {\mf g}_{\mb C}$ the projections with respect to the splitting
\begin{align}\label{equation62}
TM|_U \simeq \ov{TM} \oplus {\mf g}_{\mb C}.
\end{align}

\subsection{Admissible almost complex structures and admissible Hamiltonians}

Let $(\ov{H}_t)$ be a Hamiltonian with only nondegenerate 1-periodic orbits, which are $\ov{y}_1, \ldots, \ov{y}_s: S^1 \to \ov{M}$. 

\begin{defn}\label{defn62}
A smooth $S^1$-family of almost complex structures $J = (J_t)$ on $M$ is called {\bf admissible} (with respect to $i_U:\mu^{-1}(0) \times {\mf g}_\epsilon^* \to U$ and $(\ov{H}_t)$) if it satisfies

\begin{enumerate}
\item For each $t \in S^1$, $J_t$ is $G$-invariant and $\omega$-compatible.

\item For each $t \in S^1$, $J_t |_{M \setminus U} = {\mf J}$, where ${\mf J}$ is the one appeared in Hypothesis \ref{hyp23}.

\item With respect to the decomposition (\ref{equation62}), we can write
\begin{align*}
J_t|_U = \left( \begin{array}{cc} J_t^{(11)} & J_t^{(12)} \\ J_t^{(21)} & J_t^{(22)}
\end{array} \right).
\end{align*}
Then we require that for any $i\in \{1, \ldots, s\}$, any $t \in S^1$, and any $x_i(t) \in \mu^{-1}(0)$ with $\pi(x_i(t))= \ov{y}_i(t)$ (where $\pi: \mu^{-1}(0) \to \ov{M}$ is the projection), we have
\begin{align*}
\begin{split}
&\ J_t^{(12)}(x_i(t)) = 0,\ J_t^{(21)}(x_i(t)) = 0,\ J_t^{(22)}(x_i(t)) = {\bm i},\\
&\ \dot{J}_t^{(12)}(x_i(t)) = 0,\ \dot{J}_t^{(21)}(x_i(t)) = 0,\ \dot{J}_t^{(22)}(x_i(t)) = 0.
\end{split}
\end{align*}
\end{enumerate}

For any $G$-invariant lift $H = (H_t)$ of $(\ov{H}_t)$, we denote by $\wt{\mc J}_H $ the space of smooth $S^1$-families of admissible almost complex structures (with respect to $(\ov{H}_t)$), and define $\wt{\mc J}^l_H$ the corresponding objects in the $C^l$-category, for $l \geq 1$.
\end{defn}

\begin{lemma}\label{lemma63}
For $l \geq 1$, the space $\wt{\mc J}^l_H$ is a smooth Banach manifold. For any $J= (J_t)_{t\in S^1} \in \wt{\mc J}^l_H$, the tangent space $T_{J} \wt{\mc J}^l_H$ is naturally identified with the space of $G$-invariant sections $E: S^1 \times M \to {\rm End}_{\mb R}TM$ (of class $C^l$), supported in the closure of $S^1 \times U$, and for each $t \in S^1$ satisfying

\begin{enumerate}
\item $J_t E_t + E_t J_t = 0$;

\item $\omega( \cdot, E_t \cdot)$ is a symmetric tensor;

\item With respect to the decomposition (\ref{equation62}), if we write $E_t$ as
\begin{align*}
E_t= \left( \begin{array}{cc} E^{(11)}_t & E^{(12)}_t \\ E^{(21)}_t & E_t^{(22)} \end{array} \right),
\end{align*}
then for $i \in \{1, \ldots, s\}$ and any $x_i(t) \in \pi^{-1}(\ov{y}_i(t))$, we have
\begin{align*}
E_t^{(j_1 j_2)}(x_i(t)) = 0,\ \dot{E}_t^{(j_1 j_2)}(x_i(t)) = 0,\ (j_1 j_2) \neq (11).
\end{align*}
\end{enumerate}
\end{lemma}

\begin{rem}
It is natural to think about the space of almost complex structures which respect the splitting (\ref{equation62}). However, this space is not big enough to ``overcome'' all possible obstructions.
\end{rem}

For any lift $(H_t)$ and each $i \in \{1, \ldots, s\}$, we can choose $(x_i, \eta_i)\in {\rm Zero}\wt{\mc B}_H$ so that $\pi\circ x_i = \ov{y}_i$. The Hessian of the action functional $\wt{\mc A}_H$ (with respect to the metric induced from $\omega$ and $(J_t)$) at $\wt{x} = (x, \eta)\in {\rm Zero}\wt{\mc B}_H$ reads
\begin{align}\label{equation63}
\begin{array}{cccc}
{\mc H}_{\wt{x}}: & \Gamma\left( x^* TM \oplus {\mf g} \right) & \to & \Gamma \left( x^* TM \oplus {\mf g} \right)   \\
           & ( \xi, \beta) & \mapsto & \left( J_t \left( \nabla_t \xi + \nabla_\xi X_\eta - \nabla_\xi Y_{H_t} + X_\beta \right) , d\mu \cdot \xi \right).
\end{array}
\end{align}
It is a self-adjoint unbounded operator on $L^2 \left( x^* TM \oplus {\mf g}\right)$, where the connection $\nabla$ and the $L^2$-metric are induced from the family of metrics $g_t := \omega( \cdot, J_t \cdot)$.

\begin{defn}\label{defn65}
$H= (H_t) \in C_c^\infty (S^1 \times M)^G$ is called {\bf admissible} if
\begin{enumerate}
\item All 1-periodic orbits of $\ov{H}_t$ in $\ov{M}$ are nondegenerate;

\item For every $\wt{x}:=(x, \eta) \in {\rm Zero}{\mc B}_H$, every admissible $(J_t) \in \wt{\mc J}^l_{H}$ and every eigenvector $\wt\xi:= (\xi, \beta)$ of ${\mc H}_{\wt{x}}$ corresponding to a nonzero eigenvalue, $\pi_1 \left( \xi(t)\right)$ doesn't vanish identically.
\end{enumerate}
If $(\ov{H}_t)$ is nondegenerate, then $(H_t)$ is called an admissible lift of $(\ov{H}_t)$.
\end{defn}

As we have mentioned at the beginning of this section, we prove the existence of admissible lifts. The way is to modify the 2-jet of $H_t$ along any $(x, \eta) \in {\rm Zero} \wt{\mc B}_H$. 

\begin{lemma}\label{lemma66}
If $(\ov{H}_t)$ is nondegenerate, then it has an admissible lift $(H_t)$.
\end{lemma}

\begin{proof}
We give an explicit construction of an admissible $H_t$. Let $\ov{y}_1, \ldots, \ov{y}_s$ be all 1-periodic orbits of $\ov{H}_t$. Let $\epsilon>0$ be small enough such that it is smaller than the injectivity radius of $\ov{M}$ and such that for any two different 1-periodic orbits $\ov{y}_1, \ov{y}_2: S^1 \to \ov{M}$ of $\left( \ov{H}_t \right)$, $B_\epsilon(\ov{y}_1(t)) \cap B_\epsilon(\ov{y}_2(t)) = \emptyset$ for all $ t\in S^1$. For $i = 1, \ldots, s$, choose a smooth family of cut-off functions $\ov\rho^{(i)}_t: \ov{M} \to [0,1]$ supported in the $\epsilon$-neighborhood of $\ov{y}_i(t)$. They lift to $G$-invariant functions $\ov\rho^{(i)}_t: \mu^{-1}(0) \to [0,1]$.

Let $W_{\ov{H}_t} \in \Gamma( T\mu^{-1}(0))$ be the horizontal lift of $\ov{Y}_{H_t}$ with respect to the $G$-connection $\tau$ (see Proposition \ref{prop61}). For each $i \in \{1, \ldots, s\}$, choose a lift $z_i \in \mu^{-1}(0)$ of $\ov{y}_i(0)$. Let $\gamma_i (t)$ be the integral curve of $W_{\ov{H}_t}$ starting from $z_i$. Then $\gamma_i(1) = g_i z_i = g_i \gamma_i(0)$ for some $g_i \in G$. Then choose $\xi_i\in {\mf g}$ such that 
\begin{align}\label{equation64}
\exp \xi_i = g_i
\end{align}
Define $x_i: S^1 \to \mu^{-1}(0)$ by $x_i (t) = \exp ( - t \xi_i) \gamma_i(t)$, which satisfies
\begin{align*}
x_i'(t) = - X_{\xi_i}(x_i (t)) + W_{ \ov{H}_t}( x_i (t)).
\end{align*}

Suppose ${\rm dim} \ov{M} = 2k$. Choose a trivialization $\phi_i: S^1 \times {\mb R}^{2k} \to \ov{y}_i^* T\ov{M} $, which induces via the horizontal lift a trivialization $\phi_i: S^1 \times {\mb R}^{2k} \to x_i^* \ov{TM} $. Let $B_\epsilon^{2k}\subset {\mb R}^{2k}$ be the $\epsilon$-ball of ${\mb R}^{2k}$ centered at the origin. Then
\begin{align*}
\begin{array}{ccc}
G \times B_\epsilon^{2k} & \to & \mu^{-1}(0)\\
(g, {\bm v}) & \mapsto & g \exp_{x_i (t)} \phi_i(t, {\bm v})
\end{array}
\end{align*}
gives a family of local coordinates of $\mu^{-1}(0)$ around $x_i (t)$. Remember that $\theta$ is the coordinate of the factor ${\mf g}_\epsilon^*$ in (\ref{equation61}). Then $({\bm v}, g, \theta)$ gives a local coordinate chart of $M$ around $x_i(t)$. The points parametrized by $({\bm v}, 1, \theta)$ form a local slice through $x_i(t)$ of the $G$-action. Choose an orthonormal basis ${\bm \epsilon}_1, \ldots, {\bm \epsilon}_d$ where $d = {\rm dim} G$. At this moment our construction diverges into two cases.

$\bullet$ The easy case: when $d \leq 2k$. In this case we define
\begin{align}\label{equation65}
F^{(i)}_t({\bm v}, 1, \theta) = K \ov\rho_t^{(i)}(\exp_{x_i (t)} \phi_i(t ,{\bm v})) \sum_{l=1}^d \langle \theta, {\bm \epsilon}_l\rangle v_l.
\end{align}
Here $K>0$ is a large number to be determined. This defines a function on the local $G$-slice through $x_i(t)$. Then we can extend it to a $G$-invariant function on $M$, still denoted by $F^{(i)}_t$. Then we define
\begin{align}\label{equation66}
H_t({\bm v}, g, \theta) = \pi^* \ov{H}_t +\sum_{i=1}^s F^{(i)}_t.
\end{align}
We see that $F^{(i)}_t$ vanishes along $\mu^{-1}(0)$. Therefore $(H_t)$ is a lift of $(\ov{H}_t)$. We claim that for $K$ big enough, $(H_t)$ is admissible.

Notice that elements of ${\rm Zero} \wt{\mc B}_H$ are lifted from $\ov{y}_1, \ldots, \ov{y}_s$. Moreover, $F^{(i)}_t$ vanishes up to first order along $x_i(t)$. Therefore, 
\begin{align}\label{equation67}
Y_{H_t}(x_i(t)) = W_{\ov{H}_t}(x_i(t))
\end{align}
and $\wt{x}_i = (x_i, \xi_i) \in {\rm Zero} \wt{\mc B}_H$. Then, for each $i\in \{1, \ldots, s\}$, consider the $S^1$-family of linear maps $\theta_t^{(i)}: {\mf g} \to T_{x_i(t)}M$, defined as
\begin{align}\label{equation68}
\theta_t^{(i)} \xi  = J_t [ J_t X_\xi, Y_{H_t} ] (x_i(t)).
\end{align}
For any $Z \in T_{x_i(t)} M$, extend it properly so that $[J_t X_\xi, Z] = 0$ at $x_i(t)$ for all $\xi \in {\mf g}$. Let $\langle \cdot, \cdot \rangle_t$ be the inner product on $TM|_U$ determined by $\omega$ and $J_t$. Then we have
\begin{align}\label{equation69}
\begin{split}
\langle \theta_t^{(i)}(\xi), Z \rangle_t = &\ \omega \left( [ J_t X_\xi, Y_{H_t}], Z \right) \\
= &\  J_t X_\xi\left(  \omega(Y_{H_t}, Z) \right) - \left( {\mc L}_{J_t X_\xi} \omega \right) \left( Y_{H_t},  Z \right) \\
= &\ J_t X_\xi ( Z H_t) - \langle \xi, \Omega_\tau(Y_{H_t}, Z) \rangle.
\end{split}
\end{align}
Here $\Omega_\tau\in \Omega^2(\mu^{-1}(0), {\mf g})$ is the curvature form of $\tau$ and the last equality follows from the (2) of Proposition \ref{prop61} and (3) of Definition \ref{defn62}. If $Z \in {\mf g}_{\mb C}|_{x_i(t)}$, then $J_t X_\xi(Z H_t)$ vanishes by the definition of $H_t$ (\ref{equation66}) and $F^{(i)}_t$ (\ref{equation65}), and $\Omega_\tau(Y_{H_t}, Z)$ vanishes by the property of curvature form. Moreover, ${\mf g}_{\mb C}$ and $\ov{TM}$ are orthogonal at $x_i(t)$ by Definition \ref{defn62}. Therefore, ${\rm Im} \theta_t^{(i)}\subset \ov{TM}_{x_i(t)}$. Then with respect to the basis $({\bm \epsilon}_l)_{l=1}^d$ of ${\mf g}$ and the standard basis $({\bm v}_j)_{j=1}^{2k}$ of $\ov{TM}_{x_i(t)}\simeq {\mb R}^{2k}$, by (\ref{equation65}), the matrix of $\theta_t^{(i)}$ is equal to 
\begin{align*}
K \delta_{jl} - \langle {\bm \epsilon}_l , \Omega_\tau( Y_{H_t}, {\bm v}_j) \rangle,\ 1 \leq l \leq d,\ 1\leq j \leq 2k.
\end{align*}
By (\ref{equation67}), the second term is independent of $K$. So we take $K>0$ big enough so that $\theta_t^{(i)}: {\mf g} \to \ov{TM}_{x_i(t)}$ is injective for every $t\in S^1$. 

Suppose, by contradiction, that $(H_t)$ is not admissible. Then for some $i\in \{1, \ldots, s\}$, there is an eigenvector $\wt\xi= (\xi, \beta)$ of ${\mc H}_{\wt{x}_i}$ corresponding to an eigenvalue $\kappa \neq 0$ and $\xi(t) = X_{{\bm a}(t)} + J_t X_{{\bm b}(t)}$ where ${\bm a}, {\bm b}: S^1 \to {\mf g}$. Then by (\ref{equation63}), 
\begin{align}\label{equation610}
J_t \left( \nabla_t \xi  + \nabla_\xi X_{\xi_i} - \nabla_\xi Y_{H_t} + X_\beta	 \right) = \kappa \xi,\ d \mu \cdot \xi = \kappa \beta;
\end{align}
or equivalently (here we use the third condition of Definition \ref{defn62}),
\begin{align}\label{equation611}
J_t \left( X_{{\bm a}'(t)} + J_t X_{{\bm b}'(t)} + \left[ J_t X_{{\bm b}(t)}, X_{\xi_i} - Y_{H_t} \right] + \kappa^{-1} X_{{\bm b}(t)} \right)  = \kappa X_{{\bm a}(t)} + \kappa J_t X_{{\bm b}(t)}.
\end{align}
By the $G$-invariance of $J_t$ and $Y_{H_t}$, we see that every term above belongs to ${\mf g}_{\mb C}|_{x_i(t)}$ except for $J_t [J_t X_{{\bm b}(t)}, -Y_{H_t}]= - \theta_t^{(i)}({\bm b}(t)) \in \ov{TM}_{x_i(t)}$. Then by the injectivity of $\theta_t^{(i)}$, ${\bm b}(t) \equiv 0$. Therefore
\begin{align*}
J_t \left(  X_{{\bm a}'(t)} +  [ X_{{\bm a}(t)}, X_{\xi_i(t)}] \right) = \kappa X_{{\bm a}(t)}.
\end{align*}
The left-hand-side belong to $J_t {\mf g}_{x_i(t)}$ but the right-hand-side belong to ${\mf g}|_{x_i(t)}$. Therefore ${\bm a}(t)$ vanishes identically. Then by (\ref{equation610}), $\beta$ vanishes identically and this contradicts with $(\xi, \beta) \neq 0$. So $(H_t)$ is admissible.

$\bullet$ The hard case: when $d > 2k$. In this case the construction is more complicated than the above case. Let $(a_{jl}(t))_{1\leq l \leq d}^{1\leq j \leq 2k}$ be a smooth $S^1$-family of matrices which is undetermined. We define
\begin{align*}
F^{(i)}_t({\bm v}, 1, \theta) = \ov\rho_t^{(i)}(\exp_{x_i (t)} \phi_i(t ,{\bm v}))  \left( \langle \theta, \xi_i \rangle + \sum_{1\leq l \leq d,\ 1\leq j \leq 2k} a_{jl}(t) \langle \theta, {\bm \epsilon}_l\rangle v_j \right).
\end{align*}
Here $\xi_i\in {\mf g}$ is determined by (\ref{equation64}). Then we extend $F^{(i)}_t$ to a $G$-invariant function on $M$ and define
\begin{align*}
H_t = \pi^* \ov{H}_t + \sum_{i=1}^s F^{(i)}_t.
\end{align*}
Again, the restriction of $F_t^{(i)}$ to $\mu^{-1}(0)$ is zero and $(H_t)$ is a lift of $(\ov{H}_t)$. We will show that for some properly chosen $(a_{jl}(t))$, $(H_t)$ is admissible. 

First we see that $Y_{H_t}(x_i(t)) = W_{\ov{H}_t}(x_i(t)) - X_{\xi_i}(x_i(t))$ so $\wt{x}_i = (x_i, 0)\in {\rm Zero}\wt{\mc B}_H$, which is a lift of $\ov{y}_i$. We still consider the map $\theta_t^{(i)}$ defined by (\ref{equation68}). For the same reason as above, by (\ref{equation69}), we see that the image of $\theta_t^{(i)}$ is contained in $\ov{TM}_{x_i(t)}$, and the matrix of $\theta_t^{(i)}$ with respect to the basis $({\bm \epsilon}_l)_{l=1}^d$ of ${\mf g}$ and $({\bm v}_j)_{j=1}^{2k}$ of $\ov{TM}_{x_i(t)}$ is 
\begin{align*}
a_{jl}(t) - \langle {\bm \epsilon}_l, \Omega_\tau(Y_{H_t}, {\bm v}_j) \rangle.
\end{align*}
The choice of $(a_{jl}(t))$ is completely free and we can choose them so that ${\rm ker} \theta_t^{(i)}$ is of minimal dimension $d-2k$ and is spanned by 
\begin{align*}
{\bm e}_l(t):= {\bm \epsilon}_l + \cos (2 l \pi t) {\bm \epsilon}_d,\ l = 1, \ldots, d-2k.
\end{align*}

Now we prove that $(H_t)$ constructed above is admissible. Indeed, suppose $\wt\xi = (\xi, \beta)$ is an eigenvector of ${\mc H}_{\wt{x}_i}$ corresponding to a nonzero eigenvalue $\kappa$, and $\xi(t) = X_{{\bm a}(t)} + J_t X_{{\bm b}(t)}$. Then we have (\ref{equation611}) with $\xi_i$ replaced by 0. For the same reason, it implies that ${\bm b}(t) \in {\rm ker} \theta_t^{(i)} = {\rm span}\{ {\bm e}_l(t)\ |\ l = 1, \ldots, d-2k\}$. Suppose ${\bm b}(t) = \sum_{l=1}^{d-2k} b_l(t) {\bm e}_l(t)$. On the other hand, project (\ref{equation611}) to ${\bm i}{\mf g}|_{x_i(t)}$ and ${\mf g}|_{x_i(t)}$ respectively, we obtain
\begin{align*}
{\bm a}'(t) = \left( \kappa - {1\over \kappa} \right) {\bm b}(t),\ {\bm b}'(t) = - \kappa {\bm a}(t).
\end{align*}
Therefore,
\begin{align*}
{\bm b}''(t) = \left( 1- \kappa^2 \right) {\bm b}(t).
\end{align*}
If ${\bm b}(t) \equiv 0$, then ${\bm a}(t) \equiv 0$ and $\beta(t) \equiv 0$, which contradicts $\wt\xi \neq 0$. Therefore, ${\bm b}(t)$ is an eigenfunction of ${d^2 \over dt^2}$ on the circle corresponding to an eigenvalue $1- \kappa^2 = ( 2 m \pi)^2$, for some nonnegative integer $m$. Therefore, there exists ${\bm b}_1, {\bm b}_2\in {\mf g}$ such that
\begin{align}\label{equation612}
{\bm b}(t) = {\bm b}_1 \cos(2m \pi t) + {\bm b}_2 \sin (2m \pi t) = \sum_{l=1}^{d-2k} b_l(t) \left( {\bm \epsilon}_l + \cos(2\pi l t) {\bm \epsilon}_d \right).
\end{align}
Projecting the above equation to the ${\bm \epsilon}_l$-direction for $l = 1, \ldots, d-2k$, we have
\begin{align*}
b_l(t) = \langle {\bm b}_1, {\bm \epsilon}_l\rangle \cos (2 m \pi t) + \langle {\bm b}_2, {\bm \epsilon}_l \rangle \sin (2m \pi t),\ l =1, \ldots, d-2k.
\end{align*}
However, if we project (\ref{equation612}) to the ${\bm \epsilon}_d$-direction, we obtain
\begin{multline*}
\langle {\bm b}_1, {\bm \epsilon}_d \rangle \cos (2 m \pi t) + \langle {\bm b}_2, {\bm \epsilon}_d \rangle \sin (2m \pi t) \\
= \sum_{l=1}^{d-2k} \left( \langle {\bm b}_1, {\bm \epsilon}_l \rangle \cos (2m \pi t) \cos (2l \pi t) + \langle {\bm b}_2, {\bm \epsilon}_l \rangle \sin (2m \pi t) \cos (2l \pi t) \right).
\end{multline*}

If $m>0$, projecting the two sides orthogonally to the subspace of $L^2(S^1) \otimes {\mf g}$ spanned by $\cos(2(m+d-2k)\pi t)$ and $\sin(2(m+d-2k)\pi t)$, we have
\begin{align*}
0 = {1\over 2} \langle {\bm b}_1, {\bm \epsilon}_{d-2k} \rangle \cos(2(m+d-2k)\pi t) + {1\over 2} \langle {\bm b}_2, {\bm \epsilon}_{d-2k} \rangle \sin(2 (m+d-2k)\pi t).
\end{align*}
This means that $\langle {\bm b}_1, {\bm \epsilon}_{d-2k} \rangle = \langle {\bm b}_2, {\bm \epsilon}_{d-2k} \rangle = 0$. By induction we can show that $\langle {\bm b}_1, {\bm \epsilon}_l \rangle = \langle {\bm b}_2, {\bm \epsilon}_l \rangle = 0$ for $l =1, \ldots, d-2k$. (\ref{equation612}) also implies that $\langle {\bm b}_1, {\bm \epsilon}_l \rangle = \langle {\bm b}_2, {\bm \epsilon}_l \rangle = 0$ for other $l$'s. Therefore ${\bm b}(t) \equiv 0$. When $m=0$, we can argu similarly to derive that ${\bm b}(t) \equiv 0$. It then implies that ${\bm a}(t) \equiv 0$ and $\beta(t) \equiv 0$. This contradicts $\wt\xi \neq 0$. Therefore $(H_t)$ is admissible.
\end{proof}

\begin{rem}
The above proof only works if the symplectic quotient has positive dimension, since otherwise the kernel of $\theta_t^{(i)}$ is ${\mf g}$. This is the only reason why we have this condition in Hypothesis \ref{hyp21}.			
\end{rem}

\subsection{The main theorem}

For any admissible $G$-invariant Hamiltonian $(H_t)$, any admissible almost complex structure $J = (J_t) \in \wt{\mc J}^l_H$ for $l \geq 2$ or $l = \infty$, and any constant $\lambda >0$, we can consider the equation for connecting orbits
\begin{align}\label{equation613}
\left\{ \begin{array}{ccc} \displaystyle {\partial u \over \partial s} + X_\Phi (u) + J_t \left( {\partial u \over \partial t} + X_\Psi(u) - Y_{H_t}(u) \right) & = & 0,\\
\displaystyle {\partial \Psi \over \partial s}- {\partial \Phi \over \partial t} + [\Phi, \Psi] + \lambda^2 \mu(u) & = & 0.
\end{array} \right.
\end{align}

Let $k \geq 2$, $p > 2$. Recall that for any pair ${\mf x}_\pm \in {\rm Crit} \wt{\mc A}_H$, we have the Banach manifold ${\mc B}^{k, p}({\mf x}_-, {\mf x}_+)$ defined in Section \ref{section4}. Let ${\mc M}_\lambda({\mf x}_\pm; J, H) \subset {\mc B}^{k, p}({\mf x}_-, {\mf x}_+)$ be the space of gauge equivalence classes of solutions to (\ref{equation613}). Now we can state the main theorem of this section.
\begin{thm}\label{thm68}
Suppose ${\mf x}_\pm \in {\rm Crit} \wt{\mc A}_H$ and $\wt{\mc A}_H({\mf x}_-)- \wt{\mc A}_H({\mf x}_+) > 0$. Then there exists a subset $\wt{\mc J}^{reg}_{H, \lambda} \subset \wt{\mc J}_H$ of second category (i.e., a subset containing a countable intersection of open and dense subsets) such that for any $J \in \wt{\mc J}_{H, \lambda}^{reg}$, the moduli space ${\mc M}_\lambda \left( {\mf x}_\pm; J, H \right)$ is a smooth submanifold of ${\mc B}^{k, p}( {\mf x}_-, {\mf x}_+ )$ for any $k \geq 2$ and $p >2$. Moreover, 
\begin{align*}
{\rm dim} {\mc M}_\lambda \left( {\mf x}_\pm  ; J, H \right) = {\sf cz}( {\mf x}_-  ) - {\sf cz}( {\mf x}_+  ).
\end{align*}
\end{thm}
There is no difference if we replace ${\mf x}_\pm$ by their equivalence classes $\lbr {\mf x}_\pm \rbr \in {\rm Crit} {\mc A}_H$. 

To prove this theorem it is necessary to the equation (\ref{equation613}) with $J$ in low regularity. Let $l \geq 1$. We recall the regularity theorem.
\begin{thm}\cite[Theorem 3.1]{Cieliebak_Gaio_Mundet_Salamon_2002}
Let $l \geq 1$ and $p>2$. For any solution $\wt{u} \in W^{1, p}_{loc}(\Theta, M \times {\mf g} \times {\mf g} )$ to (\ref{equation613}) with $J\in \wt{\mc J}^l_H$, there exists a gauge transformation $g\in {\mc G}^{2, p}_{loc}(\Theta, G)$ such that $g^* \wt{u} \in W^{l+1, p}_{loc}(\Theta, M \times {\mf g}\times {\mf g})$.
\end{thm}

Then let $l \geq k \geq 2$, $p > 2$; let $J \in \wt{\mc J}^l_H$. For any pair ${\mf x}_\pm \in {\rm Crit} \wt{\mc A}_H$ and any $\lambda>0$, denote by $\wt{\mc M}_\lambda^{k, p} \left( {\mf x}_\pm; J, H \right)\subset \wt{\mc B}^{k, p}({\mf x}_-, {\mf x}_+)$ the space of all solutions to (\ref{equation613}) of class $W^{k, p}_{loc}$ which are asymptotic to ${\mf x}_\pm$, and by ${\mc M}_\lambda^{k, p} \left( {\mf x}_\pm; J, H \right)$ the quotient of $\wt{\mc M}_\lambda^{k, p} \left( {\mf x}_\pm; J, H \right)$ by the action of ${\mc G}_0^{k+1, p}$. This is a subset of ${\mc B}^{k, p}({\mf x}_-, {\mf x}_+)$.

\subsection{$G$-regular points}

The key point in proving the transversality theorem is the existence of the so-called $G$-regular points for any nontrivial connecting orbit. Let $H$ be admissible and  $J \in \wt{\mc J}^l_H$. 

\begin{thm}\label{thm610}If $\wt{u} = (u, 0, \Psi)$ is a nontrivial connecting orbit between $\wt{x}_-$ and $\wt{x}_+$ (i.e., an element in $\wt{\mc M}^{k, p}({\mf x}_\pm; J, H)$), then there exists an eigenvalue $\kappa >0$ and a nonzero eigenvector $(\xi, \beta)$ of the Hessian ${\mc H}_{\wt{x}_+}$ such that
\begin{align*}
(\xi, \beta) = \lim_{s \to +\infty} e^{\kappa s} (\partial_s u , \partial_s \Phi).
\end{align*}
Moreover, there exists $s_0, c>0$ such that for $s \geq s_0$, we have
\begin{align*}
{1\over c} e^{-\kappa s} \leq \left| (\partial_s u, \partial_s \Phi ) \right| \leq c e^{-\kappa s}.
\end{align*} 
\end{thm}

\begin{proof}
By (\ref{equation49}), $(\partial_s u, 0, \partial_s \Psi)$ lies in the kernel of ${\mf D}_{\wt{u}}$. By elliptic regularity we know that $(\partial_s u, 0, \partial_s \Psi)$ lies in $W^{l, p}_{loc}$, which is at least continuously differentiable. Then we can follow the proof of \cite[Theorem B]{Robbin_Salamon_strip} to prove that $(\partial_s u, 0, \partial_s \Psi)$ is asymptotic to $e^{-\kappa s} \wt\xi$ where $\wt\xi = (\xi, \beta', \beta)$ is an eigenvalue of the operator $\wt{J} \nabla_t + R_{\wt{x}_+}$ where $R_{\wt{x}_+}$ is given by (\ref{equation43}). This is because the operator ${\mf D}_{\wt{u}}$ is of the type considered in \cite{Agmon_Nirenberg}, and we only need $C^0$-estimate, which doesn't require that $J$ is smooth. 

Then, notice that $\beta'$ should be zero, therefore $(\xi, \beta)$ is an eigenvector of ${\mc H}_{\wt{x}_+}$ corresponding to $\kappa$.
\end{proof}

This is the reason why we need the second condition in Definition \ref{defn65}: if $(H_t)$ is admissible, we can use the above theorem to show that, any nontrivial connecting orbit cannot stay within a single ``$G_{\mb C}$-orbit''. Otherwise we will have trouble in proving the existence of $G$-regular points.

\begin{defn}\label{defn611}
Let $u: \Theta \to M$ be a $C^1$-map which converges uniformly to $x_\pm: S^1 \to \mu^{-1}(0)$ (as $s \to \pm\infty$). A point $z_0 = (s_0, t_0) \in \Theta$ is called a {\bf $G$-regular point} of $u$ (with respect to $J_t$) if 

\begin{enumerate}

\item $u(z_0) \in U$;

\item $\pi_1^{J_{t_0}} \left( \partial_s u (z_0)\right) \neq 0 $, where $\pi_1^{J_t}: TM|_U \to TM|_U$ is the orthogonal projection onto the orthogonal complement of the distribution spanned by $X_\xi$ and $J_t X_\xi$, where the metric is determined by $\omega$ and $J_t$.

\item $u(s_0, t_0) \notin Gx_\pm (t_0)$.

\item $u(s_0, t_0) \notin G u( {\mb R} -\{s_0\}, t_0)$.
\end{enumerate}
\end{defn}

\begin{prop}
For any $l \geq 2$, any $(J_t) \in \wt{\mc J}^l_H$ and any connecting orbit $\wt{u} = (u, \Phi, \Psi)$ (of class $W^{l+1, p}_{loc}$) between $\wt{x}_-$ and $\wt{x}_+$ with positive energy, the set $R(u)$ of $G$-regular points of $u$ is open and nonempty.
\end{prop}

\begin{proof}
The proof is due to Frauenfelder (see \cite[Theorem 4.9]{Frauenfelder_thesis}). We assume without loss of generality that $\Phi \equiv 0$. Suppose that $(u, \Phi, \Psi)$ is a connecting orbit with positive energy, asymptotic to $x_\pm(t)$ as $s \to \pm\infty$. The openness of $R(u)$ can be proved in the same way as in the proof of \cite[Theorem 4.9]{Frauenfelder_thesis}. It remains to prove that $R(u)$ is nonempty. By Theorem \ref{thm610} and the admissibility of $H$, there exists a nonempty open subset $S\subset \Theta$ such that all $(s, t) \in S$ satisfy the first three conditions of Definition \ref{defn611}. Therefore it suffices to show that there exists $z_0 = (s_0, t_0) \in S$ such that $u(s_0, t_0) \notin G u({\mb R} -\{s_0\}, t_0)$. To reduce our situation to the case of Frauenfelder (who assumes $H \equiv 0$), we transform $(u, \Psi)$ to a solution $(v, \Psi)$ to the vortex equation on ${\mb R}^2$ with respect to the almost complex structure $\wt{J}_t = ((\phi_H^t)_*)^{-1} \circ J_t \circ (\phi_H^t)_*$ and the vanishing Hamiltonian. This is done by using the Hamiltonian diffeomorphisms $\phi_H^t$ generated by $H_t$. Then one can continue with Frauenfelder's argument to show that there exists $(s_0, t_0) \in S$ such that $v(s_0, t_0) \notin Gv({\mb R}-\{s_0\}, t_0)$. This is also a $G$-regular point of $u$ with respect to $J_t$.
\end{proof}

\subsection{The universal moduli space}

Suppose $l \geq k \geq 2$, $p >2$. We denote
\begin{align*}
{\mc M}^{k, p} \left(  {\mf x}_\pm ; \wt{\mc J}^l_H, H \right) := \left\{ \left( [\wt{u}], J \right) \ |\  J \in \wt{\mc J}^l_H, \ [\wt{u}] = [u, \Phi, \Psi] \in {\mc M}^{k, p}\left(  {\mf x}_\pm ; J, H \right) \right\}.
\end{align*}

\begin{prop}
For $l \geq k \geq 2$, the universal moduli space is a $C^{l-k}$-Banach submanifold of ${\mc B}^{k, p}({\mf x}_-, {\mf x}_+) \times \wt{\mc J}^l_H$.
\end{prop}

\begin{proof}
We use the notations of Section \ref{section4}. The Banach space bundle ${\mc E}^{k-1, p}({\mf x}_-, {\mf x}_+)$ is pulled back to the product ${\mc B}^{k, p}({\mf x}_-, {\mf x}_+) \times \wt{\mc J}_H^l$, which is denoted by the same symbol. Then the left-hand-side of (\ref{equation613}) defines a smooth section
\begin{align*}
\wt{\mc F}_\lambda: {\mc B}^{k, p}({\mf x}_-, {\mf x}_+) \times \wt{\mc J}_H^l \to {\mc E}^{k-1, p}({\mf x}_-, {\mf x}_+)
\end{align*}
whose vanishing locus is ${\mc M}^{k, p} \left( {\mf x}_\pm; \wt{\mc J}^l_H, H \right)$. Then we just need to prove that the linearized operator
\begin{align*}
D: T_{[\wt{u}]} {\mc B}^{k, p}({\mf x}_-, {\mf x}_+) \times T_J \wt{\mc J}^l_H \to {\mc E}^{k-1, p}_{[\wt{u}]}({\mf x}_-, {\mf x}_+)
\end{align*}
is surjective at every $[\wt{u}] \in {\mc M}^{k, p}({\mf x}_\pm; J, H)$. As we did in Section \ref{section4}, we consider instead an augmented operator ${\mf D}$ at any representative $\wt{u}\in \wt{\mc M}^{k, p}\left( {\mf x}_\pm; J, H \right)$, which is a linear map
\begin{align*}
{\mf D}: T_{\wt{u}} \wt{\mc B}^{k, p}({\mf x}_-, {\mf x}_+) \times T_J \wt{\mc J}^l_H \to \wt{\mc E}^{k-1, p}_{\wt{u}} ({\mf x}_-, {\mf x}_+) \oplus W^{k-1, p} \left( \Theta, {\mf g} \right).
\end{align*}
It suffices to prove that ${\mf D}$ is surjective. 

Notice that the restriction of ${\mf D}$ to $T_{\wt{u}} \wt{\mc B}^{k, p}({\mf x}_-, {\mf x}_+)$, which is the augmented linearized operator ${\mf D}_{\wt{u}}$ (\ref{equation42}) is Fredholm, we only need to prove that ${\mf D}$ has dense range. If not the case, then there exists a nonzero vector $\wt{\eta} = (\eta, \vartheta_1, \vartheta_2) \in \wt{\mc E}^{k-1, p}_{\wt{u}}({\mf x}_-, {\mf x}_+) \oplus W^{k-1, p}\left( \Theta, {\mf g}\right)$ which lies in the $L^2$-orthogonal complement of the image of ${\mf D}$; therefore $\wt\eta$ also lies in the orthogonal complement of the image of ${\mf D}_{\wt{u}}$. Hence $\wt\eta \in {\rm ker} {\mf D}_{\wt{u}}^*$. By elliptic regularity associated to ${\mf D}_{\wt{u}}^*$, we see that $\wt\eta$ is continuous. We will show that $\wt\eta$ vanishes on an nonempty open subset of $R(u)$, which, by the unique continuation property of ${\mf D}_{\wt{u}}^*$, contradicts with the fact that $\wt\eta \neq 0$.

Notice that ${\mf D}$ restricted to $T_J\wt{\mc J}^l_H$ reads
\begin{align*}
\displaystyle T_J \wt{\mc J}^l_H \ni (E_t) \mapsto E_t \left( {\partial u \over \partial t} + X_\Psi - Y_{H_t} \right) = E_t J_t \left( {\partial u \over \partial s} \right).
\end{align*}
Suppose $z_0 = (s_0, t_0) \in R(u)$ and $\eta(z_0) \neq 0$. Since $\pi_1 \left( \partial_s u (z_0) \right) \neq 0$, and there are only finitely many 1-periodic orbits of $(\ov{H}_t)$, we may move $z_0$ a little bit inside $R(u)$ to assume that $u(z_0) \notin Gx_i(t_0)$ for all $i \in \{1, \ldots, s\}$. Then we can choose $E_0 \in {\rm End}\left( T_{u(z_0)} M \right)$ such that $J_{t_0} E_0 + E_0 J_{t_0} = 0$ and $\omega( \cdot, E_0 \cdot)$ being symmetric and such that
\begin{align*}
\left\langle E_0 \left( J_{t_0} {\partial u \over \partial s}(z_0) \right), \eta(z_0) \right\rangle > 0 .
\end{align*}
We extend $E_0$ to a $G$-invariant section $(E_t) \in \Gamma^l \left( S^1\times M, {\rm End} TM \right)$ which satisfies the first two conditions in the statement of Lemma \ref{lemma63}, such that $E_{t_0}(u(z_0)) = E_0$. On the other hand, for any $G$-invariant cut-off function $\rho = (\rho_t) \in C^\infty_c( S^1 \times M)$ supported near $(t_0, u(z_0))$ (hence its support is away from $\cup_{i=1}^s G\{ (t, x_i(t))\ |\ t\in S^1\}$, $\rho_t E_t\in T_J \wt{\mc J}^l_H$. We would like to choose such a function such that 
\begin{align}\label{equation614}
\left\langle {\mf D}(0, \rho_t E_t), \wt\eta \right\rangle = \int_\Theta \langle \rho_t(u(s, t)) E_t(u(s, t)), \eta(s, t) \rangle ds dt > 0.
\end{align}
Indeed, choose a small $G$-invariant neighborhood $W(z_0) \subset U$ of $u(z_0)$. Then since $\pi_1(\partial_s u(z_0)) \neq 0$, $u$ induces an embedding $\ov{u}$ from a small disk $B_\epsilon(z_0)\subset \Theta$ into $W(z_0)/G$. Then by \cite[Remark 4.4]{Floer_Hofer_Salamon}, there exists a smooth function $\ov\rho = (\ov\rho_t): S^1 \times U/ G$ supported inside $(t_0 - \epsilon, t_0+ \epsilon) \times W(z_0)/G$ such that for any $(s, t) \in B_\epsilon(z_0)$,
\begin{align*}
\left\langle E_t	\left( J_t {\partial u \over \partial s}(s, t) \right), \eta(s, t) \right\rangle = \ov\rho_t (\ov{u}(s, t)). 
\end{align*}
Then $\ov\rho$ lifts to a $G$-invariant function $\rho$ supported near $(t_0, u(z_0))$. We can shrink the support of $\rho$ such that (\ref{equation614}) holds. Then this contradicts with that $\wt\eta \in {\rm ker} {\mf D}^*$. Therefore $\eta$ vanishes on an open neighborhood $V_0(z_0 )$ of $z_0$.

Then look at the first component of ${\mf D}_{\wt{u}}^* \left( \eta, \vartheta_1, \vartheta_2 \right) =0$. On $V_0(z_0)$ it reads
\begin{align*}
J_t X_{\vartheta_1}(u) - X_{\vartheta_2}(u) = 0.
\end{align*}
Since for $z \in R(u)$, $u(z) \in U$ where $G$ acts freely, $\vartheta_1|_{V_0(z_0)} = \vartheta_2|_{V_0(z_0)} = 0$. So $\wt\eta$ vanishes on $V_0(z_0)$. This finishes our proof.
\end{proof}

Then apply Smale's Sard theorem to the projection ${\mc M}^{k, p} \left( {\mf x}_\pm; \wt{\mc J}^l_H, H \right) \to \wt{\mc J}^l_H$, it is standard to show the following (cf. the proof of \cite[Theorem 3.1.5]{McDuff_Salamon_2004}).
\begin{thm}
For any pair ${\mf x}_\pm  \in {\rm Crit} \wt{\mc A}_H$ with $\wt{\mc A}_H({\mf x}_-) - \wt{\mc A}_H({\mf x}_+) > 0$, there exists $l_0 = l_0 (k, {\mf x}_\pm ) \in {\mb Z}_+$ such that for any $l \geq l_0$, there exists a subset $\wt{\mc J}^{l; reg}_{H, \lambda} \subset \wt{\mc J}^l_H$ of second category, such that for any $J \in \wt{\mc J}^{l; reg}_{H, \lambda} $ and any $[\wt{u}]= [u, \Phi, \Psi] \in {\mc M}^{k, p} \left( {\mf x}_\pm ; J, H \right)$, the linearized operator 
\begin{align*}
{\mc D}_{[\wt{u}]}^{J, H}: T_{[\wt{u}]} {\mc B}^{k, p}({\mf x}_-, {\mf x}_+) \to {\mc E}^{k-1, p}_{[\wt{u}]}({\mf x}_-, {\mf x}_+)
\end{align*}
is surjective.
\end{thm}

Theorem \ref{thm68} follows from the above theorem and Taubes' trick (\cite[Page 52]{McDuff_Salamon_2004}).

\section{Floer homology}\label{section7}

In this section we use the moduli spaces ${\mc M}\left( \llbracket{\mf x}_-\rrbracket, \llbracket{\mf x}_+\rrbracket; J, H \right)$ to define the vortex Floer homology group $VHF_* \left( M, \mu; H \right)$. We also construct the continuation map, which is used to show that the homology group is independent of the choice of the data $(H, J, \lambda)$. 

By Theorem \ref{thm68}, for any admissible $H = (H_t)$ and any $\lambda>0$, we can choose a generic $S^1$-family of ``admissible'' almost complex structures $J \in \wt{\mc J}_{H, \lambda}^{reg}$ (see Definition \ref{defn62}) such that the moduli space is cut off transversely. For simplicity we only discuss the case $\lambda = 1$ but there is no difference for general $\lambda$. We abbreviate $\wt{\mc J}_{H, 1}^{reg} = \wt{\mc J}_H^{reg}$. So for $J \in \wt{\mc J}^{reg}_H$, and a pair $\llbracket{\mf x}_\pm \rrbracket \in {\rm Crit} {\mc A}_H$ with ${\mc A}_H(\llbracket {\mf x}_- \rrbracket ) \neq {\mc A}_H ( \llbracket {\mf x}_+ \rrbracket)$, ${\mc M}\left( \llbracket {\mf x}_-\rrbracket, \llbracket{\mf x}_+\rrbracket; J, H \right)$ is a smooth manifold with 
\begin{align*}
{\rm dim} {\mc M} \left( \llbracket{\mf x}_-\rrbracket, \llbracket{\mf x}_+\rrbracket; J, H \right) = {\sf cz}( \llbracket{\mf x}_-\rrbracket ) - {\sf cz}(\llbracket{\mf x}_+\rrbracket).
\end{align*}
The orbit space of the free ${\mb R}$-action on ${\mc M}\left( \llbracket{\mf x}_-\rrbracket, \llbracket{\mf x}_+\rrbracket; J, H \right)$ is $\widehat{\mc M}(\llbracket{\mf x}_-\rrbracket, \llbracket{\mf x}_+\rrbracket; J, H)$. Combining the compactness theorem (Theorem \ref{thm53}), we have
\begin{prop}
If $J \in \wt{\mc J}^{reg}_H$ and $\lbr{\mf x}_-\rbr \neq \lbr{\mf x}_+\rbr$, then 
\begin{align*}
\begin{split}
{\sf cz}( \llbracket{\mf x}_- \rrbracket ) - {\sf cz}( \llbracket{\mf x}_+\rrbracket) \leq 0 &\ \Longrightarrow {\mc M} \left( \llbracket {\mf x}_- \rrbracket, \llbracket {\mf x}_+ \rrbracket; J, H \right) = \emptyset;\\
{\sf cz}( \llbracket{\mf x}_- \rrbracket) - {\sf cz}( \llbracket{\mf x}_+ \rrbracket) = 1 &\ \Longrightarrow \# \widehat{\mc M}\left( \llbracket {\mf x}_- \rrbracket, \llbracket {\mf x}_+ \rrbracket; J, H \right) < \infty.
\end{split}
\end{align*}
\end{prop}

We also denote the following orbit space of ${\mb R}$-actions.
\begin{align*}
\begin{split}
\widehat{\mc M} \left( {\mf x}_-, {\mf x}_+; J, H \right) := &\ {\mc M}\left( {\mf x}_-, {\mf x}_+; J, H \right) / {\mb R},\\
\widehat{\mc M} \left( [{\mf x}_-], [{\mf x}_+]; J, H \right) := &\ {\mc M}\left( [{\mf x}_-], [{\mf x}_+]; J, H \right) / {\mb R}.
\end{split}
\end{align*}
We still use $\{ \cdot \}$ to denote an ${\mb R}$-orbit in these spaces.

\subsection{The gluing map and coherent orientation}

The boundary operator of the Floer chain complex is defined by the (signed) counting of $\widehat{\mc M}( \llbracket{\mf x}_-\rrbracket, \llbracket{\mf x}_+\rrbracket; J, H)$. If we want to define the Floer homology over ${\mb Z}_2$, then we don't need to orient the moduli space; otherwise, the orientation of $\widehat{\mc M}( \llbracket {\mf x}_- \rrbracket, \llbracket {\mf x}_+ \rrbracket; J, H)$ can be treated in the same way as the usual Hamiltonian Floer theory, since the augmented linearized operator ${\mf D}_{\wt{u}}$ (whose determinant is canonically isomorphic to the determinant of the actual linearization $d{\mc F}_{[\wt{u}]}$), is of the same type of Fredholm operators considered in the abstract setting of \cite{Floer_Hofer_Orientation}. We first give the gluing construction and then discuss the coherent orientations of the moduli spaces.

In this subsection we construct the gluing map for broken trajectories with only one breaking. The general case is similar. This construction is, in principle, the same as the standard construction in various types of Morse-Floer theory (see \cite{Floer_Hofer_Orientation} and \cite{Salamon_lecture}), with a gauge-theoretic flavor. The gauge symmetry makes the construction a bit more complex, since we always glue representatives, and we want to show that the gluing map is independent of the choice of the representatives. 

In this subsection, we fix the choice of the admissible family $J = (J_t)  \in \wt{\mc J}^{reg}_H$ and omit the dependence in the notations of the moduli spaces on $J$ and $H$.

For any pair $ {\mf x}_\pm \in {\rm Crit} \wt{\mc A}_H$, we say that a solution $\wt{u} \in \wt{\mc M}\left( {\mf x}_-, {\mf x}_+ \right)$ is in {\bf $r$-temporal gauge}, if its restrictions to $[r, +\infty) \times S^1$ and $(-\infty, r]\times S^1$ are in temporal gauge, for $r>0$. Now we fix a number $r= r_0$ and only consider solutions in $r_0$-temporal gauge. Similar treatment, including this gauge-fixing can be found in \cite[Section 4.7]{Frauenfelder_thesis}.

Now we take three elements ${\mf x}, {\mf y}, {\mf z}\in {\rm Crit} \wt{\mc A}_H$ with
\begin{align*}
{\sf cz}({\mf x})-1 = {\sf cz}({\mf y}) = {\sf cz}({\mf z})+1.
\end{align*}
We would like to construct, for a large $R_0>0$, the {\bf gluing map}
\begin{align*}
\mf{glue}: \widehat{\mc M}\left( [{\mf x}], [{\mf y}] \right) \times (R_0, +\infty) \times \widehat{\mc M}\left( [{\mf y}], [{\mf z}]\right) \to \widehat{\mc M}\left( [{\mf x}], [{\mf z}]\right).
\end{align*}

\subsubsection{The approximate solution}

Consider two gauge equivalence classes of solutions $[\wt{u}_\pm]= [u_\pm, \Phi_\pm, \Psi_\pm]$, $[ \wt{u}_- ] \in {\mc M}\left( [{\mf x}], [{\mf y}]\right)$, $[\wt{u}_+] \in {\mc M}\left( [{\mf y}], [{\mf z}] \right)$ with their representatives both in $r_0$-temporal gauge and $\wt{u}_\pm$ are asymptotic to ${\mf y}$ as $ s \to \mp \infty$. Then there exists $R_1>0$ such that 
\begin{align*}
\pm s \geq R_1 \Longrightarrow u_\mp (s, t) = \exp_{y(t)} \xi_\mp(s, t).
\end{align*}
Here $\Theta_{R_1}^+= [R_1, +\infty) \times S^1$, $\Theta_{R_1}^-= (-\infty, - R_1]\times S^1$ and $\xi_\pm \in W^{k, p}\left( \Theta^{\mp}_{R_1}, y^*TM\right)$. 

Next, we take a cut-off function $\rho$ such that $s\geq 1\Longrightarrow \rho(s) = 1$, $s\leq 0\Longrightarrow \rho(s) = 0$. For each $R>> r_0$, denote $\rho_R(s) = \rho(s- R)$. We construct the ``connected sum''
\begin{align*}
u_R (s, t) = \left\{\begin{array}{cc} u_-(s + R, t),& s \leq -{R\over 2}-1\\
                                            \exp_{y(t)} \left(  \rho_{R\over 2}(-s) \xi_- (s+ R, t) + \rho_{R\over 2} (s) \xi_+(s-R, t) \right), & s \in \left[-{R\over 2} -1, {R\over 2} +1 \right]\\
																						u_+( s- R, t), & s \geq {R\over 2} +1 
\end{array} \right.
\end{align*}
\begin{align*}
\left(\Phi_R, \Psi_R \right)(s, t) = \left\{\begin{array}{cc} \left( \Phi_-(s + R, t), \Psi_-(s + R, t) \right),& s \leq -{R\over 2}-1\\
                                         \left( 0, \rho_{R\over 2} (-s) \Psi_-(s + R, t) +  \rho_{R\over 2}(s) \Psi_+(s-R, t) \right), & s \in \left[-{R\over 2} -1, {R\over 2} +1 \right]\\
																						\left( \Phi_+(s-R, t), \Psi_+(s-R, t) \right). & s \geq {R\over 2} +1 \end{array}\right.
\end{align*}
Now it is easy to see that, if we change the choice of representatives $\wt{u}_\pm$ which are also in $r_0$-temporal gauge, the connected sum $\wt{u}_R:= (u_R, \Phi_R, \Psi_R)$ doesn't change for $s\in [-{R\over 2} -1, {R\over 2} + 1 ]$ and hence we obtain a gauge equivalent connected sum. Moreover, if $h \in L_0 G$ and ${\mf y}'= h^* {\mf y}$, then we choose $g_\pm: \Theta \to G$ such that $g_\pm (s, t) = {\rm Id}$ for $\pm s >>0$ and $g_\pm(s, t) = h(t)$ for $\pm s \leq -r_0$. Then we replace $\wt{u}_\pm$ by $\wt{u}_\pm' := g_\pm^* \wt{u}_\pm$, with 
\begin{align*}
\wt{u}_-' \in \wt{\mc M}\left( {\mf x}, {\mf y}' \right),\ \wt{u}_+' \in \wt{\mc M}\left( {\mf y}', {\mf z} \right)
\end{align*}
which are also in $r_0$-temporal gauge. Then if we do the above connected sum construction, we obtain $\wt{u}_R'$ which is gauge equivalent to $\wt{u}_R$.

Now we consider the augmented linearized operator ${\mf D}_R:= {\mf D}_{\wt{u}_R}$ along $\wt{u}_R \in \wt{\mc B}^{1, p}\left( {\mf x}, {\mf z} \right)$.

\begin{lemma}There exists $c>0$ and $R_2 >0$ such that for every $R\geq R_2$ and $\wt\eta \in \wt{\mc E}^{2, p}_{\wt{u}_R}\left( {\mf x}, {\mf z}\right) \oplus W^{2, p}\left( \Theta, {\mf g}\right)$, we have
\begin{align*}
\left\| {\mf D}_R^* \wt\eta \right\|_{W^{1, p}} \leq c \left\| {\mf D}_R {\mf D}_R^* \wt\eta \right\|_{L^p}.
\end{align*}
\end{lemma}
\begin{proof}
Same as the proof of \cite[Proposition 3.9]{Salamon_lecture}
\end{proof}

Hence we can construct a right inverse
\begin{align*}
{\mf Q}_R:= {\mf D}_R^* \left( {\mf D}_R {\mf D}_R^* \right)^{-1}: \wt{\mc E}^{0, p}_{\wt{u}_R} \left( {\mf x}, {\mf z} \right) \oplus L^p\left( \Theta, {\mf g}\right) \to T_{\wt{u}_R} \wt{\mc B}^{1, p} \left( {\mf x}, {\mf z} \right)
\end{align*}
with 
\begin{align*}
\left\|{\mf Q}_R \right\|\leq c.
\end{align*}

Now we write ${\mf Q}_R:= ({\mc Q}_R, {\mc A}_R)$ with ${\mc Q}_R: \wt{\mc E}_{\wt{u}_R}^{0, p} \left( {\mf x}, {\mf z} \right) \to T_{\wt{u}_R} \wt{\mc B}^{1, p}\left( {\mf x}, {\mf z} \right)$ and ${\mc A}_R: L^p\left( \Theta, {\mf g} \right) \to T_{\wt{u}_R} \wt{\mc B}^{1, p} \left( {\mf x}, {\mf z} \right)$. Then actually the image of ${\mc Q}_R$ lies in the kernel of $d_0^*$ and therefore ${\mc Q}_R$ is a right inverse to $d\wt{\mc F}_{\wt{u}_R}|_{{\rm ker} d_0^*}$. Because our construction is natural with respect to gauge transformations, we see that ${\mc Q}_R$ induces an injection
\begin{align*}
\ov{\mc Q}_R: {\mc E}_{[\wt{u}_R]}^{0, p} \left( {\mf x}, {\mf z} \right) \to T_{[\wt{u}_R]} {\mc B}^{1, p}  \left( {\mf x}, {\mf z} \right)
\end{align*}
which is a right inverse to the linearized operator $d{\mc F}_{[\wt{u}_R]}$ and which is bounded by $c$. By the implicit function theorem (\cite[Theorem A.3.4]{McDuff_Salamon_2004}), we have
\begin{prop}
There exists $R_0 >0$, $\delta_0 >0$ such that for each $R \geq R_0$, there exists a unique $\wt\xi \in {\rm Im} {\mc Q}_R \subset T_{\wt{u}_R} \wt{\mc B}^{1, p}\left( {\mf x}, {\mf z} \right)$, $\|\wt\xi \|_{W^{1, p}} \leq \delta_0$ such that
\begin{align*}
\wt{\mc F} \left( \exp_{\wt{u}_R} \wt\xi \right) = 0,\ \| \wt\xi \|_{W^{1, p}} \leq 2c \| \wt{\mc F} (\wt{u}_R ) \|_{L^p}.
\end{align*}
\end{prop}

Then we can define
\begin{align*}
\begin{array}{cccc} \wt{\mf{glue}} : & \widehat{\mc M}\left( {\mf x}, {\mf y} \right) \times (R_0, +\infty) \times \widehat{\mc M} \left( {\mf y}, {\mf z} \right) & \to & \widehat{\mc M} \left( {\mf x}, {\mf z} \right)\\
& \left( \{ \wt{u}_- \}, R, \{ \wt{u}_+ \} \right) &\mapsto & \{ \exp_{\wt{u}_R} \wt\xi \}.
\end{array}
\end{align*}
By the gauge equivariance of the construction, it induces the gluing map
\begin{align*}
\mf{glue}: \widehat{\mc M} \left( [{\mf x}], [{\mf y}] \right) \times (R_0, +\infty) \times \widehat{\mc M} \left( [{\mf y}], [{\mf z}] \right) \to \widehat{\mc M} \left( [{\mf x}], [{\mf z}] \right).
\end{align*}
We can also replace $[{\mf x}], [{\mf y}], [{\mf z}]$ by $\lbr {\mf x}\rbr, \lbr {\mf y}\rbr, \lbr{\mf z}\rbr$ respectively.

\subsubsection{Coherent orientation and the boundary operator}

On the other hand, it is easy to see that the augmented linearized operators ${\mf D}_{\wt{u}}$ for all connecting orbits $\wt{u}$ is of ``class $\Sigma$'' considered in \cite{Floer_Hofer_Orientation}. Therefore, by the main theorem of \cite{Floer_Hofer_Orientation}, there exists a ``coherent orientation'' with respect to the gluing construction. Choosing such a coherent orientation, then to each zero-dimensional moduli space $\widehat{\mc M}( \lbr {\mf x} \rbr, \lbr {\mf y} \rbr; J, H)$, we can associate the counting $\chi_J ( \lbr {\mf x} \rbr, \lbr {\mf y} \rbr)\in {\mb Z}$, where each trajectory $\{ \wt{u}\}$ contributes 1 (resp. -1) if the orientation of $ \{ \wt{u} \}$ coincides (resp. differ from) the ``flow orientation'' of the solution. For all the other cases, define $\chi_J (\lbr {\mf x}\rbr , \lbr {\mf y} \rbr ) = 0$.

Recall that in Subsection \ref{subsection24} we defined the $\Lambda_{\mb Z}$-modules $VCF(M, \mu; H; \Lambda_{\mb Z})$. It is generated by ${\rm Crit} {\mc A}_H$, and is graded by the Conley-Zehnder index. We denote by $VCF_k \left( M, \mu; H; \Lambda_{\mb Z} \right)$ the subgroup consisting of degree $k$ elements. Then we define
\begin{align}\label{equation71}
\begin{array}{cccc}
\delta_J : & VCF_k\left( M,\mu; H; \Lambda_{{\mb Z}} \right) & \to & VCF_{k-1} \left( M, \mu; H; \Lambda_{\mb Z} \right)\\[0.3cm]
          & \lbr {\mf x} \rbr  & \mapsto & \displaystyle \sum_{\lbr {\mf y}\rbr \in {\rm Crit}{\mc A}_H } \chi_J ( \lbr \mf x \rbr, \lbr \mf y \rbr) \lbr {\mf y} \rbr.
\end{array}
\end{align}

As in the ordinary Hamiltonian Floer theory, we have
\begin{thm}
For any choice of the coherent orientations on all ${\mc M}( \lbr {\mf x} \rbr, \lbr {\mf y} \rbr)$, the operator $\delta_J$ in (\ref{equation71}) defines a morphism of $\Lambda_{\mb Z}$-modules satisfying $\delta_J \circ \delta_J = 0$. 
\end{thm}

This makes $\left( VCF_*(M, \mu; H; \Lambda_{\mb Z} ), \delta_J \right)$ a chain complex of $\Lambda_{\mb Z}$-modules, to which will be generally referred as the {\bf vortex Floer chain complex}. Therefore the vortex Floer homology group is defined as the graded $\Lambda_{\mb Z}$-module
\begin{align*}
VHF_k \left( M, \mu; J, H; \Lambda_{\mb Z} \right):= {{\rm ker} \left(\delta_J: VCF_k \to VCF_{k-1} \right) \over {\rm im} \left( \delta_J: VCF_{k+1}\to VCF_{k} \right)   }.
\end{align*}

\subsection{The continuation map}

Now we prove that the vortex Floer homology group defined above is independent of the choice of admissible family of almost complex structures and the time-dependent Hamiltonians, and, if we use the moduli space of (\ref{equation17}) instead of (\ref{equation14}) to define the Floer homology, independent of the parameter $\lambda>0$. This following type of argument is standard, i.e., the continuation method.

Let $\left( \left( J^-_t \right), \left( H^-_t \right), \lambda^- \right)$ and $\left( \left( J^+_t \right) , \left( H^+_t \right), \lambda^+ \right)$ be two triples where $\lambda^\pm > 0$, $(H_t^\pm)$ are admissible Hamiltonians and $\left( J_t^\pm \right) \in \wt{\mc J}^{reg}_{H^\pm, \lambda^\pm }$. We choose a cut-off function $\rho: {\mb R}\to [0, 1]$ such that $\rho(s) = 1$ for all $s \leq -1$ and $\rho(s) = 0$ for all $s \geq 1$. Then we define
\begin{align*}
H_{s, t} = \rho(s) H^-_t + (1-\rho(s)) H^+_t,\ \lambda_s = \rho(s) \lambda^- + (1-\rho(s)) \lambda^+.
\end{align*}
We denote this family of Hamiltonians by ${\bm H}= (H_{s, t})$ and ${\bm \lambda} : = (\lambda_s)$. 

We have to choose a family of almost complex structures ${\bm J} = (J_{s, t})$ to define the equation. Let $\wt{\mc J} \left( J^-, J^+ \right)$ be the space of families of $G$-invariant, $\omega$-compatible almost complex structures $\wt{\mc J} \left( J^-, J^+ \right) $ consisting of smooth families of almost complex structures ${\bm J} = (J_{s, t})_{(s, t) \in \Theta}$, such that for all $k \geq 1$,
\begin{align}\label{equation72}
\left| e^{|s|}  \left( J_{s, t} - J^-_t \right) \right|_{C^k (\Theta_- \times M)} < \infty,\ \left| e^{|s|}\left( J_{s, t} - J^+_t \right) \right|_{C^k(\Theta_+ \times M)} < \infty.
\end{align}

For each ${\bm J} \in \wt{\mc J}\left( J^-, J^+ \right)$, consider the following equation on $\wt{u}= (u, \Phi,\Psi): \Theta \to M \times {\mf g} \times {\mf g}$
\begin{align}\label{equation73}
\left\{ \begin{array}{ccc}
\displaystyle {\partial u \over \partial s}  + X_{\Phi}(u) + J_{s, t} \left( {\partial u \over \partial t}  + X_{\Psi}(u) - Y_{H_{s, t}}(u) \right) & = & 0;\\[0.3 cm]
\displaystyle {\partial \Psi \over \partial s} - {\partial \Phi \over \partial t}  + [\Phi, \Psi] + \lambda_s^2 \mu(u) & = & 0.
\end{array}\right.
\end{align}
For the same reason as in Section \ref{section3}, any finite energy solution whose image in $M$ has compact closure is gauge equivalent to a solution which is asymptotic to some $\wt{x}_\pm \in {\rm Zero}\wt{\mc B}_H$ as $ s\to \pm\infty$. Hence for any pair $ \wt{\mf x}^\pm  \in {\rm Crit} \wt{\mc A}_{H^\pm}$, we can consider the space of solutions to (\ref{equation73})
\begin{align*}
\wt{\mc N} \left( {\mf x}^-, {\mf x}^+; {\bm J}, {\bm H}, {\bm \lambda} \right) \subset \wt{\mc B}^{k, p}( {\mf x}^-, {\mf x}^+)
\end{align*}
which are asymptotic to ${\mf x}^\pm$ as $s \to \pm\infty$.

One thing to check in defining the continuation map is the energy bound of solutions, which gurantees the compactness of the moduli space. We define the energy of a solution $(u, \Phi, \Psi)$ to (\ref{equation73}) by
\begin{align*}
 E(u, \Phi, \Psi) = \left\| \partial_s u + X_\Phi(u) \right\|_{L^2(\Theta)}^2 + \left\|\lambda_s \mu(u) \right\|_{L^2(\Theta)}^2.
\end{align*}
Here the $L^2$-norm is taken with respect to the family of metrics induced by $\omega$ and $J_{s, t}$ parametrized by $(s, t)\in \Theta$. Then we have
\begin{prop}\label{prop75}
For any $\wt{u} = (u, \Phi, \Psi) \in \wt{\mc N}\left( {\mf x}^-, {\mf x}^+; {\bm J}, {\bm H}, {\bm \lambda} \right)$, we have
\begin{align}\label{equation74}
E \left( u, \Phi, \Psi \right) = \wt{\mc A}_{H^-} ({\mf x}^- ) - \wt{\mc A}_{H^+} (  {\mf x}^+ ) - \int_\Theta {\partial H_{s, t}\over \partial s} (u) ds dt.
\end{align}
\end{prop}

\begin{proof} Transforming $\wt{u}$ into temporal gauge, the energy density reads
\begin{align*}
\begin{split}
\left| {\partial u \over \partial s} \right|^2 + \left| \lambda_s \mu(u) \right|^2 = & \omega \left( {\partial u\over \partial s} , {\partial u \over \partial t} + X_{\Psi} - Y_{H_{s, t}}(u) \right) - \mu(u) \cdot {\partial \Psi \over \partial s}\\
= & \omega \left( {\partial u \over \partial s}, {\partial u \over \partial t} \right) - {\partial \over \partial s} \left( \mu(u) \cdot \Psi \right) + {\partial \over \partial s} (H_{s, t}(u)) - {\partial H_{s, t}\over \partial s}(u).
\end{split}
\end{align*}
Then integrating over $\Theta$, we obtain (\ref{equation74}).
\end{proof}

Then we prove 
\begin{prop} 
There exists a subset $\wt{\mc J}^{reg}_{{\bm H}, {\bm \lambda}} \left( J^-, J^+\right) \subset \wt{\mc J}\left( J^-, J^+ \right)$ of second category such that for every ${\bm J} = (J_{s, t}) \in \wt{\mc J}^{reg}_{{\bm H}, {\bm \lambda}}\left( J^-, J^+ \right)$ and every pair $ {\mf x}^\pm \in {\rm Crit} \wt{\mc A}_{H^\pm}$, the moduli space ${\mc N} \left( {\mf x}^\pm; {\bm J}, {\bm H}, {\bm \lambda} \right)$ is a smooth manifold of dimension ${\sf cz}( {\mf x}^- ) - {\sf cz}( {\mf x}^+ )$.
\end{prop}

\begin{proof}
Let $l \geq 2$ and let $\wt{\mc J}^l \left(J^-, J^+\right)$ be the space of families of $G$-invariant, $\omega$-compatible almost complex structures parametrized by $(s, t)\in \Theta$ of class $C^l$, which satisfy (\ref{equation72}) up to $k = l$. These are Banach manifolds whose intersection is $\wt{\mc J}\left( J^-, J^+ \right)$. 

As we did in Section \ref{section6}, for $l \geq k \geq 2$ and $p>2$, for each ${\bm J} \in \wt{\mc J}^l \left( J^-, J^+ \right)$, we define the space of solutions to (\ref{equation73})
\begin{align*}
\wt{\mc N}^{k, p} \left( {\mf x}^-, {\mf x}^+; {\bm J}, {\bm H}, {\bm \lambda} \right) \subset \wt{\mc B}^{k, p}\left( {\mf x}^-, {\mf x}^+ \right).
\end{align*}
We denote
\begin{align*}
{\mc N}^{k, p} \left( {\mf x}^-, {\mf x}^+; {\bm J}, {\bm H}, {\bm \lambda} \right):= \wt{\mc N}^{k, p} \left( {\mf x}^-, {\mf x}^+; {\bm J}, {\bm H}, {\bm \lambda} \right)/ {\mc G}_0^{k+1, p}.
\end{align*}

Then any solution $\wt{u} \in \wt{\mc N}^{k, p} \left(  {\mf x}^-, {\mf x}^+; {\bm J}, {\bm H}, {\bm \lambda} \right)$ will go into $U$ as $|s|\to \infty$. This implies that every such $\wt{u}$ is ``irreducible'' in the sense of \cite[Definition 4.1]{Cieliebak_Gaio_Mundet_Salamon_2002}. Hence we can prove the transversality of the universal moduli space over $\wt{\mc J}^l\left( J^-, J^+ \right)$ in the same way as \cite[Theorem 4.10]{Cieliebak_Gaio_Mundet_Salamon_2002}, because now the perturbation is allowed to depend on both $s$ and $t$. By Sard-Smale theorem, this implies that for a generic ${\bm J} \in \wt{\mc J}^l\left( J^-, J^+ \right)$, the moduli space ${\mc N}^{k, p}\left( {\mf x}^-, {\mf x}^+; {\bm J}, {\bm H}, {\bm \lambda} \right)$ is cut off transversely. The Proposition follows by applying Taubes' trick. 
\end{proof}

Choose an element ${\bm J} \in \wt{\mc J}^{reg}_{{\bm H}, {\bm \lambda}}\left( J^-, J^+ \right)$. As in Section \ref{section4}, for each pair $\lbr {\mf x}^\pm \rbr \in {\rm Crit} {\mc A}_{H^\pm}$, we can define the moduli space
\begin{align}\label{equation75}
{\mc N} \left( \lbr {\mf x}^-\rbr, \lbr {\mf x}^+ \rbr; {\bm J}, {\bm H}, {\bm \lambda} \right).
\end{align}
A coherent system of orientations can be put on all these moduli spaces, for the same reason as in the case of connecting orbits. Then when ${\sf cz}(\lbr {\mf x}^- \rbr) = {\sf cz} ( \lbr {\mf x}^+ \rbr)$, the moduli space (\ref{equation75}) is a zero dimensional manifold, and its algebraic count gives an integer $\chi \left( \lbr {\mf x}^- \rbr, \lbr {\mf x}^+ \rbr \right)$. Then we define the continuation map 
\begin{align*}
\begin{array}{cccc}
\mf{cont}_{{\bm J}, {\bm H}, {\bm \lambda}}: & VCF_* \left( J^-, H^-, \lambda^- \right) & \to &  VCF_* \left( J^+, H^+, \lambda^+ \right)\\[0.2 cm]
           & \lbr {\mf x}^- \rbr & \mapsto & \displaystyle \sum_{ \lbr {\mf x}^+ \rbr  \in {\rm Crit} {\mc A}_{H^+}} \chi  \left( \lbr {\mf x}^- \rbr, \lbr {\mf x}^+ \rbr \right)\lbr {\mf x}^+ \rbr.
					\end{array}
\end{align*}
The two sides are the vortex Floer chain complex for the two triples $(J^-, H^-, \lambda^-)$ and $(J^+, H^+, \lambda^+)$ respectively. Then we have the similar results as in ordinary Hamiltonian Floer theory.
\begin{thm}
\begin{enumerate}
\item The map $\mf{cont}_{{\bm J}, {\bm H}, {\bm \lambda}}$ is a chain map. The induced map on the vortex Floer homology groups is independent of the choice of the homotopies ${\bm H}$, ${\bm \lambda}$, and the choice of ${\bm J} \in \wt{\mc J}^{reg}_{{\bm H}, {\bm \lambda}}$. In particular, $\mf{cont}_{{\bm J}, {\bm H}, {\bm \lambda}}$ is a chain homotopy equivalence.

\item For $i = 1, 2, 3$, suppose $H^i= (H^i_t)$ are admissible Hamiltonians, $\lambda^i >0$ and $(J^i_t) \in \wt{\mc J}^{reg}_{H^i, \lambda^i}$. We denote by 
\begin{align*}
\mf{cont}_{ij}: VHF \left( M, \mu; J^i, H^i, \lambda^i \right) \to VHF \left( M, \mu; J^j, H^j, \lambda^j \right),\ 1\leq i < j \leq 3
\end{align*}
the isomorphism induced by any chain-level continuation map. Then 
\begin{align*}
\mf{cont}_{23} \circ \mf{cont}_{12} = \mf{cont}_{13}.
\end{align*}
\end{enumerate}
\end{thm}

\begin{proof} The proof is essentially based on the construction of various gluing maps and the compactness results about ${\mc N}\left( \lbr {\mf x}^-\rbr, \lbr {\mf x}^+ \rbr ; {\bm J}, {\bm H}, {\bm \lambda} \right)$ when ${\sf cz}({\mf x}^-) - {\sf cz}( {\mf x}^+) = 1$. As in the gluing map constructed in proving the property $\delta_J^2 = 0$, we need to specify a gauge to construct the approximate solutions. We can still use solutions in $r$-temporal gauge, which is a notion independent of the equation. We omit the details. 
\end{proof}

\bibliography{symplectic_ref}
	
\bibliographystyle{amsplain}

\end{document}